\documentclass[12pt]{amsart}
\usepackage{xcolor}
\definecolor{cadmiumgreen}{rgb}{0.0, 0.42, 0.24}
\usepackage[T1]{fontenc}     
\usepackage[
colorlinks, citecolor=cadmiumgreen,
pdfauthor={Junta Kamiya}, 
pdftitle={tropical singular intersection homology},
pdfstartview ={FitV},
]{hyperref}
\usepackage{xurl}
\usepackage[left=2.5cm,top=2.7cm,right=2.5cm,bottom=2.7cm]{geometry}
\usepackage{tikz-cd}
\usepackage{mathptmx} 
\usepackage[scaled]{helvet}
\usepackage{courier}
\usepackage{amssymb,amsthm, amsmath,mathtools,calligra,mathrsfs}	
\usepackage{dsfont}						
\usepackage{tikz}							
\usetikzlibrary[matrix, arrows, calc]
\usepackage{enumitem}
\usepackage{colonequals}
\usepackage{textcomp}
\usepackage{graphics}

\newcommand{\Hyper}{{\mathds H}}
\newcommand{\ZZ}{{\mathbb Z}}
\newcommand{\Z}{{\mathbb Z}}
\newcommand{\NN}{{\mathbb N}}
\newcommand{\CC}{{\mathcal C}}
\newcommand{\RR}{{\mathbb R}}
\newcommand{\R}{{\mathbb R}}
\newcommand{\QQ}{{\mathbb Q}}
\newcommand{\TT}{{\mathbb T}}
\newcommand{\T}{{\mathbb T}}
\newcommand{\D}{{\mathcal D}}

\renewcommand{\S}{{\mathcal S}}
\renewcommand{\SS}{{\mathfrak S}}

\renewcommand{\Im}{\text{Im}}
\newcommand{\Q}{{\mathcal Q}}

\newcommand{\Linear}{\mathbb{L}}
\newcommand{\DD}{\mathrm{D}}
\newcommand{\XX}{\mathfrak{X}}

\newcommand{\imap}{\iota_{\tau,\sigma}}

\newcommand{\perv}{\bar{\mathsf{p}}}
\newcommand{\perq}{\bar{\mathsf{q}}}

\newcommand{\per}{\bar{\mathsf{r}}}
\newcommand{\perz}{\bar{0}}
\newcommand{\perd}{D\perv}

\DeclareMathOperator{\Homs}{\mathscr{H}\textnormal{\kern -3pt {\calligra\large om}}\,}
\DeclareMathOperator{\Aff}{Aff}
\DeclareMathOperator{\relint}{relint}
\DeclareMathOperator{\sed}{sed}
\DeclareMathOperator{\Hom}{Hom}
\DeclareMathOperator{\St}{St}

\DeclareMathOperator{\Ker}{Ker}
\DeclareMathOperator{\Coker}{Coker}

\DeclareMathOperator{\id}{id}
\DeclareMathOperator{\BM}{BM}
\DeclareMathOperator{\GM}{GM}

\DeclareMathOperator{\HH}{H}
\DeclareMathOperator{\IH}{I^{\perv}H}
\DeclareMathOperator{\IHq}{I^{\perq}H}

\DeclareMathOperator{\IHr}{I^{\per}H}

\DeclareMathOperator{\IHd}{I^{\perd}H}
\DeclareMathOperator{\IC}{I^{\perv}C}

\DeclareMathOperator{\HC}{C}
\DeclareMathOperator{\codim}{codim}
\DeclareMathOperator{\Idelta}{I^{\perv}\Delta}

\DeclareMathOperator{\Ideltad}{I^{\perd}\Delta}

\DeclareMathOperator{\Ext}{Ext}

\DeclareSymbolFont{cmletters}{OML}{cmm}{m}{it}
\DeclareMathSymbol{\jmat}{\mathalpha}{cmletters}{"7C}

\newtheorem{thm}{Theorem}[section]
\newtheorem{prop}[thm]{Proposition}
\newtheorem{proposition}[thm]{Proposition}
\newtheorem{lemma}[thm]{Lemma}

\newtheorem{corollary}[thm]{Corollary}

\theoremstyle{definition}
\newtheorem{defn}[thm]{Definition}
\newtheorem{definition}[thm]{Definition}
\newtheorem{rem}[thm]{Remark}

\newtheorem{example}[thm]{Example}

\numberwithin{equation}{section}

\begin{document}
\address{Department of Mathematics, Graduate School of Science, Kyoto University, Kyoto 606-8502, Japan}
\email{kamiya.junta.math@gmail.com}
\title{tropical singular intersection homology}
\author{Junta Kamiya}
\date{}

\begin{abstract}
    We introduce tropical singular intersection homologies (non-$\GM$ and $\GM$) with the tropical coefficients on rational polyhedral spaces using their filtrations. 
    We investigate their fundamental properties. 
    In the non-$\GM$ case, we give a Poincar\'e duality and two bilinear pairings analogous to the cup and cap products under some assumptions. 
    We compute the homologies in some cases. 
\end{abstract}

\maketitle

\tableofcontents

\section{Introduction}\label{sec:intro}
\renewcommand{\thethm}{\Alph{thm}}
\subsection{Tropical singular intersection homology}
In tropical geometry, the main objects of study are tropical manifolds, and more generally, rational polyhedral spaces. 
Roughly speaking, a rational polyhedral space is a locally finite collection of rational polyhedra, and we assume 
that it is equipped with an open face structure $\CC$ (see Definitions \ref{def:polyhedralspace} and \ref{def:facestructure}). 
Let $Q$ be a Dedekind domain such that $\ZZ \subset Q \subset \RR$. 
In this setting, we can define the tropical homology $\HH_{p,q}(X,Q)$ as in \cite{MR3961331}, \cite{MR3894860}. 
The tropical homology  differs from the usual singular homology in that $\HH_{p,q}(X,Q)$ uses coefficients ${\bf F}^Q_{p}(\sigma)$ 
for $\sigma\in \CC$, which is a $Q$-module called the $p$-th multi-tangent of $\CC$ at $\sigma\in\CC$ 
(see Definition \ref{def:multitangent} for details). 

On the other hand, an $n$-dimensional rational polyhedral space has the natural filtration $\XX_X^{\text{trop}}$ obtained as follows: 
the top stratum consists of points that have an $n$-dimensional affine neighborhood; 
the $(n-1)$-dimensional stratum consists of points that have an $(n-1)$-dimensional affine neighborhood in the remaining space; repeating this process yields the filtration:
\begin{equation}\label{eq:introfilter}
     \XX_X^{\text{trop}} : X = X^{n} \supset X^{n-1} \supset \cdots \supset X^0 \supset X^{-1} = \emptyset.
\end{equation}
The set of strata $\SS$ is the set consisting of the connected components of $X^i\smallsetminus X^{i-1}$ for all $i$.  
Moreover, we fix a perversity $\perv$, which is just a map $\perv\colon\SS\to\Z$ that takes zero on strata in $X^{n}\smallsetminus X^{n-1}$. 
By combining the coefficients ${\bf F}^Q_{p}(\sigma)$ above and the classical singular intersection theory as in \cite{Friedman_2020}, we define the tropical singular intersection homology groups 
\begin{equation}\label{eq:introhomdef}
    \IH_{p,q}(X,\XX_X^{\text{trop}},Q)\qquad\text{ and }\qquad{}_{\GM}\IH_{p,q}(X,\XX_X^{\text{trop}},Q),
\end{equation}
which we call the non-$\GM$ tropical intersection homology and the $\GM$ tropical intersection homology, respectively (see Definitions \ref{def:tropchainhomo} and \ref{def:tropchaincohomo}). 
The $\GM$ here stands for "Goresky-MacPherson". 
If $p=0$, the non-$\GM$ and $\GM$ tropical singular intersection homologies $\IH_{0,q}(X,\XX_X^{\text{trop}},Q)$ and ${}_{\GM}\IH_{0,q}(X,\XX_X^{\text{trop}},Q)$  
coincide respectively with the non-$\GM$ and $\GM$ intersection homologies $\IH_{q}(X,\XX_X^{\text{trop}},Q)$ and ${}_{\GM}\IH_{q}(X,\XX_X^{\text{trop}},Q)$ in \cite{Friedman_2020}. 
Thus, the non-$\GM$ and $\GM$ tropical singular intersection homologies in \eqref{eq:introhomdef} are natural generalizations of the non-$\GM$ and $\GM$ intersection homologies in a tropical setting.  

The tropical singular intersection homologies differ from the usual tropical homology in that the chain complexes are composed only of chains that intersect the strata of $\XX_X^{\text{trop}}$ in a controlled way, where the perversity $\perv$ serves 
as a parameter governing how the chains are allowed to meet the strata. 
The $\GM$ tropical homology is close to the classical definition of the intersection homology, which generalizes the usual tropical homology (see Remark \ref{rem:perv}). 
On the other hand, the non-$\GM$ tropical homology enjoys good homological-algebraic properties such as 
the Poincar\'e duality (see \textsection\ref{subsec:introPoin}). 
As usual, similar constructions can be made in relative (co)homology, cohomology, Borel--Moore homology, and compactly supported cohomology. 
Let \[\XX\colon X=X^{n}\supset X^{n-1}\supset \cdots \supset X^0 \supset X^{-1} = \emptyset\] 
be a filtration by closed sets on $X$, not necessarily equal to $\XX_X^{\text{trop}}$. 
We can still define the tropical singular intersection homology groups $\IH_{p,q}(X,\XX,Q)$ and 
$_{\GM}\IH_{p,q}(X,\XX,Q)$ (see Definitions \ref{def:tropchainhomo} and \ref{def:tropchaincohomo}). Note that these homology groups depend on the choice of a filtration. If it is clear from the context, we often omit $\XX$ and simply write $\IH_{p,q}(X,Q)$ and $_{\GM}\IH_{p,q}(X,Q)$.
In the following, the notation $(\GM)$ stands for both $\GM$ and non-$\GM$. 

\subsection{Basic properties}
Let $\XX$ be a filtration on $X$.  
The  homology $_{(\GM)}\IH_{p,q}(X,\XX,Q)$ satisfies properties that hold for the tropical homology and the singular intersection homology as follows:
\begin{itemize}
    \item Excision (see Proposition \ref{prop:oriexcision});
    \item Mayer-Vietoris sequences (see Proposition \ref{prop:Mayer});
    \item Cone formulas (see Proposition \ref{prop:coneGM} for $\GM$ and Proposition \ref{prop:conenon} for non-$\GM$). 
\end{itemize} 

The tropical intersection homologies also satisfy the following properties. 
\begin{proposition}[$=$Proposition \ref{prop:indep}]\label{prop:introindep}
    The homology $_{(\GM)}\IH_{p,q}(X,\XX, Q)$ is independent of the choice of an open face structure on $X$. 
\end{proposition}
Let $\mathcal{F}$ be a map assigning each rational polyhedral space $X$ a filtration $\XX$ on $X$ in  a functorial way (see Definition \ref{def:tropfilter}). 
We call such an $\mathcal{F}$ a tropical filtration.  
An example of a tropical filtration is $\mathcal{F}^{\text{trop}}\colon X \mapsto \XX^{\text{trop}}_X$ in \eqref{eq:introfilter}. 
\begin{prop}[$=$Proposition \ref{prop:tropstr}]\label{prop:introtropstr}
    Let $\mathcal{F}$ be a tropical filtration and let $X$ and $Y$ be rational polyhedral spaces. 
    Let $\mathcal{F}(X)$ and $\mathcal{F}(Y)$ be filtrations on $X$ and $Y$ determined by $\mathcal{F}$, respectively.   
    Let perversities $\perv$ and $\perq$ be perversities on $(X,\mathcal{F}(X))$ and $(Y,\mathcal{F}(Y))$, respectively. 
    We assume that $X$ has an $\mathcal{F}(X)$-stratified open face structure and that $Y$ has an $\mathcal{F}(Y)$-stratified open face structure. 
    If there exists an isomorphism $f:X\to Y$ between rational polyhedral spaces such that $\perv(S)=\perq(f(S))$ for all strata $S$ of $(X,\mathcal{F}(X))$, then the following hold. 
    \begin{align*}
        _{(\GM)}\IH_{p,q}(X,\mathcal{F}(X),Q)\cong {}_{(\GM)}\IHq_{p,q}(Y,\mathcal{F}(Y),Q),\\ {}_{(\GM)}\IH^{p,q}(X,\mathcal{F}(X),Q)\cong {}_{(\GM)}\IHq^{p,q}(Y,\mathcal{F}(Y),Q).
    \end{align*}
\end{prop}

\subsection{Poincar\'e duality in non-GM (co)homology}\label{subsec:introPoin}
In the non-$\GM$ case, we prove a Poincar\'e duality. 
Here we use a sheaf description of a chain complex in order to make use of the sheaf theory of intersection homology.  
Let $\Gamma_{c}$ denote the functor taking compactly supported global sections.  
For a singular intersection chain sheaf ${}_{(\GM)}\Idelta^{p,\bullet}$ (see \textsection\ref{subsec:sheaf}), we obtain the following statement. 
\begin{prop}[$=$Proposition \ref{prop:sheaf}]\label{prop:introsheaf}
    $\Gamma_{c}({}_{(\GM)}\Idelta^{p,-\bullet}) \cong {}_{(\GM)}\IC_{p, \bullet}(X, Q)$.
\end{prop}

In the non-$\GM$ case, let $\XX$ be a filtration on $X$ and let $n=\dim(X)$. 
We assume that $X_{n}\colonequals X^{n}\smallsetminus X^{n-1}$ is an $n$-dimensional tropical manifold and dense in $X$. 
In this setting, we call $(X,X_{n})$ an $n$-dimensional tropical manifold with singularities. 
In order to use the generalized Deligne sheaf theory developed in \cite{Friedman_2010},  
we introduce a condition $(C)$ on a filtration $\XX$ (see Definition \ref{def:semilocgoodratpolyspace}). 
Under the condition $(C)$, we construct a tropical Deligne sheaf (see \textsection\ref{subsec:deligne}). 
As in the classical case,  a Poincar\'e duality for intersection (co)homology can be proved using the tropical Deligne sheaf. 
The next theorem is a main result of this paper. 

\begin{thm}[$=$Theorem \ref{thm:genpoin}, Poincar\'e duality]\label{thm:intropoin}
    Assume that $Q$ is a field. Let $(X,\XX)$ be an  $n$-dimensional  filtered rational polyhedral space with a face structure. 
    Suppose that $(X,X_n)$ is an $n$-dimensional tropical manifold with singularities and $\XX$ satisfies $(C)$.  
    Let $\perv$ be a perversity  on $(X,\XX)$. 
    Then there exists an isomorphism 
    \begin{equation*}
        \IH^{n-p,n-q}(X,\XX,Q)\cong \IHd_{p,q}^{\BM}(X,\XX,Q)
    \end{equation*}
    between the tropical intersection cohomology and the tropical intersection Borel--Moore homology. 
\end{thm}
Using Theorem \ref{thm:intropoin}, we obtain bilinear pairings similar to the cup and cap products (see \eqref{eq:gencup} and \eqref{eq:gencap}).  

We note that, as in the classical $\GM$ intersection homology, the Poincar\'e duality does not hold for the $\GM$ tropical intersection (co)homology. 
For example, the Bergman fan induced from the matroid $U_3^1$ gives such an example (see \textsection\ref{subsec:counter}). 

We give some examples for the tropical intersection homologies. 
If $X$ is a $1$-dimensional polyhedral space, we can easily compute the tropical intersection homology: 
\begin{align*}
    \IH_{p,q}(X,\XX_X^{\text{trop}},Q)\cong
    \begin{cases}
        Q^{\oplus a} &\text{if } q=1,\text{ }p=0,1,\\
        Q^{\oplus b} &\text{if } q=0,\text{ }p=0,1,\\
        0 &\text{otherwise}, 
    \end{cases}
\end{align*}
where $a$ and $b$ are the number of edges of $X$ satisfying some conditions related to the perversity (see \textsection\ref{subsubsec:comp}). 

In the $\GM$ case, 
we have 
\[_{\GM}\IH_{p,q}(X,\XX_X^{\text{trop}},Q)\cong\HH_{p,q}(X\smallsetminus  V_{\perv},Q),\]
where $V_{\perv}$ is a set defined by the perversity and $\HH_{p,q}(X\smallsetminus  V_{\perv},Q)$ is the usual tropical homology (see \textsection\ref{subsubsec:comp}). 

We give another example. 
Let $X$  be a compact $1$-dimensional rational polyhedral space. 
Let $(X,U)$ be a tropical manifold with singularities. 
We define the filtration $\XX$ as follows:
\[\XX\colon X\supset X\smallsetminus U \supset \emptyset. \]
With a perversity $\perv$ such that there exists $m \in \ZZ$ satisfying $\perv(\{v\})=m$ for any $v\in X\smallsetminus U$, the homology is computed as 
\begin{align*}
    \IH_{p,q}(X,\XX,Q)\cong
    \begin{cases}
        \HH_{p,q}(U,Q) &\text{if } m<0,\\
        \HH_{p,q}^{\BM}(U,Q) &\text{if } m\geq 0 
    \end{cases}
\end{align*} 
(see \textsection\ref{subsubsec:tropmfd}). 
Further, if $Q$ is a field, then we see $\IH^{1-p,1-q}(X,\XX,Q)\cong \IHd_{p,q}^{\BM}(X,\XX,Q)$ by direct computation, which verifies Theorem \ref{thm:intropoin} in this case.

\subsection{Background and related works}
Intersection homology was first introduced by Goresky and MacPherson in \cite{GORESKY1980135} in order to generalize some salient features of manifold theory to spaces with singularities. 
They used the language of piecewise linear (PL) chain complexes and worked over PL pseudomanifolds without codimension-one strata. The intersection homology was soon reformulated in terms of sheaves 
in derived categories, as suggested by Deligne \cite{MR696691}, where the functorial machinery of derived categories could be applied. 
King \cite{KING1985149} introduced the singular chain intersection homology, and 
Friedman and McClure \cite{MR3046315} defined the singular chain intersection 
homology with possibly codimension-one strata. 
Friedman introduced two types of singular intersection homologies: the GM intersection homology and the non-GM intersection homology (see \cite{Friedman_2020}). 
The GM and non-GM intersection homology groups coincide under 
the classical assumption (that is, without codimension-one strata and with an adequate perversity), but 
they differ in general. Furthermore, Friedman and McClure \cite{MR3046315} proved the 
K\"unneth formula and the Poincar\'e duality for the non-GM intersection 
homology, while these formulas do not generally hold for the GM intersection homology. 
The textbook by Maxim \cite{Maxim_2019} covers the theory of intersection (co)homology from the view point of algebraic geometry. 
The textbook by Friedman \cite{Friedman_2020} treats the general theory of intersection (co)homology. 

Tropical geometry is a relatively new field at the intersection of
algebraic geometry and combinatorics. 
Itenberg, Mikhalkin, Katzarkov, and Zharkov \cite{MR3961331} defined the tropical
homology for certain polyhedral complexes in tropical projective space.
Jell, Shaw, and Smacka \cite{MR3903579} and Jell, Rau, and Shaw \cite{MR3894860} generalized
this definition to rational polyhedral spaces with face 
structures and proved the Poincar\'e duality. Gross and Shokrieh \cite{Gross_2023}
provided an alternative definition of the tropical homology using sheaves
and proved the Poincar\'e-Verdier duality among others. 

Our work is related to that of Mikami \cite{mikami2025} in that both papers focus on intersection (co)homology in tropical geometry. 
The (co)homology in \cite{mikami2025} is defined on tropical toric varieties, and  a tropical analogue of the Deligne sheaf of the middle perversity is considered. 
Furthermore, this (co)homology satisfies a Poincar\'e duality.  
On the other hand, our approach is simpler, and the definitions do not coincide in general cases. We define the (co)homology, more generally, on rational polyhedral spaces with filtrations and we allow various perversities. 
Properties of our (co)homology depend on perversities. 
If we are able to construct some mezzoperversity sheaf similar to \cite{MR3284392}, then we might be able to find a relation between our homology and homology in \cite{mikami2025}.

\subsection{Structure of the paper}
In \textsection\ref{sec:def}, we recall the theory of rational polyhedral spaces following \cite{MR3330789,MR3894860,Gross_2023} and the theory of singular intersection homology following \cite{Friedman_2020}. 
In \textsection\ref{sec:property}, we first show that the singular tropical intersection homology satisfies basic properties analogous to those of the ordinary singular intersection homology (cf. \cite{Friedman_2020}). 
We then introduce a sheaf-theoretic definition of the tropical intersection homology. 
In \textsection\ref{sec:non-GMmfd}, we show a Poincar\'e duality for the non-$\GM$ case by using the generalized Deligne sheaf theory from \cite{Friedman_2010}. 
In \textsection\ref{sec:ex}, we present several examples and computations. 
In \textsection\ref{subsec:counter}, we provide a counterexample to the Poincar\'e duality in the $\GM$ case similar to that in the ordinary singular intersection homology. 
In \textsection\ref{subsec:onedim}, we compute tropical intersection homology groups for the $1$-dimensional case. 
In \textsection\ref{subsec:filtsing}, we give examples of tropical intersection homology groups using various filtrations. 

\medskip
\noindent
{\bf{Acknowledgement.}}
The author would like to thank Shu Kawaguchi and Ryota Mikami for helpful advice. 

\section{Definitions}\label{sec:def}
\renewcommand{\thethm}{\thesection.\arabic{thm}}
\subsection{Preliminaries}\label{subsec:prelim}
\subsubsection{Tropical geometry}

We recall the notion of a rational polyhedral space. For details, we refer to \cite{MR3330789,MR3894860,Gross_2023}. We set $\T=[-\infty,\infty)$ and equip $\T$ with Euclidean topology. 
We equip the product $\T^r$ with the product topology. For $I \subset [r] \colonequals \{1,2,\cdots,r\}$, we set $\R^r_I\colonequals$ 
$\{x \in \T^r \mid x_i=-\infty \text{ if and only if }  i\in I\}$ and $T^r_I \colonequals \overline{\R^r_I}$. 
For $x\in \T^r$, we define the \emph{sedentarity} of $x$ by $\sed(x)\colonequals\{i\in[r]\mid x_i=-\infty\}$. 
In other words, we take a unique $\R^r_I$ such that $x\in \R^r_I$, and we define $\sed(x)\colonequals I$. 
If $U\subset \R^r_I$, then we similarly define $\sed(U)$ by $\sed(U)\colonequals I$. 

We define a \emph{rational polyhedron} in $\R^r$ to be $\{x\in \R^r\mid (w_1,x)\leq c_1,\cdots,(w_k,x)\leq c_k\}$ for some $w_i \in \Z^r$,  
$c_i \in \R^r$, and $k\geq 1$. A \emph{face} of a rational polyhedron $\sigma$ is obtained by turning some  inequalities defining $\sigma$ to equalities. 
If $\tau$ is a face of $\sigma$, then we write $\tau\prec\sigma$. 

We define a \emph{rational polyhedron} in $\T^r$ to be the closure of a rational polyhedron in $\R^r_I$ for some 
$I \subset [r]$. A face of a rational polyhedron $\sigma$ in $\T^r$ is the closure of the face of $\sigma\cap\R^r_J$ in $\T^r$ for $J\subset [r]$. 
Moreover, the relative interior of $\sigma$, denoted by $\relint(\sigma)$, is the set consisting of points in $\sigma$ which is not in a proper face. For rational polyhedra $\sigma,\tau$ in $\T^r$, if $\tau$ is a face of $\sigma$, then we also write $\tau\prec\sigma$.
A \emph{rational polyhedral complex} $\CC$ in $\T^r$ is a locally finite collection of polyhedra in $\T^r$ satisfying the following properties:
\begin{enumerate}
    \item If $\sigma \in \CC$ and $\tau \prec \sigma$, then $ \tau \in \CC$.
    \item If $\sigma, \sigma' \in \CC$ and $\sigma \cap \sigma'\neq \emptyset$, then $\sigma\cap\sigma'$ is a face of both $\sigma$ and $\sigma'$.
\end{enumerate}
The \emph{support} of a rational polyhedral complex $\CC$ in $\T^r$  is defined by $\bigcup_{\sigma\in\CC}\sigma$ and is denoted by $\vert \CC \vert$. 
If  $X = |\CC|$ for some rational polyhedral complex $\CC$, then $X$ is called a \emph{rational polyhedral subspace} of $\T^r$. 
For $\sigma\in\CC$, we define $d(\sigma)$ to be the dimension of the affine span of $\relint(\sigma)$. 
For a rational polyhedral subspace $X=|\CC|$, if there exists $n\in\NN$ such that $d(\sigma)\leq n$ for any $\sigma\in\CC$, we say that $X$ is \emph{at most $n$-dimensional}.  

Let $M\subset\T^m$ and $N\subset\T^n$ be any subsets. We say that a map $f \colon M \to N$ is an \emph{extended affine $\Z$-linear map} if it is continuous 
and $f(x) = Ax +b$ for some $A \in \text{Mat}(n \times m, \Z)$,
$b \in \R^n$ for all $x \in M$. Here we use the convention that $0\cdot (-\infty)=0$. 
If $N=\R$, we call an extended affine $\Z$-linear map $f \colon M \to \R$ an \emph{integral affine $\ZZ$-linear function} on a subset $X\subset \T^n$. 
For every subset $X\subset \T^n$, the integral affine $\ZZ$-linear functions
on open subsets of $X$ define a sheaf of abelian groups on $X$, denoted by $\Aff_X$. 

We use the notion of a \emph{matroidal fan} in \cite[Definition 2.22]{MR3032930}. 
A matroidal fan in $\R^r$ is a fan in $\R^r$ which is canonically constructed from a matroid. 
Since we do not use the construction in this paper, we omit the details and  refer the readers to \cite{MR3032930}. 
Here we only note that $\RR^n$ is the support of an $n$-dimensional matroidal fan and that star-shaped sets are the supports of $1$-dimensional matroidal fans. 
Here, a star-shaped set is the set of the form $\bigcup_{0\leq i \leq n}\R_{\geq 0} \mathbf{e}_i\subset \R^{n+1}/\R\mathbf{1}$ for $n>0$, 
where $\mathbf{1}$ is the vector whose coordinates are all $1$ and $\mathbf{e}_i$ is the $i$-th standard basis vector.

\begin{defn}[Rational polyhedral space {\cite[Definitions 2.1 and 5.1]{MR3894860}}]
\label{def:polyhedralspace}
A \emph{rational polyhedral space} $X$ is a paracompact, second countable Hausdorff topological space with an atlas of charts 
$(\varphi_{\alpha} \colon U_{\alpha} \rightarrow \Omega_{\alpha} \subset \T^{r_\alpha})_{{\alpha} \in A}$ such that:
\begin{itemize}
\item[\rm (1)]
The $U_{\alpha}$ are open subsets of $X$, the $\Omega_{\alpha}$ are open subsets of rational polyhedral subspaces  
$X_{\alpha}\subset\T^{r_{\alpha}}$, and the maps 
$\varphi_{\alpha} \colon U_{\alpha} \rightarrow \Omega_{\alpha}$ are homeomorphisms for all $\alpha$; 
\item[\rm (2)]
for all $\alpha, \beta \in A$, the transition maps
\begin{align*}
\varphi_{\alpha}\circ \varphi^{-1}_ \beta\colon \varphi_ \beta(U_{\alpha} \cap U_ \beta) \rightarrow {\varphi_\alpha(U_{\alpha} \cap U_ \beta)}
\end{align*}
are  extended affine $\Z$-linear maps;
\item [\rm (3)]
there exists $N\in\NN$ such that $X_{\alpha}$ is at most $N$-dimensional for any $\alpha\in A$.
\end{itemize} 
Furthermore, if $X_{\alpha}=\T^{s_{\alpha}}\times B_{\alpha}$ for the support $B_{\alpha}$ of some matroidal fan in $\R^{r_{\alpha}-s_{\alpha}}$ for any $\alpha\in A$ in the condition (1), we call $X$ a \emph{tropical manifold}. 
\end{defn} 
\begin{rem}
    A chart $(\varphi\colon U \rightarrow \Omega\subset \T^r)$ is \emph{compatible} with an atlas $(\varphi_{\alpha} \colon U_{\alpha} \rightarrow \Omega_{\alpha} \subset \T^{r_\alpha})_{{\alpha} \in A}$ if their union is again an atlas. 
    We may enlarge the given atlas whenever necessary so that it contains compatible charts. 
\end{rem}
Let $U$ be an open subset of a rational polyhedral space $X$. Then we may regard $U$ as a rational polyhedral space by the restriction of the atlas of $X$. 
For a rational polyhedral space $X$, we define the sheaf $\Aff_X$ by gluing $\Aff_{\Omega_\alpha}$. 
For a rational polyhedral space $X$, let $S^X$ be the set consisting of a point $x\in X$ such that $x$ have $n_x$-dimensional affine neighborhood. 
Then we define $d(X)$ to be the $\max_{x\in S^X} n_x$. 
If there exists $n\in\NN$ such that $n_x=n$ for any $x\in S^X$, we say that $X$ is \emph{pure $n$-dimensional}. 

A morphism between two rational polyhedral spaces $X$ and $Y$ is a continuous map 
$f:X\to Y$ such that if we take atlases $\{\varphi_\alpha^X\}_{\alpha \in A}$ and $\{\varphi_\beta^Y\}_{\beta \in B}$ of $X$ and $Y$, respectively,
then the restriction to each chart $\varphi_\beta^Y \circ f\circ \varphi^X_\alpha {}^{-1} :\varphi_\alpha^X(U^X_\alpha\cap f^{-1}(U^Y_\beta)) \to \varphi^Y_\beta (U^Y_\beta)$ is an extended affine $\Z$-linear map.

We define an \emph{open rational polyhedron} in $\T^r$ to be a connected open subset of a polyhedron in $\T^r$. 
Let $\sigma$ be an open rational polyhedron in $\T^r$. 
We take a rational polyhedron $P$ and an open subset $U$ in $\mathbb{T}^r$ such that $\sigma$ is a connected component of $P\cap U$. 
We say that a subset $\tau$ of $\sigma$ is a \emph{face} of $\sigma$ if there exists a face $Q$ of $P$ such that $\tau$ is a connected component of $Q\cap U$. 
In this case, we write $\tau\prec \sigma$. 
One can show whether $\tau\prec \sigma$ or not does not depend on the choice of $P$ and $U$ for which $\sigma$ is a connected component of $P\cap U$ (see Remark \ref{rem:openprec}). 
An \emph{open rational polyhedral complex} $\CC$ in $\T^r$ is a locally finite collection of open rational polyhedra in $\T^r$ satisfying the same conditions as in the definition of a rational polyhedral complex in $\T^r$.  
For an open rational polyhedron $\sigma$, $\relint(\sigma)$ is the set of points  that are not contained in any $\tau\prec\sigma$. 
We define $d(\sigma)$ to be the dimension of the affine span of $\relint(\sigma)$, and 
we define $\sed(\sigma)$ by $\sed(\sigma)\colonequals \sed(\relint(\sigma))$. 
We denote by $\CC_I$ the union of polyhedra $\sigma\in \CC$ such that $\relint(\sigma)\subset \R^r_I$. 
We set $\sigma_I\colonequals\sigma\cap\R^r_I$ for an open rational polyhedron $\sigma$.

\begin{defn}[(Open) face structure {\cite[Definition 2.2]{MR3894860}}]\label{def:facestructure}
Let $X$ be a rational polyhedral space.
A \emph{face structure} (resp. an \emph{open face structure})
on $X$ is a locally finite family $\CC$ of closed subsets of $X$ such that the following conditions hold:
\begin{itemize}
\item[\rm (1)]
$X = \bigcup_{\sigma \in \CC} \sigma$;
\item[\rm (2)]
for each $\sigma$, there exists a chart $\varphi_\sigma \colon U_\sigma \rightarrow \Omega_\sigma \subset \T^{r_\sigma}$
such that if we set $\overline{\St}(\sigma)\colonequals\{\tau\in \CC \mid \text{there exists } \sigma'\in\CC \text{ such that }\tau\cap\sigma'\neq\emptyset \text{ and }\sigma\cap\sigma'\neq\emptyset \} \subset U_\sigma$, 
then $\{\varphi_\sigma(\tau) \mid \tau \in \overline{\St}(\sigma)\}$ is a rational polyhedral complex (resp. an open rational polyhedral complex) 
in $\T^{r_\sigma}$.
\end{itemize}
\end{defn}

Let $X$ be a rational polyhedral space with an open face structure $\CC$. 
Any open subset $U\subset X$ is also a rational polyhedral space and has an open face structure $\CC|_U$ induced by $\CC$.  
Let $\CC$ and $\CC'$ be two open face structures on $X$. 
We denote by $\CC\cap\CC'$ the family of closed sets consisting of $\overline{(\sigma_{\sed(\sigma)}\cap \sigma'_{\sed(\sigma')})}$ for $\sigma\in\CC$ and $\sigma'\in\CC'$ with $\sigma_{\sed(\sigma)}\cap \sigma'_{\sed(\sigma')}\neq\emptyset$, 
where the closure is taken in $X$ above. 
Note that this family $\CC\cap \CC'$ may not be an open face structure. 
\begin{rem}\label{rem:openprec}
    Let $\sigma$ in $\T^r$ be an open rational polyhedron. For $i=1,2$, we take a rational polyhedron $P_i$ and an open subset $U_i$ in $\T^r$ such that $\sigma$ is a connected component of $P_i\cap U_i$. 
    Let $Q_1$ be a face of $P_1$ and let $\tau\subset \sigma$ be a connected component of $Q_1\cap U_1$. Then there exists a face $Q_2$ of $P_2$ such that $\tau$ is a connected component of $Q_2\cap U_2$. 
    For simplicity, we show the existence of such a face $Q_2$ in the case $\sed(Q_1)=\sed(P_1)$ and $\sigma=P_1\cap U_1=P_2\cap U_2$. 
    In this case, it suffices to show that there exists a face $Q_2$ such that $Q_1\cap U_1\cap \RR^r_{\sed(P_1)}=Q_2\cap U_2\cap \RR^r_{\sed(P_1)}$. 
    Recall that $(Q_1)_{\sed(P_1)}=Q_1\cap\RR^r_{\sed(P_1)}$ is obtained by turning some inequalities defining $(P_1)_{\sed(P_1)}$ to equalities. 
    Let $\{f\}$ denote the inequalities such that $(Q_1)_{\sed(P_1)}$ is obtained by turning these to equalities as above. 
    Since  $\sigma=P_1\cap U_1=P_2\cap U_2$, $\{f\}$ also can be seen as some inequalities defining $(P_2)_{\sed(P_1)}$. 
    By turning $\{f\}$ to equalities, we obtain a face $Q'_2$ of $(P_2)_{\sed(P_1)}$. 
    Let $Q_2$ be the closure of $Q'_2$ in $\T^r$. 
    This rational polyhedron $Q_2$ is a face of $P_2$ which satisfies $Q_1\cap U_1\cap \RR^r_{\sed(P_1)}=Q_2\cap U_2\cap \RR^r_{\sed(P_1)}$. 
\end{rem}

\subsubsection{Filtered space}
We introduce the basic terminology of filtered spaces. See \cite{Borel_1984} and \cite{Friedman_2020} for details. 
\begin{defn}[Filtered space {\cite[Definition 2.2.1]{Friedman_2020}}]\label{def:filteredspace}
    A \emph{filtered space} $(X,\XX)$ is a Hausdorff topological space $X$ with a filtration $\XX$ by closed sets: 
    \[
    \XX \colon X = X^{n} \supset X^{n-1} \supset \cdots \supset X^0 \supset X^{-1} = \emptyset. 
    \]
    We call $n$ in the above filtration the \emph{formal dimension} of $(X,\XX)$, and denote it by $\dim((X,\XX))$ or $\dim(X)$.    
\end{defn}
If $\XX$ is clear from the context, we omit $\XX$ and denote $(X,\XX)$ by $X$. 
Let $\SS$ denote the set of connected components of $X_i \colonequals X^i \smallsetminus X^{i-1}$ for all $i$. We call an element in $\SS$ a \emph{stratum} of $X$, and
call a stratum in $X_{\dim(X)}$ a \emph{regular stratum}. We define the \emph{singular locus} $\S_X$ of $X$ by  $\S_X \colonequals X^{\dim(X)-1}$. 
We call a stratum in $\S_X$ a \emph{singular stratum}. 
The \emph{formal dimension} of $S\in\SS$, denoted by $\dim(S)$, is the integer $i$ such that $S\subset X_i$.  
The \emph{formal codimension} of $S\in\SS$ is defined by $\codim(S)\colonequals \dim(X)-\dim(S)$. 
For a filtered space $(X,\XX)$, any open subset $U$ of $X$ has the natural filtration $\XX|_U$ obtained by the restriction of the filtration $\XX$ to $U$, and we say that $(U,\XX|_U)$ is a \emph{filtered open subspace} of $(X,\XX)$. 

\begin{defn}[Stratified map, stratified homeomorphism {\cite[Definition 2.3.2]{Friedman_2020}}]\label{def:stratifiedmap}
    Given filtered spaces $X$ and $Y$, and a continuous map $f \colon X \to Y$, we say that $f$ is a \emph{stratified map} if for each stratum $S \subset X$, 
    there is a stratum $T\subset Y$ satisfying $f(S)\subset T$. Note that if such $T$ exists, then $T$ is uniquely determined by $S$. 
    Furthermore, if $f$ is bijective, the inverse map $f^{-1}\colon Y\to X$ is a stratified map, and $f$ preserves the formal dimensions, 
    then we say that $f$ is a \emph{stratified homeomorphism}. 
    If $X$ and $Y$ are stratified-homeomorphic, then we write $X\cong Y$. 
\end{defn}
\begin{defn}[Cones]
    Let $X$ be a topological space. 
    An \emph{open cone} $cX$ for $X$ is defined to be $[0,1)\times X/(\{0\}\times X)$. 
    Let $v$ denote the point corresponding to $\{0\}\times X$ in $cX$. 
    We call $v$ the \emph{vertex} of $cX$. 
    Further, let $(X,\XX)$ be a filtered space. 
    Following the notation in Definition \ref{def:filteredspace}, we define the filtration of $cX$ such that $(cX)^i\colonequals c(X^{i-1})\subset cX$ for $i>0$ and $(cX)^0=\{v\}$. 
\end{defn}

Let $X$ be a topological space. 
A \emph{closed cone} $\bar{c}X$ for $X$ is defined to be $[0,1]\times X/(\{0\}\times X)$. 

\begin{defn}[CS set {\cite[Definition 2.3.1]{Friedman_2020}}]\label{def:CS}
    A \emph{CS set} is a filtered space $(X,\XX)$ such that:
    \begin{enumerate}[label=(CS\arabic*)]
        \item\label{itm:cs1} Each stratum in $X_i$ is an $i$-dimensional topological manifold.
        \item\label{itm:cs2} For each $i$ and each $x\in X_i$, there exist an open neighborhood $U$ of $x$ in $X_i$, 
        a neighborhood $N$ of $x$ in $X$, a compact filtered space $L$ (which may be empty), and a stratified homeomorphism 
        $h:U\times cL \to N$ such that $h(U\times c(L^k))=X^{i+k}\cap N$.
    \end{enumerate}
\end{defn}
Here a CS set stands for a \emph{locally conelike topological stratified set}. 
For a point $x$ in a CS set $X$, we call a neighborhood $N$ as described above a \emph{distinguished neighborhood}.

To define an intersection homology, we need the notion of a perversity. 

\begin{defn}\label{def:perversity}
    A \emph{perversity} on a filtered space $X$ is a map $\perv \colon \SS \rightarrow \Z$ which is zero on regular strata.
\end{defn}
For example, for any function $f \colon \Z_{\geq 0}\to \Z$ such that $f(0)=0$, we can define a perversity $\perv \colon \SS \to \ZZ$ by $\perv(S)=f(\codim(S))$ for any $S\in\SS$.  
We denote by $\perz$ the perversity such that $\perv(S)=0$ for any $S\in \SS$.

\subsubsection{Filtered space in tropical geometry}
We introduce filtrations that arise naturally in tropical geometry. 
\begin{example}[Filtration induced by a rational polyhedral space]
    Let $X$ be a rational polyhedral space.
    Then $X$ has the natural filtration $\XX_X^{\text{trop}}$ as follows:
    \begin{equation}\label{eq:canstrat}
        \XX_X^{\text{trop}} \colon X = X^{d(X)} \supset X^{d(X)-1} \supset \cdots \supset X^0 \supset X^{-1} = \emptyset,
    \end{equation}
    where $X^i$ is obtained from $X^{i+1}$ by removing points which have an $(i+1)$-dimensional affine neighborhood. 
\end{example}
A rational polyhedral space is pure $n$-dimensional if and only if $X_{d(X)}\colonequals X^{d(X)}\smallsetminus X^{d(X)-1}$ for the filtration $\XX_X^{\text{trop}}$ is dense in $X$. 

We define a \emph{filtered rational polyhedral space} to be the data $(X,\XX)$ of a rational polyhedral space $X$ and a filtration $\XX$. 
Let $(X,\XX)$ be a filtered rational polyhedral space and let $\CC$ be any open face structure on $X$. Then we say that  
$\CC$ is \emph{$\XX$-stratified} if for any $\sigma\in \CC$, there exists $S_{\sigma}\in\SS$ satisfying:
\begin{enumerate}
    \item $\relint(\sigma)\subset S_{\sigma}$ and $d(\sigma)\leq\dim(S_{\sigma})$;
    \item if $\tau,\sigma\in \CC$ satisfies $\tau\prec\sigma$, then $\dim(S_{\tau})\leq\dim(S_{\sigma})$. 
\end{enumerate} 
Note that for each $\sigma\in\CC$, $S_{\sigma}\in\SS$ is uniquely determined. 
Since for $\sigma \in\CC$, $\relint(\sigma)$ is an open subset of $d(\sigma)$-dimensional affine space, $\relint(\sigma)$ is contained in some $d(\sigma)$-dimensional stratum of $\XX^{\text{trop}}_X$. 
Therefore, any open face structure $\CC$ is $\XX_X^{\text{trop}}$-stratified. 
\begin{defn}[Filtered rational polyhedral space with an (open) face structure]\label{ex:CS}
    We define a \emph{filtered rational polyhedral space with a face structure (resp. an open face structure)} to be the data $(X,\XX,\CC)$ of a filtered rational polyhedral space $(X,\XX)$ and an $\XX$-stratified face structure $\CC$ (resp. open face structure). 
\end{defn}
For $(X,\XX,\CC)$ above, 
We often say that $(X,\XX)$ is a filtered rational polyhedral space with a face structure (resp. an open face structure) $\CC$. 
We define the dimension of a filtered rational polyhedral space with an (open) face structure $(X,\XX,\CC)$ to be the dimension $\dim((X,\XX))$ of the filtered space $(X,\XX)$. 
\begin{example}[Filtration induced by a open face structure]
    Let $X$ be a rational polyhedral space and let $\CC$ be an open face structure on $X$. 
    We define the filtration $\XX_X^{\CC}$ to be the filtration such that $X^i$ is the union of $\sigma\in\CC$ with $d(\sigma)\leq i$. 
    Note that $\CC$ is $\XX^{\CC}_X$-stratified. 
\end{example}
For another open face structure $\D$ on a rational polyhedral space $X$, we say that $\D$ is $\CC$-stratified 
if $\D$ is $\XX_X^{\CC}$-stratified. 
\begin{example}[Filtration induced by two open face structures]
    Let $X$ be a rational polyhedral space and let $\CC$ and $\CC'$ be open face structures on $X$. 
    We define the filtration $\XX_X^{\CC\cap \CC'}$ to be the filtration such that $X^i$ is the union of $\sigma\in \CC\cap \CC'$ with $d(\sigma)\leq i$. 
    Note that the identity map $X\to X$ is a stratified map from $(X,\XX_X^{\CC\cap \CC'})$ to $(X,\XX_X^{\CC})$ or $(X,\XX_X^{\CC'})$. 
\end{example}

Let $K$ be a locally finite simplicial complex. 
We denote by $|K|$ the topological space obtained by gluing simplices of $K$ along their face. 
The dimension of a simplex $\sigma$ in $K$ is denoted by $d(\sigma)$. 
A \emph{triangulation} $\mathcal{T}$ on a topological space $X$ is a pair $(K,h)$ of a locally finite simplicial complex $K$ and a homeomorphism 
$h\colon |K|\to X$.  
\begin{example}[Filtration induced by a triangulation]\label{ex:triangulation}
    Let $\mathcal{T}=(K,h)$ be a triangulation on a topological space $X$. 
    Suppose that $\max_{\sigma\in K}d(\sigma)<\infty$. 
    Then we define the $\max_{\sigma\in K}d(\sigma)$-dimensional filtration $\XX_X^{\mathcal{T}}$ to be the filtration such that $X^i$ is the union of $h(\sigma)$ with $\sigma \in K$ and $d(\sigma)\leq i$.   
    By \cite[Lemma 2.5.17]{Friedman_2020}, the filtered space $(X,\XX_X^{\mathcal{T}})$ is a CS set. 
\end{example}
We say that a triangulation $\mathcal{T}$ of the underlying space $X$ of a filtered space $(X,\XX)$ is \emph{$\XX$-stratified} if the identity map $X\to X$ is a stratified map from $(X,\XX_X^{\mathcal{T}})$ to $(X,\XX)$. 
Let $X$ be a rational polyhedral space with an open face structure $\CC$. 
If $\mathcal{T}=(K,h)$ is $\XX^{\CC}_X$-stratified, then $\max_{\sigma\in K}d(\sigma)<\infty$ is satisfied. 
Therefore, $(X,\XX_X^{\mathcal{T}})$ is a CS set.

From these filtrations, we obtain the following. 
\begin{proposition}\label{prop:tropfaceCS}
    Let $X$ be a rational polyhedral space with an open face structure $\CC$. 
    Then $X$  is a CS set with filtration $\XX_X^{\CC}$.
\end{proposition}
\begin{proof}
    Since we can easily verify the condition \ref{itm:cs1} in Definition \ref{def:CS},  
    it suffices to show the condition \ref{itm:cs2} in Definition \ref{def:CS}. 
    Let $x$ be any point in $X$. 
    We choose a chart $\varphi\colon U\to \Omega \subset \T^r$ of $U$ containing $x$. 
    Then we choose a sufficiently small closed cube $B$ in $\T^r$ such that 
    $B\cap \Omega$ is a rational polyhedral subspace of $\T^r$ and $x\in \relint(B)$. 
    Set $U=\relint(B)\cap \Omega$. 
    Then there exists a triangulation $\mathcal{T}$ on $B\cap \Omega$ such that 
    $\XX_{B\cap \Omega}^{\mathcal{T}}|_U$ is $\XX^{\CC}_X|_U$-stratified and 
    the dimensions of the strata of $\XX_X^{\CC}$ and $\XX_{B\cap \Omega}^{\mathcal{T}}$ containing $x$ coincide. 
    We denote this dimension by $i$. 
    We can construct such a triangulation $\mathcal{T}$ by the induction as follows: 
    By the above construction, we have a rational polyhedral complex $\CC$ in $\T^r$ such that $|\CC|=B\cap \Omega$. 
    We consider the filtration in Example \ref{ex:CS} for $B\cap \Omega$ and $\CC$, 
    and denote by $(B\cap \Omega)^i$ the closed set in the filtration. 
    Since $(B\cap \Omega)^0$ is a discrete set, there exist the triangulation on $(B\cap \Omega)^0$ such that the points of $(B\cap \Omega)^0$ correspond to $0$-simplices. 

    Assume that we construct a triangulation on $(B\cap \Omega)^k$. 
    We will construct a triangulation on $(B\cap \Omega)^{k+1}$. 
    Let $\sigma\in \CC$ be a $(k+1)$-dimensional rational polyhedron. 
    Let $D^{k+1}$ and $\mathring{D}^{k+1}$ be the closed unit disc and the open unit disc in $\R^{k+1}$. 
    Since $\sigma$ is compact, by \cite[Example A.4]{Gross_2023}, the pair of topological space $(\sigma,\relint(\sigma))$ is homeomorphic to $(D^{k+1},\mathring{D}^{k+1})$. 
    Moreover, if $x$ is in $\sigma$, we may suppose that $x$ dees not correspond to the origin of $D^{k+1}$ by the homeomorphism by replacing the homeomorphism appropriately. 
    By the assumption of the induction, $\sigma \smallsetminus \relint(\sigma)$ is triangulated. 
    Since we may regard $\sigma$ as the closed cone of $\sigma \smallsetminus \relint(\sigma)$ by the above homeomorphism,  
    the triangulation on $\sigma \smallsetminus \relint(\sigma)$ induces a triangulation on $\sigma$ by considering the closed cone of each simplex of the triangulation on $\sigma \smallsetminus \relint(\sigma)$. 
    By triangulating each $(k+1)$-dimensional rational polyhedron $\sigma\in\CC$ in this way, we obtain a desired triangulation on $(B\cap \Omega)^{k+1}$. 
    As a result, we obtain a desired triangulation $\mathcal{T}$ on $B\cap \Omega$ by the induction. 

    By Example \ref{ex:triangulation}, the filtered space $(U,\XX_{B\cap \Omega}^{\mathcal{T}}|_U)$ is a CS set and therefore, there is an open neighborhood $N$ of $x$ that is stratified-homeomorphic to $\R^i\times cL$ via $h$ as in Definition \ref{def:CS}. 
    Let $v\in cL$ be the vertex of $cL$ and let $\pi\colon \R^i\times (cL\smallsetminus \{v\})\to L$ be the projection  from $\R^i\times (cL\smallsetminus\{v\})=(\R^i\times (0,1))\times L$ to $L$. 
    As $\XX_{B\cap\Omega}^{\mathcal{T}}|_U$ is $\XX_X^{\CC}|_U$-stratified, we can construct a new filtration on $L$ such that $L^k\colonequals \pi(h^{-1}(X^{i+k+1}_{\XX_X^{\CC}}\cap N)\smallsetminus\{v\})$ for $k\geq0$. 
    Then $\R^i\times cL$ with the filtration induced from this new filtration is stratified-homeomorphic to $N$ with the filtration induced by $\XX_X^{\CC}$. 
    Thus the condition \ref{itm:cs2} is shown. 
\end{proof}
In the same way as above, we obtain the following. 
\begin{proposition}\label{prop:tropfacesCS}
    Let $X$ be a rational polyhedral space and let  $\CC$ and $\CC'$ are open face structures on $X$. 
    Then $X$  is a CS set with filtration $\XX_X^{\CC\cap \CC'}$.
\end{proposition}
Finally, we introduce the notion of a tropical filtration. 
\begin{defn}[Tropical filtration]\label{def:tropfilter}
    Let $\mathcal{A}$ be the subcategory of rational polyhedral spaces whose morphisms are isomorphisms; in other words, let $\mathcal{A}$ be the core of the category of rational polyhedral spaces.  
    Moreover, let $\mathcal{B}$ be the category of filtered spaces whose morphisms are stratified maps preserving formal codimensions. 
    We define functors $p\colon \mathcal{A} \to \mathbf{Top}$ and $q\colon \mathcal{B} \to \mathbf{Top}$ to be the forgetful functors to the category of topological spaces. 
    A \emph{tropical filtration} is a functor $\mathcal{F}\colon \mathcal{A} \to\mathcal{B}$ satisfying $p=q\circ \mathcal{F}$. 
\end{defn} 
\begin{example}
The natural filtration \eqref{eq:canstrat} defines a tropical filtration $X\mapsto (X,\XX^{\text{trop}}_X)$. 
\end{example}
\begin{example}\label{ex:anothertropfil}
    Let $X$ be a rational polyhedral space, and 
    let $n=\dim((X,\XX^{\text{trop}}_X))$. 
    Let $U\subset X$ be the open subset of $X$ consisting of all the point that has a neighborhood which is an $n$-dimensional tropical manifold. 
    We define the $n$-dimensional filtration $\XX'_X$ on $X$ such that $X^n_{\XX'_X}\colonequals X$ and $X^k_{\XX'_X} \colonequals (X\smallsetminus U)\cap X^k_{\XX^{\text{trop}}_X}$. 
    Such a construction defines a tropical filtration $X\mapsto (X,\XX'_X)$. 
\end{example}

\subsubsection{Other preliminaries}

We introduce tools for CS sets. 
\begin{defn}\label{def:subcone}
    Let $L$ be a compact filtered space. 
    Let $cL=[0,1)\times L/ \{0\}\times L$ be the open cone over $L$ with vertex $v$. 
    For $I\subset [0,1)$, we define $L_I$  by the image under the quotient map $I\times L \to cL$. 
\end{defn}
For example, $L_{[0,1)}=cL$ and $L_{\{0\}}=v$. 

Let $X$ be a paracompact and separable CS set. 
In the following claim, we use the category $O_{X}$ whose objects are open subsets
    of $X$ and whose morphisms are inclusions, and the category $Q \text{-}\mathbf{Mod}_{*}$ which is the category of graded $Q$-modules for some fixed ring $Q$. 
\begin{prop}[Mayer-Vietoris argument {\cite[Proposition 13.2]{MR3807751}}]\label{prop:argument}
    Suppose that $X$ is a paracompact and separable CS set. 
    Let $F_{\bullet}$ and $\,G_{\bullet}\colon O_{X}\to Q \text{-}\mathbf{Mod}_{*}$
    be two functors and
    $\Phi\colon F_{\bullet}\to G_{\bullet}$ a natural transformation satisfying
    the conditions listed
    below.
    \begin{enumerate}[label=(MV\arabic*)]
    \item\label{itm:mv1} The functors $F_{\bullet}$ and $G_{\bullet}$ admit exact Mayer-Vietoris exact sequences, i.e. if $U,V$ are filtered open subspaces of $X$ (see the paragraph before Definition \ref{def:stratifiedmap}), 
    then there is an exact sequence
    \[
    \cdots\to F_i(U\cap V)\to F_i(U)\oplus F_i(V) \to F_i(U\cup V)\to F_{i-1}(U\cap V)\to\cdots,
    \]
    and similarly for $G_*$, such that the natural transformation $\Phi$ 
    induces a commutative diagram between these sequences.
    \item\label{itm:mv2} If $\{U_{\alpha}\}$ is a disjoint collection of open subsets of $X$  and $\Phi\colon F_{\bullet}(U_{\alpha})\to G_{\bullet}(U_{\alpha})$ is an isomorphism for each $\alpha$, then $\Phi\colon F_{\bullet}(\bigsqcup_{\alpha}U_{\alpha})\to G_{\bullet}(\bigsqcup_{\alpha}U_{\alpha})$  is an isomorphism.
    \item\label{itm:mv3} Any open subset $N\subset X$ which is stratified-homeomorphic to $\R^i\times cL$ for some compact filtered space $L$ satisfies the following condition:   
        \begin{itemize}
            \item For the vertex $v$ of $cL$,  if  $\Phi\colon F_{\bullet}(\R^i\times cL\smallsetminus \{v\})\to G_{\bullet}(\R^i\times cL\smallsetminus \{v\})$
                    is an isomorphism, then so is
                    $\Phi\colon F_{\bullet}(\R^i\times cL)\to G_{\bullet}(\R^i\times cL)$.   
        \end{itemize} 
    \item\label{itm:mv4} If $U$ is an open subset of $X$ contained within a single stratum and homeomorphic
    to a Euclidean space $\R^n$ or $\emptyset$, then $\Phi\colon F_{\bullet}(U)\to G_{\bullet}(U)$ is an isomorphism.
    \end{enumerate}
    Then $\Phi\colon F_{\bullet}(X)\to G_{\bullet}(X)$ is an isomorphism.
\end{prop}
\begin{rem}
    In \cite[Proposition 13.2]{MR3807751}, Proposition \ref{prop:argument} was stated for graded abelian groups and the category whose objects are (stratified-homeomorphic to) open subsets
    of a given paracompact and separable CS set $X$ and whose morphisms are  stratified homeomorphisms and inclusions. 
    However, Proposition \ref{prop:argument} holds with the same proof.
\end{rem}

For later arguments, we introduce the notion of a family of supports. 
Let $Q_X$ be the constant sheaf on $X$ associated to a ring $Q$. 
We call a sheaf of $Q$-modules on $X$ a $Q_X$-module. 
We denote the interior of a topological space $X$ by $\mathrm{int}(X)$. 
\begin{defn}[{\cite[I Definition 6.1]{Bredon_1997}}]\label{def:famofsupp}
    Let $X$ be a topological space. A \emph{family of supports} on $X$ is a family $\Phi$ of closed subsets of $X$ such that:
    \begin{enumerate}
        \item any closed subset of a member of $\Phi$ is a member of $\Phi$;
        \item $\Phi$ is closed under finite unions. 
    \end{enumerate}    
        A family $\Phi$ is said to be \emph{paracompactifying family of supports} if in addition:
    \begin{enumerate}
        \setcounter{enumi}{2}
        \item each element of $\Phi$ is paracompact;
        \item for each $K\in \Phi$, there exists $L\in \Phi$ such that $K\subset \mathrm{int}(L)$. 
    \end{enumerate}
\end{defn}
We denote by $c$ the family of supports consisting of all the compact subsets of $X$. 
Let $X$ be a rational polyhedral space. 
Then $c$ is a paracompactifying family of supports. 
The family of all the  closed subset of $X$ is also a paracompactifying family of supports. 
For a (pre)sheaf $\Q$ of $Q$-modules on $X$ and a family of supports $\Phi$ on $X$, we denote by $\Gamma_\Phi(\Q)$ the global section of $\Q$ whose support is in $\Phi$. 
If $\Phi$ is the family of supports consisting of all the closed subset of $X$, we omit $\Phi$ and simply denote $\Gamma_\Phi$ by $\Gamma$. 

We introduce basic notions for derived categories. For details on derived categories, we refer to \cite{MR1074006,MR2182076,MR932640}. 
Let $\DD(Q_X)$ denote the unbounded derived category of $Q_X$-modules, and 
$\DD(Q)$ denote the unbounded derived category of $Q$-modules. 
We write $\DD^b$ and $\DD^+$ for the bounded derived category and the bounded below derived category, respectively. 
For $\Gamma_\Phi$, we define the right derived functor $R\Gamma_\Phi \colon \DD(Q_X)\to\DD(Q)$. 
Then, for $\Q\in \DD(Q_X)$, we define the \emph{hypercohomology} by $$\Hyper_{\Phi}^q(\Q)\colonequals \HH^q(R\Gamma_\Phi(\Q)).$$ 
If $\Phi$ is the family of supports consisting of all the closed subset of $X$, we omit $\Phi$ and simply denote $\Hyper^q_{\Phi}(\Q)$ by $\Hyper^q(\Q)$. 
For a rational polyhedral space $X$ with an open face structure,  we can define the Verdier dualizing functor $\DD_X\colon \DD^b(Q_X)\to \DD^b(Q_X)$. 
\begin{rem}\label{rem:Verdier}
    We do not go into the details of the existence of the Verdier dual functor, but it can be verified roughly as follows: 

    We note that the cohomological dimensions $\dim_{\ZZ}(X)$ and $\dim_{Q}(X)$ satisfy $\dim_{\ZZ}(X)\leq \dim(X)$ and $\dim_{Q}(X)\leq \dim(X)$ and 
    that a Dedekind domain $Q$ has a finite global dimension (see \cite[Lemma 6.3.46]{Friedman_2020}, \cite[II Definitions 16.3 and 16.15]{Bredon_1997}, and \cite[Theorem C-5.128]{MR3677125}). 
Then the assumption in \cite[Definition 3.1.16]{MR1074006} to define the Verdier dualizing functor is satisfied. 
\end{rem}

\subsection{Tropical intersection (co)homology}\label{subsec:tropinthom}

In the following, we follow the notation in the post-published version of \cite{MR3894860}. 
Let  $\CC$ be  an open  rational polyhedral complex in $\TT^r$. 
Let $I\subset [r]$. 
For a face  $\sigma \in \CC_I$,  denote by $\mathbb{L} (\sigma) \subset \R^r_I$
the subspace generated by the vectors in $\R^r_I$ tangent to $\sigma$; in other words, $\mathbb{L} (\sigma)$ is the parallel transport to the origin of $\R^r_I$ of the affine span of $\sigma$. 
(For the definition of $\CC_I$, see the paragraph above Definition \ref{def:facestructure}.)
Set $\mathbb{L}_{\Z} (\sigma) = \mathbb{L} (\sigma) \cap \Z^r_I \subset \Z^r_I$.

\begin{defn}[Multi-tangent {\cite[Definition 2.3]{MR3894860}}]
\label{def:multitangent}
 Let $p$ be a non-negative integer.  For $\sigma \in \CC_I$,
the  $p$-th \emph{integral multi-tangent} 
    of $\CC$ at $\sigma$ is the $\Z$-module
\[{\bf F}^{\Z}_p(\sigma) =  \left(\sum \limits_{ \sigma' \in \CC_I: \sigma \prec \sigma'} \bigwedge^p \Linear(\sigma')\right) \cap \bigwedge^p \Z^r_I.\]
If  $\tau$ is a face of $\sigma$, then there are natural  maps 
\begin{align*}
\imap \colon {\bf F}^{\Z}_p(\sigma) \to {\bf F}^{\Z}_p(\tau). 
\end{align*}
\end{defn}
For  any Dedekind domain $Q$ with $\Z \subset Q \subset \R$, define the $Q$-module ${\bf F}^Q_{p}(\sigma)$ by 
${\bf F}^Q_{p}(\sigma) \colonequals { \bf F}^{\Z}_{p}(\sigma) \otimes Q$. 
Let $\imap^Q\colon {\bf F}^{Q}_p(\sigma) \to {\bf F}^{Q}_p(\tau)$ be the homomorphism of $Q$-modules induced by $\imap$.   

\begin{rem}
    For any $\sigma\in \CC$, ${\bf F}^Q_{0}(\sigma)=Q$.  
\end{rem}

In the following, we always consider a filtered rational polyhedral space $(X,\XX)$ with an open face structure $\CC$. (Here, we do not assume that $\XX=\XX_X^{\text{trop}}$.) 
In this case, the $Q$-module ${\bf F}^{Q}_p(\sigma)$ 
and the map $\imap^Q$ 
are well-defined for any $\tau,\sigma \in \CC$ with $\tau\prec \sigma$ and any chart. 
These are independent of the choice of a chart up to isomorphism. 

Let $q$ be a non-negative integer. We let $\Delta_q\subset \R^q$ denote an abstract $q$-dimensional simplex defined by the convex hull of the following points:
\[
v_0=(0,\cdots,0), v_1=(1,0,\cdots,0),\cdots,v_q=(0,\cdots,0,1)\in \R^q. 
\] 
We write $\Delta_q=[v_0,\cdots,v_q]$. The \emph{k-skeleton} of $\Delta_q$ is the union of faces of $\Delta_q$ of dimension at most $k$. 
By definition, if $k\geq q$, then the $k$-skeleton of $\Delta_q$ is equal to $\Delta_q$.

A \emph{$\CC$-stratified $q$-simplex} in $X$ is a continuous map $\delta \colon \Delta_q\to X$
such that
\begin{itemize}
\item 
for each face $\Delta' \subset \Delta_q$, there exists $\tau\in\CC$ such that $\delta(\relint(\Delta')) \subset \relint(\tau)$. 
\end{itemize} 

Let $\delta_j=\delta|_{[v_0,\ldots,\hat{v_j},\ldots,v_q ]}$  for $\Delta_q=[v_0,\ldots,v_q]$. 
For $\tau\in \CC$, let $_{\GM}\HC_q(\tau)$ denote the abelian group generated by $\CC$-stratified $q$-simplices 
$\delta \colon \Delta_q \to X$ such that $\delta(\relint(\Delta_q)) \subset \relint(\tau)$, with the ordinary boundary map $\partial\colon \HC_q(\tau)\to \HC_{q-1}(\tau)$ as follows:
\begin{equation*}
   \partial\delta=\sum_{0\leq j\leq q}(-1)^j\delta_j. 
\end{equation*}

Let $\S_X$ be the singular locus (see the paragraph after Definition \ref{def:filteredspace}). 
For $\tau \in \CC$, let $\HC_q(\tau)$ denote the abelian group generated by $\CC$-stratified $q$-simplices 
$\delta \colon \Delta_q \to X$ such that $\delta(\relint(\Delta_q)) \subset \relint(\tau)$ and $\delta(\relint(\Delta_q))\not\subset \S_X$. 
We define the boundary map $\hat{\partial}\colon \HC_q(\tau)\to \HC_{q-1}(\tau)$ as follows:
\begin{equation*}
    \hat{\partial}\delta=\sum_{\substack{\mathrm{Im}(\delta_j)\not\subset \S_X,\\ 0\leq j\leq q}}(-1)^j\delta_j. 
\end{equation*}
As in  \cite[Lemma 6.2.3]{Friedman_2020}, $\hat{\partial}\hat{\partial}=0$. 
Indeed, $\hat{\partial}\delta-\partial\delta$ is a linear combination of simplices in $\S_X$. 
Then $\hat{\partial} \hat{\partial}\delta-\partial \partial\delta$ is also a linear combination of simplices in $\S_X$. 
Since $\partial \partial\delta=0$ holds, $\hat{\partial} \hat{\partial}\delta$ is a linear combination of simplices in $\S_X$. 
By definition, each coefficient of such simplices is zero. 
Therefore, $\hat{\partial} \hat{\partial}=0$.

\begin{defn}\label{def:tropchain}
    Let $(X,\XX)$ be a filtered rational polyhedral space with an open face structure $\CC$.  

    The group of $\GM$ \emph{tropical} $(p,q)$-\emph{chains}  
    with respect to $\CC$ with $Q$-coefficients is defined by
    \begin{align*}
        _{\GM}\HC_{p,q}(X, Q) &\colonequals \bigoplus_{\tau \in \CC} {\bf F}_p^{Q}(\tau) \otimes_{\Z} (_{\GM}\HC_q(\tau)) . 
    \end{align*}    

    The group of \emph{tropical} $(p,q)$-\emph{chains} 
        with respect to $\CC$ with $Q$-coefficients is defined by 
    \begin{align*}
        \HC_{p,q}(X, Q) &\colonequals \bigoplus_{\tau \in \CC} {\bf F}_p^{Q}(\tau) \otimes_{\Z} \HC_q(\tau) . 
    \end{align*}
\end{defn}

Let $\perv\colon \SS \rightarrow \Z$ be a perversity (see Definition \ref{def:perversity}). We say that a $\CC$-stratified $q$-simplex $\sigma:\Delta_q \rightarrow X$ is \emph{$\perv$-allowable} if for any stratum $S$ of $X$,
\begin{equation}\label{eq:allowablecond}
    \sigma^{-1}(S)\subset (q-\codim (S) +\perv (S))\text{-skeleton of }\Delta_q. 
\end{equation}
Here $\codim (S)$ is the formal codimension.

To simplify the notation in the following, the notation $_{(\GM)}\HC_{p,q}(X, Q)$ 
stands for both $\HC_{p,q}(X, Q)$ and $_{\GM}\HC_{p,q}(X, Q)$. 
In other words, $(\GM)$ stands for both $\GM$ and non-$\GM$. 
We say that a $\CC$-stratified $q$-simplex $\sigma$ \emph{constitutes} a chain $\xi\in {}_{(\GM)}\HC_{p,q}(X, Q)$ if the coefficient of $\sigma$ in $\xi$ is not zero. 

A chain $\xi \in _{\GM}\HC_{p,q}(X, Q)$ (resp. $\HC_{p,q}(X, Q)$)  is $\perv$-allowable if all the $\CC$-stratified $q$-simplices constituting $\xi$ and all the $\CC$-stratified $(q-1)$-simplices constituting $\hat{\partial}\xi$ (resp. $\partial \xi$) 
are $\perv$-allowable. Let $_{(\GM)}\IC_{p,q}(X, Q)$ denote the abelian subgroup of $\perv$-allowable chains in $_{(\GM)}\HC_{p,q}(X, Q)$. 

In the following, by a \emph{$q$-simplex}, we mean a $\CC$-stratified $q$-simplex. By a simplex, we mean a $q$-simplex for some $q$. 
If the perversity $\perv$ is clear from the context, we simply say that a simplex $\sigma$ is \emph{allowable} (omitting $\perv$).  Similarly, we say that a chain $\xi\in {}_{(\GM)}\HC_{p,q}(X, Q)$ is \emph{allowable}. 
\begin{defn}\label{def:tropchainhomo}
    Let $(X,\XX)$ be a filtered rational polyhedral space with an open face structure $\CC$. 
    The complex of ($\GM$) tropical intersection $(p, \bullet)$-chains  
    with respect to $\CC$, with a perversity $\perv$ and with $Q$-coefficients is 
    \[(_{(\GM)}\IC_{p,\bullet}(X, Q), \hat{\partial}\text{ (resp. $\partial$)}). \]

    The  \emph{tropical intersection homology} groups with coefficients in $Q$ are defined to be  
    \begin{align*}
        _{(\GM)}\IH_{p,q}(X, Q) \colonequals \HH_q( _{(\GM)}\IC_{p,\bullet}(X, Q)).  
    \end{align*}
\end{defn}
\begin{defn}\label{def:tropchaincohomo}
    Let $(X,\XX)$ be a filtered rational polyhedral space with an open face structure $\CC$. 
    The complex of ($\GM$) tropical intersection $(p, \bullet)$-cochains
        with respect to $\CC$, with a perversity $\perv$ and with $Q$-coefficients is  
    \begin{align*}
        _{(\GM)}\IC^{p,\bullet}(X, Q) &\colonequals \Hom_Q(_{(\GM)}\IC_{p,\bullet}(X, Q), Q) , 
    \end{align*}
    where the boundary map $\hat{d}$ (resp. $d$) is dual to $\hat{\partial}$ (resp. $\partial$). 

    The  \emph{tropical intersection cohomology} groups with coefficients in $Q$ are defined to be  
    \begin{align*}
        _{(\GM)}\IH^{p,q}(X, Q) \colonequals \HH^q( _{(\GM)}\IC^{p,\bullet}(X, Q)). 
    \end{align*}
\end{defn}

We use the notation $_{(\GM)}\IH$ to indicate that the statement holds for both $\IH \text{ and } _{\GM}\IH$. 
To make the filtration $\XX$ and the open face structure $\CC$ explicit, we may also denote the (co)homology by $_{(\GM)}\IH_{p,q}(X,\XX,\CC, Q)$ and $_{(\GM)}\IH^{p,q}(X,\XX,\CC,Q)$, respectively. 
Furthermore, we sometimes omit obvious parts (the filtration or the open face structure) of the above notation such as $_{(\GM)}\IH_{p,q}(X,\XX, Q)$, $_{(\GM)}\IH_{p,q}(X,\CC, Q)$. 
We will show in Proposition \ref{prop:indep} that the homology $_{(\GM)}\IH_{p,q}(X,\XX,\CC, Q)$ does not depend on the choice of an open face structure $\CC$. 

For any filtered open subspace $U\subset X$, a perversity $\perv$ on $X$ naturally induces a perversity on $U$. 
Therefore, in this setting, we also denote the (co)homology of $U$ for the restriction of $\perv$ to $U$ by $_{(\GM)}\IH_{p,q}(U, Q)$ and $_{(\GM)}\IH^{p,q}(U,Q)$, respectively. 
\begin{rem}\label{rem:GMandTropical}
As in \cite[Definition 2.6]{MR3894860}, we define the \emph{tropical homology} by $$\HH_{p,q}(X, Q)\colonequals \HH_q(_{\GM}\HC_{p,\bullet}(X, Q)).$$ 
    Since ${}_{\GM}\IC_{p,q}(X,Q)$ is a sub-complex of ${}_{\GM}\HC_{p,q}(X,Q)$, the natural map ${}_{\GM}\IH_{p,q}(X,Q)\to \HH_{p,q}(X,Q)$ exists. 
\end{rem}

For an open subset $U\subset X$, we define a relative version of a chain complex ${}_{(\GM)}\IC_{p,\bullet}(X,U,Q)$ as in the classical intersection homology. 
To be precise, we define $${}_{(\GM)}\IC_{p,\bullet}(X,U,Q) \colonequals {}_{(\GM)}\IC_{p,\bullet}(X,Q)/{}_{(\GM)}\IC_{p,\bullet}(U,Q).$$ 
This is a chain complex with the boundary map induced by the boundary of ${}_{(\GM)}\IC_{p,\bullet}(X,Q)$. 
A relative version of a cochain complex ${}_{(\GM)}\IC^{p,\bullet}(X,U,Q)$ is defined to be the dual cochain of ${}_{(\GM)}\IC_{p,\bullet}(X,U,Q)$. 
\begin{defn}\label{Borel--Moore}
    The \emph{tropical Borel--Moore intersection chain complex}  is defined by 
    \[
    _{(\GM)}\IC_{p,\bullet}^{\BM}(X, Q)\colonequals \varprojlim_{K:compact}{}_{(\GM)}\IC_{p,\bullet}(X,X\smallsetminus K,Q),
    \] 
    where $K$ runs through all compact subsets of $X$. 
    The \emph{tropical Borel--Moore intersection homology} $_{(\GM)}\IH_{p,q}^{\BM}(X, Q)$ is defined as the homology of $_{(\GM)}\IC_{p,\bullet}^{\BM}(X, Q)$. 

    Similarly, the \emph{tropical compactly supported intersection cochain complex}  is defined by
    \[
    _{(\GM)}\IC^{p,\bullet}_c(X, Q)\colonequals \varinjlim_{K:compact}{}_{(\GM)}\IC^{p,\bullet}(X,X\smallsetminus K ,Q),
    \]
    where $K$ runs through all compact subsets of $X$. 
    The \emph{tropical compactly supported intersection cohomology} $_{(\GM)}\IH^{p,q}_c(X, Q)$ is defined as the cohomology of $_{(\GM)}\IC^{p,\bullet}_c(X, Q)$. 
\end{defn}
By the above definition, if $X$ is compact, we have ${}_{(\GM)}\IC_{p,\bullet}^{\BM}(X, Q)\cong{}_{(\GM)}\IC_{p,\bullet}(X, Q)$ and ${}_{(\GM)}\IC^{p,\bullet}_c(X, Q)\cong {}_{(\GM)}\IC^{p,\bullet}(X, Q)$. 

\begin{rem}\label{rem:perv}
    \leavevmode
    \begin{enumerate}
        \item   Since ${\bf F}^Q_{0}(\sigma)=Q$ for any $\sigma\in\CC$, the $p=0$ part of the tropical intersection (co)homology coincides with the classical intersection (co)homology for any perversity $\perv$. 
        \item Assume that $\perv$ satisfies $\perv(S) \geq \codim(S)$. Then any simplex is $\perv$-allowable. 
        Therefore, $_{\GM}\IH_{p,q}(X, Q) = \HH_{p,q}(X, Q)$ holds. 
        \item   If $\perv$ satisfies $\perv(S) \leq \codim(S)-2 $, then non-$\GM$ (co)homology and $\GM$ (co)homology coincide. 
                Indeed, with such a perversity $\perv$, for any $q$-simplex $\sigma$ and any stratum $S$ of $X$, we have $\sigma^{-1}(S)\subset (q-2)$-skeleton of $\Delta_q$. 
                Thus, any $\perv$-allowable simplices and their boundaries are never entirely contained in $\S_X$. 
                Therefore, any allowable chain in the $\GM$ case is also allowable in non-$\GM$ case, and  
                the boundaries of an allowable chain in the $\GM$ case coincides that in non-$\GM$ case.  
    \end{enumerate}
    
\end{rem}
Let ${}_{(\GM)}\HC^{lf}_{p,\bullet}(X,Q)$ be the $Q$-module obtained from the construction of ${}_{(\GM)}\HC_{p,\bullet}(X,Q)$ by allowing locally finite sums of simplices. 
Then we define ${}_{(\GM)}\IC_{p,\bullet}^{lf}(X,Q)$ to be the submodule of ${}_{(\GM)}\HC^{lf}_{p,\bullet}(X,Q)$ consisting of $\xi$ such that all the simplices constituting $\xi$ and the boundary of $\xi$ are $\perv$-allowable. 
\begin{lemma}\label{lem:lf}
    The two chain complexes ${}_{(\GM)}\IC_{p,\bullet}^{lf}(X,Q)$ and ${}_{(\GM)}\IC_{p,\bullet}^{\BM}(X,Q)$ are chain homotopy equivalent. 
    In particular, $\HH_q({}_{(\GM)}\IC_{p,\bullet}^{lf}(X,Q))\cong {}_{\GM}\IH_{p,q}^{\BM}(X,Q)$. 
\end{lemma}
\begin{proof}
    For $\xi \in {}_{(\GM)}\IC_{p,q}^{\BM}(X,Q)$, we construct $\zeta\in{}_{(\GM)}\IC_{p,q}^{lf}(X,Q)$. 
    Let $\sigma\colon \Delta_q\to X$ be a $\perv$-allowable $q$-simplex. 
    We take any compact set $K$ such that $\Im(\sigma)\subset K$. 
    By properties of a projective limit, there exists the projection map $\pi\colon {}_{(\GM)}\IC_{p,q}^{\BM}(X,Q)\to {}_{(\GM)}\IC_{p,q}(X,X\smallsetminus K,Q)$. 
    Let $\xi_\sigma\in Q$ be the coefficient of $\sigma$ in $\pi(\sigma)$. 
    Then $\xi_\sigma$ is independent of the choice of $K$. 
    Thus  $\zeta=\sum \xi_{\sigma}\sigma$ is the element of ${}_{(\GM)}\IC_{p,q}^{lf}(X,Q)$. 
    We define $f\colon {}_{(\GM)}\IC_{p,q}^{\BM}(X,Q)\to {}_{(\GM)}\IC_{p,q}^{lf}(X,Q)$ by $f(\xi)\colonequals \zeta$. 

    Next, for $\zeta \in {}_{(\GM)}\IC_{p,q}^{lf}(X,Q)$, we construct $\xi\in {}_{\GM}\IC_{p,q}^{\BM}(X,Q)$. 
    Let $b\colon {}_{(\GM)}\IC_{p,q}^{lf}(X,Q) \to {}_{(\GM)}\IC_{p,q}^{lf}(X,Q)$ be the barycentric subdivision operator. 
    In other words, $b$ is a map obtained by performing barycentric subdivision on each simplex. 
    Let $\zeta_{\sigma}\in Q$ be the coefficient of $\sigma$ in $b(\zeta)$. 
    We take any compact set $K$. We denote by $K_{\zeta}$ the set consisting a simplex constituting $b(\zeta)$ and intersecting $K$. 
    Then, since $K$ is compact, $K_{\zeta}$ is a finite set.  
    By the proof of \cite[Lemma 2.12]{MR2276609}, there exists a finite set $K'_{\zeta}$ of simplices constituting $b(\zeta)$ such that $\sum_{\sigma\in K'_{\zeta}}\zeta_{\sigma}\sigma\in {}_{(\GM)}\IC_{p,q}(X,Q)$ and $K_{\zeta}\subset K'_{\zeta}$. 
    Therefore, $\sum_{\sigma\in K'_{\zeta}}\zeta_{\sigma}\sigma$ defines the element in ${}_{(\GM)}\IC_{p,q}(X,X-K,Q)$. 
    Then $\zeta_{\sigma}$ is independent of the choice of $K$. 
    Since $K$ is an arbitrary compact subset of $X$, this defines the element in ${}_{\GM}\IC_{p,q}^{\BM}(X,Q)$. 
    We define $g\colon{}_{\GM}\IC_{p,q}^{lf}(X,Q)\to {}_{\GM}\IC_{p,q}^{\BM}(X,Q)$ by $g(\zeta)\colonequals \xi$. 

    Then $f\circ g$ and $g\circ f$ are maps performing barycentric subdivision on each simplex. 
    Such maps are homotopic to the identities by the argument in \cite[Propositions 4.4.14 and 6.3.9]{Friedman_2020}. 
    Therefore, we obtain the claim. 
\end{proof}

\section{Properties of tropical intersection (co)homology}\label{sec:property}

\subsection{Basic properties}\label{subsec:basic}
In this subsection, we show basic properties of tropical intersection homology, analogous to those of classical intersection homology. 
We show excision, Mayer-Vietoris sequence, and cone formulas. 
The proofs closely mirror the classical setting, and we refer to \cite{Friedman_2020} for further details. 
In this subsection, we omit the coefficient ring $Q$, which is a Dedekind domain with $\Z\subset Q \subset \R$.

Let $X$ be a filtered rational polyhedral space with an open face structure. 
We call a family of open subsets $\mathcal{U}=\{U\}$ with $\bigcup_{U\in \mathcal{U}}U=X$ an open cover of $X$. 
For an open cover $\mathcal{U}$ of $X$,  
we denote by $_{(\GM)}\IC_{p,q}^{\mathcal{U}}(X)$ the submodule $\sum_{U\in \mathcal{U}}{}_{(\GM)}\IC_{p,q}(U) \subset {}_{(\GM)}\IC_{p,q}(X)$ generated by $\perv$-allowable chains contained in some $U\in\mathcal{U}$. 
Then we define $_{(\GM)}\IH_{p,q}^{\mathcal{U}}(X)$ to be  the homology of $_{(\GM)}\IC_{p,\bullet}^{\mathcal{U}}(X)$. 
\begin{prop}\label{prop:excision}
    Let $X$ be a filtered rational polyhedral space with an open face structure. 
    Then $_{(\GM)}\IH_{p,q}(X)\cong {}_{(\GM)}\IH_{p,q}^{\mathcal{U}}(X)$. 
\end{prop}
\begin{proof}
    By applying the same argument as in \cite[Lemma 6.5.4]{Friedman_2020}, there exists a subdivision operator $T\colon {}_{(\GM)}\IC_{p,q}(X)\to {}_{(\GM)}\IC^{\mathcal{U}}_{p,q}(X)$. 
    (This operator $T$ is constructed by the induction using a relative barycentric subdivision. 
    Roughly speaking, $T$ is obtained by subdividing the simplices finely enough so that each simplex is contained in some $U\in\mathcal{U}$. 
    For details, see \cite[Sublemma 6.5.5]{Friedman_2020}.)
    This morphism $T$ is a quasi-isomorphism by the argument \cite[Propositions 4.4.14 and 6.3.9]{Friedman_2020}. 
\end{proof}
This statement implies excision. 
\begin{prop}[Excision]\label{prop:oriexcision}
    Let $X$ be a filtered rational polyhedral space with an open face structure. 
    Suppose that $U$ and $V$ are open subsets of $X$ such that $U\cup V=X$. 
    Then there exists an isomorphism $_{(\GM)}\IH_{p,q}(X,U)\cong _{(\GM)}\IH_{p,q}(V,U\cap V)$. 
\end{prop}
\begin{proof}
    Let ${}_{(\GM)}\IC_{p,\bullet}(U+V)$ denote ${}_{(\GM)}\IC_{p,\bullet}(U)+{}_{(\GM)}\IC_{p,\bullet}(V)$. 
    By Proposition \ref{prop:excision}, the chain group $_{(\GM)}\IC_{p,\bullet}(X)$ is quasi-isomorphic to ${}_{(\GM)}\IC_{p,\bullet}(U+V)$. 
    Using the following diagram:
    \[
    \begin{tikzcd}[column sep=1.0em]
        0 \arrow[r] & {}_{(\GM)}\IC_{p,\bullet}(U) \arrow[r] \arrow[d,"="] & {}_{(\GM)}\IC_{p,\bullet}(U+V) \arrow[r] \arrow[d,"\text{q.i.}"] & {}_{(\GM)}\IC_{p,\bullet}(U+V)/{}_{(\GM)}\IC_{p,\bullet}(U) \arrow[r] \arrow[d] & 0 \\
        0 \arrow[r] & {}_{(\GM)}\IC_{p,\bullet}(U) \arrow[r]  & {}_{(\GM)}\IC_{p,\bullet}(X) \arrow[r]  & {}_{(\GM)}\IC_{p,\bullet}(X,U) \arrow[r] & 0,
    \end{tikzcd}
    \]
    we obtain a map between the induced long exact sequences. 
    Applying the five lemma, we obtain $${}_{(\GM)}\IH_{p,q}(X,U)\cong \HH^q({}_{(\GM)}\IC_{p,\bullet}(U+V)/{}_{(\GM)}\IC_{p,\bullet}(U))\cong {}_{(\GM)}\IH_{p,q}(V,U\cap V),$$
    as desired. 
\end{proof}

\begin{prop}[Mayer-Vietoris sequences]\label{prop:Mayer}
    Let $X$ be a filtered rational polyhedral space with an open face structure. 
    Let $U,V \subset X$ be open subsets of $X$ such that $U\cup V=X$ and $A \subset U$, $B \subset V$ be open subsets. Then there exist the following exact sequences. 
    \begin{enumerate}
        \item \text{homology}
        \begin{align*} 
        \cdots \to
        {}_{(\GM)}\IH_{p,q}(U\cap V,A \cap B) \xrightarrow{}
        {}_{(\GM)}\IH_{p,q}(U,A) \oplus {}_{(\GM)}\IH_{p,q}(V,B) \\
        \xrightarrow{}
        {}_{(\GM)}\IH_{p,q}(X,A\cup B) 
        \xrightarrow{}
        {}_{(\GM)}\IH_{p,q-1}(U\cap V,A \cap B) \to \cdots
        \end{align*}

        \item \text{cohomology}
        \begin{align*}
        \cdots \to {}_{(\GM)}\IH^{p,q}(X,A\cup B) \xrightarrow{}
        {}_{(\GM)}\IH^{p,q}(U, A) \oplus {}_{(\GM)}\IH^{p,q}(V, B) \\
        \xrightarrow{}
        {}_{(\GM)}\IH^{p,q}(U \cap V,A \cap B) 
        \xrightarrow{}
        {}_{(\GM)}\IH^{p,q+1}(X, A\cup B) \to \cdots
        \end{align*}
        
        \item \text{BM homology}
        \begin{align*}
        \cdots \to {}_{(\GM)}\IH_{p,q}^{\BM}(X) \xrightarrow{}
        {}_{(\GM)}\IH_{p,q}^{\BM}(U) \oplus {}_{(\GM)}\IH_{p,q}^{\BM}(V) \xrightarrow{}
        {}_{(\GM)}\IH_{p,q}^{\BM}(U \cap V) \\
        \xrightarrow{}
        {}_{(\GM)}\IH_{p,q-1}^{\BM}(X) \to \cdots
        \end{align*}

        \item \text{compactly supported cohomology}
        \begin{align*}
        \cdots \to {}_{(\GM)}\IH_c^{p,q}(U\cap V) \xrightarrow{}
        {}_{(\GM)}\IH_c^{p,q}(U) \oplus {}_{(\GM)}\IH_c^{p,q}(V) \xrightarrow{}
        {}_{(\GM)}\IH_c^{p,q}(X) \\
        \xrightarrow{}
        {}_{(\GM)}\IH_c^{p,q+1}(U\cap V) \to \cdots  
        \end{align*}
    \end{enumerate}
\end{prop}
\begin{proof}
    First, we prove (1). 
    Let ${}_{(\GM)}\IC_{p,\bullet}(U+V,A+B)$ denote the chain complex \\${}_{(\GM)}\IC_{p,\bullet}^{\{U,V\}}(X)/{}_{(\GM)}\IC_{p,\bullet}^{\{A,B\}}(A\cup B)$. 
    By the same argument as in the classical case (see \cite[p. 152]{MR1867354}), we have an exact sequence:
    \begin{align}\label{eq:Mayer}
        0\to {}_{(\GM)}\IC_{p,\bullet}(U\cap V,A\cap B) \to {}_{(\GM)}\IC_{p,\bullet}(U,A)\oplus {}_{(\GM)}\IC_{p,\bullet}(V,B) \\
        \to {}_{(\GM)}\IC_{p,\bullet}(U+V,A+B) \to 0. \notag
    \end{align}
    By Proposition \ref{prop:excision} and the diagram argument as in the proof of Proposition \ref{prop:oriexcision}, ${}_{(\GM)}\IC_{p,\bullet}(U+V,A+B)$ is quasi-isomorphic to ${}_{(\GM)}\IC_{p,\bullet}(X,A\cup B)$. 
    Hence the long exact sequence associated with the above short exact sequence \eqref{eq:Mayer} yields our claim. 
    
    Next, we prove (2). 
    Since our complexes consist of projective modules, the dual complexes of \eqref{eq:Mayer} induce the following dual exact sequence:
    \begin{align*}
        0\to \Hom ({}_{(\GM)}\IC_{p,\bullet}(U+V,A+B),Q) \to {}_{(\GM)}\IC^{p,\bullet}(U,A)\oplus {}_{(\GM)}\IC^{p,\bullet}(V,B) \\
        \to {}_{(\GM)}\IC^{p,\bullet}(U\cap V,A\cap B) \to 0.
    \end{align*}
    The corresponding long exact sequence gives (2). 

    For (3), we set $U=V=X$, $A=X\smallsetminus K$, and $B=X\smallsetminus L$ in (1), where $K\subset U$ and $L\subset V$ are compact subsets. 
    Since $X$ is a rational polyhedral space, 
    there exists a sequence $\{K_i\}$ of compact subsets $K_i\subset U$ for $i\in \ZZ_{\geq 0}$  
    such that $U=\cup_{i\in \ZZ_{\geq 0}}K_i$ and $K_i\subset \text{int}(K_{i+1})$. 
    Then $\varprojlim_i {}_{(\GM)}\IC_{p,q}(U,U\smallsetminus K_i)$ is isomorphic to $\varprojlim_K {}_{(\GM)}\IC_{p,q}(U,U\smallsetminus K)$. 
    Similarly, there exists a sequence $\{L_i\}$ of compact subsets $L_i\subset V$ for $i\in \ZZ_{\geq 0}$ 
    such that $V=\cup_{i\in \ZZ_{\geq 0}}L_i$ and $L_i\subset \text{int}(L_{i+1})$. 
    Then $\varprojlim_i {}_{(\GM)}\IC_{p,q}(V,V\smallsetminus L_i)$ is isomorphic to $\varprojlim_L {}_{(\GM)}\IC_{p,q}(V,V\smallsetminus L)$. 
    By Proposition \ref{prop:oriexcision}, we have the two quasi-isomorphisms 
    ${}_{(\GM)}\IC_{p,q}(U,U\smallsetminus K_i)\to{}_{(\GM)}\IC_{p,q}(X,X\smallsetminus K_i)$ and 
    ${}_{(\GM)}\IC_{p,q}(V,V\smallsetminus L_i)\to{}_{(\GM)}\IC_{p,q}(X,X\smallsetminus L_i)$. 
    By taking the projective limit, we obtain quasi-isomorphisms 
    $\varprojlim_i {}_{(\GM)}\IC_{p,q}(U,U\smallsetminus K_i)\to \varprojlim_i {}_{(\GM)}\IC_{p,q}(X,X\smallsetminus K_i)$ and 
    $\varprojlim_i {}_{(\GM)}\IC_{p,q}(V,V\smallsetminus L_i)\to \varprojlim_i {}_{(\GM)}\IC_{p,q}(X,X\smallsetminus L_i)$. 
    Using \eqref{eq:Mayer}, we have an exact sequence:
    \begin{align}\label{eq:Mayerprebm}
        0\to {}_{(\GM)}\IC_{p,\bullet}(X,X\smallsetminus (K_i\cup L_i))  \to {}_{(\GM)}\IC_{p,\bullet}(X,X\smallsetminus K_i)\oplus {}_{(\GM)}\IC_{p,\bullet}(X,X\smallsetminus L_i) \\
     \to {}_{(\GM)}\IC_{p,\bullet}(X,X\smallsetminus K_i+X\smallsetminus L_i) \to 0, \notag
    \end{align}
    for $i\in \ZZ_{\geq 0}$. 
    By \cite[Proposition 10.2]{MR242802}, the projective limit of \eqref{eq:Mayerprebm} is also exact as follows: 
    \begin{align}\label{eq:Mayerbm}
    0\to \varprojlim_{i} {}_{(\GM)}\IC_{p,\bullet}(X,X\smallsetminus (K_i\cup L_i))  \to \varprojlim_{i} ({}_{(\GM)}\IC_{p,\bullet}(X,X\smallsetminus K_i)\oplus {}_{(\GM)}\IC_{p,\bullet}(X,X\smallsetminus L_i)) \\
     \to \varprojlim_{i} {}_{(\GM)}\IC_{p,\bullet}(X,X\smallsetminus K_i+X\smallsetminus L_i) \to 0. \notag
    \end{align}
    We have $\HH_q(\varprojlim_{i}  {}_{(\GM)}\IC_{p,\bullet}(X,X\smallsetminus (K_i\cup L_i)))\cong \IH_{p,q}^{\BM}(U\cup V)$. 
    Since $X\smallsetminus K_i\cup X\smallsetminus L_i=X\smallsetminus (K_i\cap L_i)$, there exists 
    a quasi-isomorphism $\varprojlim_{i} {}_{(\GM)}\IC_{p,\bullet}(X,X\smallsetminus K_i+X\smallsetminus L_i)\to \varprojlim_{i} {}_{(\GM)}\IC_{p,\bullet}(X,X\smallsetminus (K_i\cap L_i))$. 
    Therefore, we have $\HH_q(\varprojlim_{i} {}_{(\GM)}\IC_{p,\bullet}(X,X\smallsetminus K_i+X\smallsetminus L_i))\cong\IH_{p,q}^{\BM}(U\cap V)$. 
    By the associated long exact sequence induced by \eqref{eq:Mayerbm}, we have (3). 

    For (4), we again take $U=V=X$, $A=X\smallsetminus K$, and $B=X\smallsetminus L$ in (2), where $K\subset U$ and $L\subset V$ are compact subsets. Taking the inductive limit for $K\subset A$ and $L\subset B$, we obtain (4) as in \cite[Lemma 7.4.5]{Friedman_2020}. 
\end{proof}

\begin{defn}\label{def:polyfan}
    Let $v\in \T^r$. 
    We define an \emph{$n$-dimensional rational polyhedral fan}  $F$ with vertex $v$ in $\T^r$ to be a rational polyhedral space $F$ in $\T^r$ equipped with a filtration $\XX_F$ and an $\XX_F$-stratified open face structure $\CC_F$ 
    that satisfies the following conditions:
    \begin{enumerate}
        \item $\{v\}$ is the smallest face of $\CC_F$ with respect to $\prec$ of $\CC_F$;
        \item The formal dimension of $(F,\XX_F)$ is $n$;
        \item The formal dimension of the stratum $\{v\}$ is $0$; 
        \item $(F,\XX^{\CC_F}_F)$ is stratified-homeomorphic to the cone $cL$ of some compact filtered space $L$. 
    \end{enumerate}
\end{defn}

Note that if (1) above holds, $F\subset \R^r$, and $\CC_F$ is a complex consisting of (non-open) rational polyhedra in $\R^r$, then the condition (4) holds.  

As in the case of the classical intersection homology, we obtain cone formulas. 
\begin{prop}[Cone formula for $\GM$]\label{prop:coneGM}
    Let $F=(F,\XX_F,\CC_F)$ be an $n$-dimensional rational polyhedral fan with vertex $v$ in $\T^r$. 
    Then the following holds:
    \begin{align*}
        _{\GM}\IH_{p,q}(F)\cong
        \begin{cases}
            0 &\text{if } q \geq n-\perv(\{v\})-1\text{ and }q \neq 0,\\
            {\bf F}^{Q}_p(\{v\}) &\text{if } q \geq n-\perv(\{v\})\text{ and }q=0,\\
            \sum \limits_{ \substack{{\sigma' \in \CC_F},\\ \perv(\sigma')\geq \codim(\sigma')}}
                {\bf F}^{Q}_p(\sigma') &\text{if } q =n-\perv(\{v\})-1\text{ and }q=0,\\
            _{\GM}\IH_{p,q}(F\smallsetminus \{v\}) &\text{if } q < n-\perv(\{v\})-1.
        \end{cases}
    \end{align*}
    In the above, for $\sigma \in \CC_F$, we define $\perv(\sigma)\colonequals\perv(S)$ and $\codim(\sigma)\colonequals\codim(S)$, where the relative interior of $\sigma$ is in $S\in \SS_F$. 
    The above sum is taken in ${\bf F}^{Q}_p(\{v\})$. 
\end{prop}
\begin{proof}
    We follow the strategy in \cite[Theorem 4.2.1]{Friedman_2020}. 
    We first consider the case $q\geq n-\perv(\{v\})-1$ and $q \neq 0$. 
    We fix a stratified-homeomorphism $F\overset{\cong}{\to}cL$. 
    For a $q$-simplex $\sigma\colon \Delta_q\to F\cong cL$, we define $\bar{c}\sigma$ to be the $(q+1)$-simplex constructed as follows: 
    We write $\sigma(x)=(\sigma_I(x),\sigma_L(x))\in [0,1)\times L/ \{0\}\times L=cL$ as in \cite[Example 3.4.7]{Friedman_2020}. 
    Here $\sigma_I\colon \Delta_q\to [0,1)$ and $\sigma_L\colon \Delta_q \to L$ are maps. 
    (Note that $\sigma_L$ may not be unique or continuous. However, this causes no problem because the subsequent argument does not depend on the choice of $\sigma_L$.) 
    We regard $\Delta_{q+1}$ as the closed cone $\bar{c}\Delta_q=[0,1]\times \Delta_q / (\{0\}\times \Delta_q)$. 
    Then we obtain the function $\bar{c}\sigma\colon\Delta_{q+1}=[0,1]\times \Delta_q / (\{0\}\times \Delta_q)\ni (s,x)\mapsto (s\sigma_I(x),\sigma_L(x))\in cL\cong F$. 
    Each $q$-face of $\bar{c}\sigma$ is $\sigma$ or $\bar{c}\tau$ for some $(q-1)$-face $\tau$ of $\sigma$. 
    For a chain $\xi=\sum\sigma\in {}_{\GM}\IC_{p,q}(F)$, we define $\bar{c}\xi=\sum \bar{c}\sigma\in {}_{\GM}\HC_{p,q+1}(F)$. 
    By the argument in \cite[p. 131]{Friedman_2020}, if $q\geq n-\perv(\{v\})-1$ and $\sigma$ is an allowable $q$-simplex, 
    then $\bar{c}\sigma$ is also allowable. 
    It follows that  $\bar{c}\xi\in {}_{\GM}\IC_{p,q+1}(F)$. 
    If $\xi$ satisfies $\partial\xi=0$ with $q>0$, then $\partial \bar{c}\xi=-\bar{c}\partial\xi+\xi=\xi$. 
    Here we use the assumption $q\neq 0$. 
    Thus $\xi$ is a boundary, and the homology in such a case is trivial. 
    
    Since we consider the $\GM$ case, $v$ may be an allowable $0$-simplex. 
    Thus the case $q=0$ is more intricate  than the non-$\GM$ case. 

    Let us consider the case $q \geq n-\perv(\{v\})$ and $q=0$.  
    For a $0$-simplex $v'$ of $F$, let $\sigma'\in \CC_F$ denote the smallest polyhedron including $v'$. 
    We denote $\Im(\iota_{\{v\},\sigma'}\colon{\bf F} ^Q_p(\sigma')\to {\bf F} ^Q_p(\{v\}))$ by ${\bf F} ^Q_p(v')$. 
    Then $\{v\}\prec \sigma'$ holds and ${\bf F} ^Q_p(v')\subset {\bf F} ^Q_p(\{v\})$. 
    If $v'$ is an allowable 0-cycle, then $\bar c{v'}$ is also allowable. 
    We write a $0$-cycle $\xi\in {}_{\GM}\IC_{p,0}(F)$ as $\xi=\sum f_{v'}v'$ for some $f_{v'}\in {\bf F} ^Q_p(\sigma')$. 
    Then $\xi-\sum f_{v'}\partial\bar{c}{v'}\in {}_{\GM}\IC_{p,0}(F)$ is a $0$-cycle which can be written as $f_v v$ for some $f_v\in {\bf F} ^Q_p(\{v\})$ and the $0$-simplex $v$. 
    Therefore, any $0$-cycle $\xi$ is homologous to $f_{\xi} v$ for some $f_{\xi}\in {\bf F} ^Q_p(\{v\})$. 
    Thus we have a surjective map
    \begin{equation}\label{eq:coneGMzero}
        {\bf F} ^Q_p(\{v\})\ni f \mapsto fv \in {}_{\GM}\IH_{p,0}(F).
    \end{equation}
    We need to show that this map is also injective. 
    Let $\xi=\sum f_{\sigma}\sigma\in {}_{\GM}\IC_{p,1}(F)$ be a $1$-chain and let $v^{\sigma}_1$ and $v^{\sigma}_2$ be points such that $\partial\sigma=v^{\sigma}_1-v^{\sigma}_2$. 
    We define $\xi'$ to be $\sum f_{\sigma}(\bar{c}v^{\sigma}_1-\bar{c}v^{\sigma}_2)\in {}_{\GM}\IC_{p,1}(F)$. 
    Then $\partial\xi=\partial\xi'$ holds. 
    Moreover, $\partial\xi'\in {}_{\GM}\IC_{p,0}(F)$ can be written as $fv-\sum_{v'\neq v} f_{v'}v'$, where $f=\sum_{v'\neq v}\iota_{\{v\},\sigma'}(f_{v'})$. 
    Therefore, if $\partial\xi$ can be written as $fv$, then $\partial\xi=0$. 
    In other words, if $fv$ is $0$ in ${}_{\GM}\IH_{p,0}(F)$, then $f=0$.  
    It follows that the above map \eqref{eq:coneGMzero} is injective.  
    Hence we obtain the statement for the case $q \geq n-\perv(\{v\})$ and $q=0$. 
    This argument is valid for $\HH_{p,0}(X)$. Thus we have $\HH_{p,0}(X)\cong {\bf F} ^Q_p(\{v\})$. 

    For the case $q =n-\perv(\{v\})-1$ and $q=0$, the $0$-simplex $v$ is not allowable since the allowable condition for the stratum $\{v\}$ and $0$-simplex $v$ is $\Delta_0 \subset (0-n+\perv(\{v\}))$-skeleton of $\Delta_0$ and $n+\perv(\{v\})<0$. 
    On the other hand, for any allowable $0$-cycle $v'$, 
    $\bar{c}v'$ is allowable because $\bar{c}v'$ satisfies the allowable condition $0\leq 1-n+\perv(\{v\})$ for the stratum $\{v\}$. 
    We consider the map 
    \begin{equation}\label{eq:conegm0case}
        {}_{\GM}\IH_{p,0}(F)\ni \sum f_{v'}v' \to (\sum \iota_{\{v\},\sigma'}(f_{v'}))v \in \HH_{p,0}(F)\cong {\bf F} ^Q_p(\{v\})
    \end{equation}
    in Remark \ref{rem:GMandTropical}. 
    If $\sum \iota_{\{v\},\sigma'}(f_{v'})=0$,  then $\sum f_{v'}v'$ is the boundary of $\sum f_{v'}\bar{c}v'\in {}_{\GM}\IC_{p,1}(F)$. 
    It follows that the map \eqref{eq:conegm0case} is injective. 
    Thus the 0-degree homology is 
    given by the sum of ${\bf F}^Q_p(v')$ for allowable 0-cycles $v'$ in $\HH_{p,0}(F)\cong{\bf F}^Q_p(\{v\})$. 

    Finally, we consider the case $q < n-\perv(\{v\})-1$. 
    Let $\sigma$ be a $q+1$-simplex intersecting $v$. 
    Then the allowable condition for $\sigma$ and the stratum $\{v\}$ is $\sigma^{-1}(\{v\})\subset ((q+1)-n+\perv(\{v\}))$-skeleton of $\Delta_{q+1}$. 
    However,  $(q+1)-n+\perv(\{v\})<0$ holds here. 
    Thus $\sigma$ is not allowable. 
    For any $q$-simplex $\sigma$ intersecting $v$, $\sigma$ is not allowable similarly.  
    Hence, all the allowable chains in $_{\GM}\IC_{p,q}(F)$ and $_{\GM}\IC_{p,q+1}(F)$ are in $F\smallsetminus \{v\}$, and  the statement holds. 
\end{proof} 

\begin{prop}[Cone formula for non-$\GM$]\label{prop:conenon}
    Let $F$ be an $n$-dimensional rational polyhedral fan with vertex $v$ in $\T^r$. 
    Then, the following holds:
    \begin{align*}
        \IH_{p,q}(F)\cong
        \begin{cases}
            0 &\text{if } q \geq n-\perv(\{v\})-1,\\
            \IH_{p,q}(F\smallsetminus \{v\}) &\text{if } q < n-\perv(\{v\})-1.
        \end{cases}
    \end{align*}
\end{prop}  
\begin{proof}
    The proof is essentially the same as in the $\GM$ case.  
    The only difference occurs in the case $q=0$, which corresponds to the second and third cases in the $\GM$ setting. 
    In such cases, under the non-$\GM$ setting, since the boundary of $\bar{c}{v'}$ is $v'$,
    $1$-chain $\bar{c}{v'}$ is allowable for any allowable $0$-chain $v'$. 
    Hence, in each case, any allowable $0$-cycle $\xi=\sum f_{v'} v'$ is the boundary of the allowable $1$-chain $\sum f_{v'}\bar{c}{v'}$. 
    Therefore, the homology is trivial in these cases.  
\end{proof}
\begin{lemma}\label{lem:inclusion}
    Let $F$ be an $n$-dimensional rational polyhedral fan with vertex $v$ in $\T^r$. 
    We identify $F$ with $cL$ in Definition \ref{def:polyfan}. 
    Recall that for $I\subset [0,1)$, $L_I$ is the image of $I\times L$ under the quotient map $I\times L \to cL$ (see Definition \ref{def:subcone}). 
    For any $1>r>r'\geq 0$, the inclusion $F\smallsetminus L_{[0,r]}\to F\smallsetminus L_{[0,r']}$ induces an isomorphism ${}_{(\GM)}\IH_{p,q}(F\smallsetminus L_{[0,r]})\cong{}_{(\GM)}\IH_{p,q}(F\smallsetminus L_{[0,{r'}]})$.  
\end{lemma}
\begin{proof}
    We have $L_{(r,1)}=F\smallsetminus L_{[0,r]}$ and $L_{(r',1)}=F\smallsetminus L_{[0,r']}$. 
    Let $i\colon L_{(r,1)}\to L_{(r',1)}$ denote the inclusion. 
    For any a simplex $\sigma$ of $F\smallsetminus L_{(r,1)}$, we write $\sigma(x)=(\sigma_{(r,1)}(x),\sigma_L(x))\in (r,1)\times L=L_{(r,1)}$.  
    We define $j\colon  \IC_{p,q}(L_{(r,1)},Q)\to \IC_{p,q}(L_{(r',1)},Q)$ by $(\sigma_{(r,1)},\sigma_L)\mapsto (\frac{1-r'}{1-r}\sigma_{(r,1)}-\frac{r-r'}{1-r},\sigma_L)$. 
    By the argument in \cite[Propositions 4.1.10 and 6.3.7]{Friedman_2020}, $i\circ j$ and $j\circ i$ are chain homotopic to the identities. 
    Therefore, $i$ is a chain homotopy equivalence and we obtain the claim. 
\end{proof}

Following the strategy of \cite[Proposition 3.7]{Chataur_2019}, we prove the following proposition. 
\begin{prop}[Cone formula for non-$\GM$ $\BM$]\label{prop:conebm}
    Let $F$ be an $n$-dimensional rational polyhedral fan with vertex $v$ in $\T^r$. 
    Then the following holds:
    \begin{align*}
        \IH_{p,q}^{\BM}(F)\cong
        \begin{cases}
            \IH_{p,q-1}(F\smallsetminus \{v\})  &\text{if } q \geq n-\perv(\{v\}),\\
            0 &\text{if } q < n-\perv(\{v\}).
        \end{cases}
    \end{align*}
\end{prop}
\begin{proof}
    First, we prove that $\IH_{p,q}^{\BM}(F)\cong \IH_{p,q}(F, F\smallsetminus \{v\})$. 
    Let $(a_r)_{r\in \ZZ_{\geq 0}}$ be a monotone increasing sequence taking values in $[0,1)$ with $a_0=0$ and converging to $1$.  
    We identify $F$ with $cL$ in Definition \ref{def:polyfan}. 
    For any $r,r'\in \Z_{\geq 0}$ with $r>r'$, the inclusion $F\smallsetminus L_{[0,a_r]}\to F\smallsetminus L_{[0,a_{r'}]}$ induces an isomorphism $\IH_{p,q}(F\smallsetminus L_{[0,a_r]})\cong\IH_{p,q}(F\smallsetminus L_{[0,a_{r'}]})$ by Lemma \ref{lem:inclusion}. 
    Therefore, by the similar argument using a diagram as in Proposition \ref{prop:oriexcision}, $\IH_{p,q}(F,F\smallsetminus L_{[0,a_r]})\cong\IH_{p,q}(F,F\smallsetminus L_{[0,a_{r'}]})$. 
    Since $\IH_{p,q}(F,F\smallsetminus L_{[0,a_r]})\cong\IH_{p,q}(F,F\smallsetminus L_{[0,a_{r'}]})$ holds for any $r,r'\in \Z_{\geq 0}$ with $r>r'$, 
    we have an isomorphism $\HH_q(\IC_{p,\bullet}(F,F\smallsetminus L_{[0,a_r]}))\to \varprojlim\nolimits_{r\in \ZZ_{\geq 0} } \IH_{p,q}(F,F\smallsetminus L_{[0,a_r]})$ by \cite[Theorem 3.5.8]{Weibel_1994}. 
    Since any compact subset $K\subset F$ is contained in $L_{[0,a_r]}$ for some $r\in \ZZ_{\geq 0}$, $\IC_{p,q}^{\BM}(F)$ and $\varprojlim_{r\in \ZZ_{\geq 0} } \IC_{p,q}(F,F\smallsetminus L_{[0,a_r]})$ are isomorphic. 
    Therefore, we have an isomorphism $\IH_{p,q}^{\BM}(F)\cong\varprojlim_{r\in \ZZ_{\geq 0} } \IH_{p,q}(F,F\smallsetminus L_{[0,a_r]})$. 
    Applying the isomorphism $\IH_{p,q}(F,F\smallsetminus L_{[0,a_r]})\cong\IH_{p,q}(F,F\smallsetminus L_{[0,a_{r'}]})$, 
    we obtain an isomorphism $\IH_{p,q}^{\BM}(F)\cong \IH_{p,q}(F, F\smallsetminus \{v\})$.

    Associating to a short exact sequence: 
    \[0\to \IC_{p,q}(F\smallsetminus \{v\})\to \IC_{p,q}(F) \to \IC_{p,q}(F,F\smallsetminus \{v\})\to 0, \]
    we obtain the long exact sequence:
    \[
    \begin{tikzcd}[column sep=1.0em, row sep=1.0em]
        \cdots \arrow[r]&
        \IH_{p,q}(F\smallsetminus \{v\}) \arrow[r,"i^q"]&
        \IH_{p,q}(F)\arrow[r]&
        \IH_{p,q}(F,F\smallsetminus \{v\}) \arrow[r]&
        \IH_{p,q-1}(F\smallsetminus \{v\}) \arrow[r,"i^{q-1}"]& \cdots. \\
        &&&\IH_{p,q}^{\BM}(F)\arrow[u,"\cong"]&
    \end{tikzcd}
    \]
    By the cone formula for the non-$\GM$ case (Proposition \ref{prop:conenon}), 
    the map $i^q$ is an isomorphism if $q < n-\perv(\{v\})-1$ and $\IH_{p,q}(F)=0$ if $q \geq n-\perv(\{v\})-1$. 
    Therefore, if $q \leq n-\perv(\{v\})-1$, $\Ker i^{q-1}=0$ and $\Coker i^q=0$. 
    Then, if $q < n-\perv(\{v\})-1$, we obtain an exact sequence from the above long exact sequence.  
    $$0\to \IH_{p,q}^{\BM}(F)\to 0.$$
    Thus $\IH_{p,q}^{\BM}(F)=0$ for $q < n-\perv(\{v\})$. 

    Since $\IH_{p,q}(F)=0$ holds in the case $q \geq n-\perv(\{v\})-1$, we have an exact sequence in the case $q \geq n-\perv(\{v\})$ as follows. 
    $$0\to \IH_{p,q}^{\BM}(F)\to \IH_{p,q-1}(F\smallsetminus \{v\})\to 0.$$
    Therefore, we have $\IH_{p,q}^{\BM}(F)\cong \IH_{p,q-1}(F\smallsetminus \{v\})$ for $q\geq n-\perv(\{v\})$. 
\end{proof}
\begin{corollary}[Another form of the cone formula for non-$\GM$ $\BM$]\label{cor:anotherconebm}
    Let $F$ be an $n$-dimensional rational polyhedral fan with vertex $v$ in $\T^r$. 
    Then the following holds:
    \begin{align*}
        \IH_{p,q}^{\BM}(F)\cong
        \begin{cases}
            \IH_{p,q}^{\BM}(F\smallsetminus \{v\})  &\text{if } q \geq n-\perv(\{v\}),\\
            0 &\text{if } q < n-\perv(\{v\}).
        \end{cases}
    \end{align*}
\end{corollary}
\begin{proof}
    By Proposition \ref{prop:conebm}, it suffices to 
    consider the case $q \geq n-\perv(\{v\})$ and show that $\IH_{p,q}^{\BM}(F\smallsetminus \{v\})\cong\IH_{p,q}^{\BM}(F)$. 
    We identify $F$ with $cL$ in Definition \ref{def:polyfan}. 
    We have that 
    $\IH_{p,q}^{\BM}(F)$$\cong \IH_{p,q}(F,F\smallsetminus L_{[0,\frac{1}{2}]})$ and $\IH_{p,q}^{\BM}(F\smallsetminus \{v\})\cong \IH_{p,q}(F\smallsetminus \{v\}, (F\smallsetminus \{v\})\smallsetminus L_{\{\frac{1}{2}\}})\cong$$\IH_{p,q}(F,F\smallsetminus L_{\{\frac{1}{2}\}})$ by an argument similar to that in Proposition \ref{prop:conebm}. 
    Moreover, there exists an isomorphism: 
    $$\IH_{p,q}(F\smallsetminus L_{\{\frac{1}{2}\}},F\smallsetminus L_{[0,\frac{1}{2}]})\cong \IH_{p,q}(L_{[0,\frac{1}{2})} \sqcup F\smallsetminus L_{[0,\frac{1}{2}]},F\smallsetminus L_{[0,\frac{1}{2}]})\cong \IH_{p,q}(L_{[0,\frac{1}{2})})\cong \IH_{p,q}(F).$$ 
    Then we consider the relative homology long exact sequence:
    \begin{align*}
        \cdots\to\IH_{p,q}(F\smallsetminus L_{\{\frac{1}{2}\}},F\smallsetminus L_{[0,\frac{1}{2}]})\to\IH_{p,q}(F,F\smallsetminus L_{[0,\frac{1}{2}]})\to\IH_{p,q}(F,F\smallsetminus L_{\{\frac{1}{2}\}})\\
        \to\IH_{p,q-1}(F\smallsetminus L_{\{\frac{1}{2}\}},F\smallsetminus L_{[0,\frac{1}{2}]})\to\cdots.
    \end{align*}
    By the above isomorphism, this long exact sequence becomes the long exact sequence: 
    \begin{align*}
        \cdots\to\IH_{p,q}(F)\to\IH_{p,q}^{\BM}(F)\to\IH_{p,q}^{\BM}(F\smallsetminus \{v\})\to\IH_{p,q-1}(F)\to\cdots. 
    \end{align*}
    Since we assume that $q \geq n-\perv(\{v\})$, by Proposition \ref{prop:conenon}, $\IH_{p,q}(F)$ and $\IH_{p,q-1}(F)$ are zero. 
    Therefore, we obtain an isomorphism $\IH_{p,q}^{\BM}(F)\to\IH_{p,q}^{\BM}(F\smallsetminus \{v\})$. 
\end{proof}

\begin{prop}[Cone formula for non-$\GM$ cohomology]\label{prop:conecohom}
    Let $F$ be an $n$-dimensional rational polyhedral fan with vertex $v$ in $\T^r$. 
    Then the following holds:
    \begin{align*}
        \IH^{p,q}(F)\cong
        \begin{cases}
            0 &\text{if } q > n-\perv(\{v\})-1,\\
            \Ext^1(\IH_{p,q-1}(F\smallsetminus \{v\}), Q) &\text{if } q=n-\perv(\{v\})-1,\\
            \IH^{p,q}(F\smallsetminus \{v\}) &\text{if } q < n-\perv(\{v\})-1.
        \end{cases}
    \end{align*}
\end{prop}
\begin{proof}
By the universal coefficient theorem for a chain complex (see \cite[Theorem 3.6.5]{Weibel_1994}), we have the following diagram: 
\begin{equation}\label{eq:univcoef}
\begin{tikzcd}[column sep=1.0em]
0 \arrow[r] & \Ext^1(\IH_{p,q-1}(F), Q) \arrow[r] \arrow[d] & \IH^{p,q}(F) \arrow[r] \arrow[d] & \Hom(\IH_{p,q}(F),Q) \arrow[r] \arrow[d] & 0 \\
0 \arrow[r] & \Ext^1(\IH_{p,q-1}(F\smallsetminus \{v\}), Q) \arrow[r] & \IH^{p,q}(F\smallsetminus \{v\}) \arrow[r] & \Hom(\IH_{p,q}(F\smallsetminus \{v\}),Q) \arrow[r] & 0.
\end{tikzcd}
\end{equation}

If $q > n-\perv(\{v\})-1$, then $\IH_{p,q-1}(F)$ and $\IH_{p,q}(F)$ are zero by the cone formula (Proposition \ref{prop:conenon}). 
It follows that $\IH^{p,q}(F)=0$ by \eqref{eq:univcoef}. 

If $q=n-\perv(\{v\})-1$, then $\IH_{p,q-1}(F)\cong \IH_{p,q-1}(F\smallsetminus \{v\})$ and $\IH_{p,q}(F)=0$ by the cone formula  (Proposition \ref{prop:conenon}). 
It follows that $\IH^{p,q}(F)\cong \Ext^1(\IH_{p,q-1}(F\smallsetminus \{v\}), Q)$ by \eqref{eq:univcoef}. 

If $q < n-\perv(\{v\})-1$, then $\IH_{p,q-1}(F)\cong \IH_{p,q-1}(F\smallsetminus \{v\})$ and $\IH_{p,q}(F)\cong \IH_{p,q}(F\smallsetminus \{v\})$ by the cone formula  (Proposition \ref{prop:conenon}). 
Thus, the two outer vertical arrows in \eqref{eq:univcoef} are isomorphisms. 
By the five lemma, we have $\IH^{p,q}(F)\cong \IH^{p,q}(F\smallsetminus \{v\})$. 
\end{proof}

The next lemma is proved in the same way as in the proof of the cone formulas. We omit details. 
\begin{lemma}\label{lem:gencone}
    Let $(F,\XX_F)$ be an $n$-dimensional filtered rational polyhedral space with an open face structure $\CC_F$, and in the non-$\GM$ case, we assume that the singular locus of $(F,\XX_F)$ is non-empty.
    Further, 
    assume that there exists a $\CC_F$-stratified open face structure $\CC'_F$ such that $(F,\XX^{\CC'_F}_F,\CC'_F)$ is a rational polyhedral fan with vertex $v$ in $\T^r$. 
    Let $U=F\smallsetminus\{v\}$. Let $S$ denote the stratum of $\XX_F$ including $v$. Then the following holds:
    \begin{align*}
        &{}_{\GM}\IH_{p,q}(F,\XX_F,\CC_F)\cong
        \begin{cases}
            0 &\text{if } q \geq \codim(S)-\perv(S)-1\text{ and }q \neq 0,\\
            {\bf F}^{Q}_p(\{v\}) &\text{if } q \geq \codim(S)-\perv(S)\text{ and }q=0,\\
            \sum \limits_{ \substack{{\sigma' \in \CC'_F},\\ \perv(\sigma')\geq \codim(\sigma')}}
                {\bf F}^{Q}_p(\sigma') &\text{if } q =\codim(S)-\perv(S)-1\text{ and }q=0,\\
            _{\GM}\IH_{p,q}(U,\XX_F|_{U},\CC_F|_{U}) &\text{if } q < \codim(S)-\perv(S)-1.
        \end{cases}\\
        &\IH_{p,q}(F,\XX_F,\CC_F)\cong
        \begin{cases}
            0 &\text{if } q \geq \codim(S)-\perv(S)-1,\\
            \IH_{p,q}(U,\XX_F|_{U},\CC_F|_{U}) &\text{if } q < \codim(S)-\perv(S)-1.
        \end{cases}
    \end{align*}
\end{lemma}

\subsection{Sheaf description of tropical intersection chains}\label{subsec:sheaf}

In this subsection, we prove Proposition \ref{prop:introsheaf} in the introduction. 
Let $(X,\XX)$ be a filtered rational polyhedral space with an open face structure $\CC$.  

We define a presheaf ${}_{(\GM)}\Idelta'^{p,-\bullet}_X$ as follows:
\begin{align}\label{eq:sheafdef}
    U \mapsto {}_{(\GM)}\Idelta'^{p,-\bullet}_X(U)\colonequals {}_{(\GM)}\IC_{p,\bullet}(X, Q)/{}_{(\GM)}\IC_{p,\bullet}(X\smallsetminus  \bar{U}, Q)=
    {}_{(\GM)}\IC_{p,\bullet}(X,X\smallsetminus  \bar{U}, Q), 
\end{align}
and define a sheaf ${}_{(\GM)}\Idelta^{p,-\bullet}_X$ to be the sheafification of ${}_{(\GM)}\Idelta'^{p,-\bullet}_X$. 
If the base space $X$ is clear from the context, we may omit $X$ and simply write ${}_{(\GM)}\Idelta'^{p,-\bullet}$ and ${}_{(\GM)}\Idelta^{p,-\bullet}$. 

Let $x\in X$ and $N$ be a distinguished neighborhood of $x$ for $\XX^{\CC}_X$. 
Then by the argument as in Proposition \ref{prop:conebm}, we have $${}_{(\GM)}\IH^{\BM}_{p,q}(N,Q)\cong {}_{(\GM)}\IH_{p,q}(N,N\smallsetminus \{v\},Q).$$ 
Therefore, the cohomology of the stalk of ${}_{(\GM)}\Idelta^{p,\bullet}$ is computed as follows: 
\begin{equation*}\label{eq:stalkofsheaf}
    \HH^{-q}(({}_{(\GM)}\Idelta^{p,\bullet})_x)\cong {}_{(\GM)}\IH_{p,q}(X,X\smallsetminus \{v\},Q)\cong {}_{(\GM)}\IH_{p,q}^{\BM}(N,Q).
\end{equation*}

Recall that we define ${}_{(\GM)}\IC^{lf}_{p,q}$ in the paragraph above Lemma \ref{lem:lf}.  
By replacing ${}_{(\GM)}\IC_{p,\bullet}$ with ${}_{(\GM)}\IC_{p,\bullet}^{lf}$, we construct the sheaf ${}_{(\GM)}\Idelta^{p,\bullet}_{lf}$. 

\begin{prop}\label{prop:sheaf}
    $\Gamma_{c}({}_{(\GM)}\Idelta^{p,-\bullet}) \cong {}_{(\GM)}\IC_{p, \bullet}(X, Q)$.
\end{prop}
\begin{proof}
    There exists the natural morphism ${}_{(\GM)}\Idelta^{p,\bullet}\to {}_{(\GM)}\Idelta^{p,\bullet}_{lf}$ induced by the natural map ${}_{(\GM)}\IC_{p,\bullet}(X,X\smallsetminus  \bar{U}, Q)\to {}_{(\GM)}\IC_{p,\bullet}^{lf}(X,X\smallsetminus  \bar{U}, Q)$. 
    By considering each stalk, we see that this morphism is an isomorphism. 
    Then we obtain $\Gamma_{c}({}_{(\GM)}\Idelta^{p,-\bullet}) \cong \Gamma_{c}({}_{(\GM)}\Idelta^{p,-\bullet}_{lf})$. 
    By \cite[I Theorem 6.2]{Bredon_1997} and \cite[Lemma 3.3]{MR2276609}, we have $\Gamma_{c}({}_{(\GM)}\Idelta'^{p,-\bullet}_{lf})\cong \Gamma_{c}({}_{(\GM)}\Idelta^{p,-\bullet}_{lf})$. 
    The Q-module $\Gamma_{c}({}_{(\GM)}\Idelta'^{p,-\bullet}_{lf})$ is the submodule of ${}_{(\GM)}\IC_{p,\bullet}^{lf}(X,Q)$ consisting of chains $\xi$ of ${}_{(\GM)}\IC_{p,\bullet}^{lf}(X,Q)$ such that the union of the simplices constituting $\xi$ is a compact subset of $X$. 
    It follows that ${}_{(\GM)}\IC_{p, \bullet}(X, Q)\cong \Gamma_{c}({}_{(\GM)}\Idelta'^{p,-\bullet}_{lf})$. 
    Combining the above isomorphisms, we obtain an isomorphism:
    $$\Gamma_{c}({}_{(\GM)}\Idelta^{p,-\bullet}) \cong \Gamma_{c}({}_{(\GM)}\Idelta^{p,-\bullet}_{lf})\cong \Gamma_{c}({}_{(\GM)}\Idelta'^{p,-\bullet}_{lf})\cong {}_{(\GM)}\IC_{p, \bullet}(X, Q),$$
    as desired. 
\end{proof}

For a family of supports $\Phi$, we denote by $\IC_{p,\bullet}^\Phi(X, Q)$ the chain complex $\Gamma_\Phi({}_{(\GM)}\Idelta^{p,-\bullet})$, 
and we denote by $\IH_{p,q}^\Phi(X,Q)$ the homology of $\IC_{p,\bullet}^\Phi(X, Q)$. 
By Proposition \ref{prop:sheaf}, if $\Phi$ consists of all the compact subsets of $X$, 
    then ${}_{(\GM)}\IH_{p,q}^\Phi(X,Q) \cong {}_{(\GM)}\IH_{p,q}(X,Q)$. 
\begin{example}
    By the same arguments as in the proof of Proposition \ref{prop:sheaf}, we have the following. 
    If $\Phi$ consists of all the closed subsets of X, then ${}_{(\GM)}\IH_{p,q}^\Phi(X,Q) \cong $${}_{(\GM)}\IH_{p,q}^{\BM}(X,Q)$. 
\end{example}

\begin{definition}[{\cite[II Exercises 32]{Bredon_1997}}]
    Let $X$ be a topological space. 
    Let $\Q^\bullet$ be a cochain complex of $Q_X$-module. 
    We say that $\Q^{\bullet}$ is \emph{homotopically fine} if $\Q^{\bullet}$ satisfies the following property: 
    for any locally finite open cover $\{U_{\alpha}\}$, there exist endomorphism $h_{\alpha}\in \Hom(\Q^\bullet,\Q^\bullet)$ of degree zero with $|h_{\alpha}|\subset U_{\alpha}$ and 
    such that $h=\sum h_{\alpha}$ is chain homotopic to the identity. 
\end{definition}
In the same way as in \cite[Proposition 3.5]{MR2276609}, ${}_{(\GM)}\Idelta^{p,\bullet}$ is homotopically fine by using the subdivision operator $T$ in the proof of Proposition \ref{prop:excision}. 

The above homology ${}_{(\GM)}\IH_{p,q}^\Phi(X,Q)$ can be considered as a hyper cohomology in some case. 
\begin{proposition}\label{prop:homofine}
    Let $X$ be a CS set which is metrizable as a topological space. 
    Let $\Q^\bullet$ be a cochain complex of $Q_X$-module. Suppose that $\Phi$ is a paracompactifying family, and there exists $\Q'{}^{\bullet}\in \DD^+(Q_X)$ such that $\Q^{\bullet}\cong \Q'{}^{\bullet}$ in $\DD(Q_X)$.    
    If $\Q^\bullet$ is homotopically fine, 
    then $\Hyper^\bullet_\Phi(\Q^\bullet)\cong\HH^\bullet(\Gamma_\Phi(\Q^\bullet))$ holds. 
\end{proposition}
\begin{proof}
    By the same argument as in \cite[p. 2013]{MR2276609}, and taking an appropriate truncation of $\Q^\bullet$, which we denote by $\mathcal{P}^\bullet$, we have the natural quasi-isomorphism $\Q^\bullet\to \mathcal{P}^\bullet$ and $\mathcal{P}^\bullet$ is bounded below. 
    Then an injective resolution $\mathcal{I}^\bullet$ of $\mathcal{P}^\bullet$ is also a (homotopically) injective resolution of $\Q^\bullet$. 
    Since $\Q^\bullet$ is homotopically fine, by \cite[II Exercises 32 and Proposition 16.11]{Bredon_1997}, we can apply \cite[IV Theorem 2.2]{Bredon_1997} to $\Q^\bullet$ and $\mathcal{I}^\bullet$. 
    Then the map $\Q^\bullet\to \mathcal{I}^\bullet$ induces the isomorphism $\HH^\bullet(\Gamma_\Phi(\Q^\bullet))\cong \HH^\bullet(\Gamma_\Phi(\mathcal{I}^\bullet))$. 
    Since  $\HH^\bullet(\Gamma_\Phi(\mathcal{I}^\bullet))\cong \Hyper^\bullet_\Phi(\Q^\bullet)$, we obtain the statement.  
\end{proof}
Since a rational polyhedral space $X$ is paracompact, second countable, and Hausdorff, $X$ is a metrizable. 
Therefore, if ${}_{(\GM)}\Idelta^{p,\bullet}$ is quasi-isomorphic to a sheaf complex in $\DD^+(Q_X)$, 
then there exists an isomorphism ${}_{(\GM)}\IH_{p,q}^\Phi(X,Q)\cong \Hyper^{-q}_{\Phi}({}_{(\GM)}\Idelta^{p,\bullet})$ for any paracompactifying family of supports $\Phi$ by Proposition \ref{prop:homofine}. 
Note that this condition is satisfied in the setting considered in \textsection\ref{sec:non-GMmfd}. 

For later arguments, we introduce a map ${}_{(\GM)}\IH_{p,q}^{\BM}(X,Q)\to {}_{(\GM)}\IH_{p,q}^{\BM}(V,Q)$ induced by the inclusion $V\hookrightarrow X$. 
Let $V\subset X$ be an open subset of a rational polyhedral space $X$. 
As in \cite[Proposition 3.7]{MR2276609}, we can construct a quasi-isomorphism $r\colon {}_{(\GM)}\Idelta^{p,\bullet}_X|_V \to {}_{(\GM)}\Idelta^{p,\bullet}_V$  
such that the map ${}_{(\GM)}\Idelta^{p,\bullet}_V\to {}_{(\GM)}\Idelta^{p,\bullet}_X|_V$ induced by the natural map ${}_{(\GM)}\IC_{p,\bullet}(V,V\smallsetminus \bar{U},Q)\to {}_{(\GM)}\IC_{p,\bullet}(X,X\smallsetminus \bar{U},Q)$ is quasi-inverse of $r$. 
Then $r$ induces $\Gamma({}_{(\GM)}\Idelta^{p,\bullet}_X)\to \Gamma({}_{(\GM)}\Idelta^{p,\bullet}_V)$, and we obtain a map ${}_{(\GM)}\IH_{p,q}^{\BM}(X,Q)\to {}_{(\GM)}\IH_{p,q}^{\BM}(V,Q)$.

This map ${}_{(\GM)}\IH_{p,q}^{\BM}(X,Q)\to {}_{(\GM)}\IH_{p,q}^{\BM}(V,Q)$  coincides with the composition of the map of homology groups induced by $\varprojlim_{K\subset X}{}_{(\GM)}\IC_{p,\bullet}(X,X\smallsetminus K,Q) \to \varprojlim_{K\subset V}{}_{(\GM)}\IC_{p,\bullet}(X,X\smallsetminus K,Q)$ and the inverse map of the map of homology groups induced by the quasi-isomorphism $ \varprojlim_{K\subset V}{}_{(\GM)}\IC_{p,\bullet}(V,V\smallsetminus K,Q)\to  \varprojlim_{K\subset V}{}_{(\GM)}\IC_{p,\bullet}(X,X\smallsetminus K,Q)$, where $K$ is a compact subset of $V$. 

In particular, if there exists a compact subset $L\subset V$ such that the canonical projections of the projective limits of chain complexes $\varprojlim_{K\subset X}{}_{(\GM)}\IC_{p,\bullet}(X,X\smallsetminus K,Q)\to {}_{(\GM)}\IC_{p,\bullet}(X,X\smallsetminus L,Q)$ and $\varprojlim_{K\subset V}{}_{(\GM)}\IC_{p,\bullet}(V,V\smallsetminus K,Q) \to {}_{(\GM)}\IC_{p,\bullet}(V,V\smallsetminus L,Q)$ are quasi-isomorphisms, 
then the inclusion $V\subset X$ induces an isomorphism ${}_{(\GM)}\IH_{p,q}^{\BM}(X,Q)\cong {}_{(\GM)}\IH_{p,q}^{\BM}(V,Q)$. 

In order to apply the results of \cite{Gross_2023} to sheaves in \textsection\ref{sec:non-GMmfd}, we introduce the sheaf of tropical $p$-form. 
\begin{defn}[{\cite[Definition 2.7]{Gross_2023}}]\label{def:pformsheaf}
    We consider the filtration $\XX_X^{\text{trop}}$ on a rational polyhedral space $X$. 
    Let $\iota\colon X_{\dim((X,\XX_X^{\text{trop}}))} \hookrightarrow X$  denote the inclusion. 
    We define the sheaf of graded rings $\Omega^{\bullet}_{X,\ZZ}$ as the image of $\bigwedge^{\bullet} 
    \Omega_{X,\ZZ}^1$ in $\iota_*\left(\bigwedge^{\bullet}\Omega^1_{X,\ZZ}|_{X_{\dim((X,\XX_X^{\text{trop}}))}}\right)$. 
    We define $\Omega_{X}^{\bullet}\colonequals \Omega_{X,\ZZ}^{\bullet}\otimes_{\ZZ_X} Q_X$. 
    (Recall that $Q_X$ is the constant sheaf associated to $Q$.)
    Sections of $\Omega^p_X$ are called \emph{tropical $p$-forms}.
\end{defn}
For a $Q$-module $A$, we denote by $A_X$ the constant sheaf associated with $A$ on $X$. 
For a compact subset $K\subset X$ and a $Q$-module $A$, 
we define $Q_X$-module $A_K$ to be $i_*i^*A_X$, where $i\colon K \to X$ is the inclusion. 
For $\CC$-stratified simplex $\sigma$, we denote by ${\bf F}_p^{\sigma}$ the multi-tangent at the smallest element in  $\CC$ including $\sigma$. 
Let $\Delta^{p,\bullet}$ be the sheafification of the presheaf constructed by replacing ${}_{(\GM)}\IC_{p,\bullet}$ with ${}_{\GM}\HC_{p,\bullet}$ in \eqref{eq:sheafdef}. 
By the same argument as in   \cite[Theorem 4.20]{Gross_2023}, we obtain $\Delta^{p,\bullet}\cong \bigoplus_{\sigma}({\bf F}_p^{\sigma})_{\Im(\sigma)}$. 
The sheaf $\bigoplus_{\sigma}({\bf F}_p^{\sigma})_{\Im(\sigma)}$ is isomorphic to $\Homs(\Omega^p_X, \Delta^{0,\bullet})$ by \cite[Theorem 4.20]{Gross_2023}. 
Then $\Delta^{p,\bullet}\cong \Homs(\Omega^p_X, \Delta^{0,\bullet})$ holds. 

\subsection{Other properties}\label{subsec:otherprop}
In this subsection, we prove Proposition \ref{prop:introindep} and Proposition \ref{prop:introtropstr} in the introduction. 
Let $(X,\XX)$ be a filtered rational polyhedral space with an open face structure $\CC$. 
Let $\perv$ be a perversity on $(X,\XX)$. 
Since $(X,\XX_X^{\CC})$ is a CS set (see Proposition \ref{prop:tropfaceCS}), any point has a neighborhood $N$ that is stratified-homeomorphic to $\R^i\times cL$ by the definition \ref{def:CS}. 
In what follows, we introduce homology on $cL$. 

By restricting $\CC$ to $N$, $N$ has the open face structure $\CC|_N$. 
The open face structure $\CC|_N$ induces a complex structure on $cL$. 
We denote by $\CC_{cL}$ the induced open face structure on $cL$. 
Let $\pi\colon \R^i\times cL\to cL$ be the projection.  
Then the complex $\CC_{cL}$ consists of $\pi(\sigma)$ for $\sigma \in \CC_N$. 
The topological space $cL$ has the natural filtration $\XX_{cL}$ induced by $\XX$ such that $\R^i\times (cL)^k=(\R^i\times cL)^k$ for any $k\in \ZZ_{\geq 0}$. 
The perversity $\perv$ induces the perversity on $(cL,\XX_{cL})$ given by $S\mapsto \perv(\pi^{-1}(S))$.  
Using the filtration $\XX_{cL}$, the complex $\CC_{cL}$, the multi-tangent $\mathbf{F}^Q_{p}(\pi^{-1}(\sigma))$ at $\sigma\in \CC_{cL}$, and the perversity induced by $\perv$, we define a chain group ${}_{(\GM)}\IC'_{p,\bullet}(cL,Q)$ as in \textsection\ref{subsec:tropinthom}. 

We write a $q$-simplex $\sigma\colon \Delta_q\to N$ of $N\cong \R^i\times cL$ as $\sigma(x)=(\sigma_{\R^i}(x),\sigma_{cL}(x))\in \R^i\times cL$. 
Let $\zeta\in {}_{(\GM)}\IC'_{p,q}(cL,Q)$ be a $q$-chain of $cL$. 
We denote by $0$ the $q$-simplex $\sigma\colon \Delta_q\to \R^i$ such that $\sigma(x)\equiv 0$. 
Then we have the the natural map $$\imath\colon {}_{(\GM)}\IC'_{p,q}(cL,Q)\to {}_{(\GM)}\IC_{p,q}(N,Q),\quad \zeta=\sum f_{\tau}\tau \mapsto \sum f_{\tau}(0,\tau).$$ 

\begin{lemma}\label{lem:projiso}
    The natural map $$\imath\colon {}_{(\GM)}\IC'_{p,\bullet}(cL,Q)\to {}_{(\GM)}\IC_{p,\bullet}(N,Q)$$ is a quasi-isomorphism. 
\end{lemma}
\begin{proof}
    We define $$\jmat\colon {}_{(\GM)}\IC_{p,q}(N,Q)\to {}_{(\GM)}\IC'_{p,q}(cL,Q),\quad \sigma=(\sigma_{\R^i},\sigma_{cL})\mapsto \sigma_{cL}.$$ 
    Then $\jmat\circ \imath$ is the identity and $\imath\circ \jmat$ is the map induced by $(\sigma_{\R^i},\sigma_{cL})\mapsto (0,\sigma_{cL})$.   
    By the argument in \cite[Propositions 4.1.10 and 6.3.7]{Friedman_2020}, there exists a homomorphism $P\colon {}_{(\GM)}\IC_{p,q}(N,Q)$$\to {}_{(\GM)}\IC_{p,q+1}(N,Q)$, called the prism operator  
    such that $\xi -\imath\circ \jmat(\xi)= \partial P(\xi)+P(\partial \xi)$ for any chain $\xi\in  {}_{(\GM)}\IC_{p,q}(N,Q)$ and the boundary map $\partial$. 
    Therefore, the map $\imath$ is a chain homotopy equivalence. 
    In particular, $\imath$ is a quasi-isomorphism. 
\end{proof}
We define ${}_{(\GM)}\IH'_{p,q}(cL,\CC_{cL}, Q)\colonequals \HH_{q}({}_{(\GM)}\IC'_{p,\bullet}(cL,Q))$. 
By Lemma \ref{lem:projiso}, there exists the isomorphism ${}_{(\GM)}\IH'_{p,q}(cL,\CC_{cL},Q)$$\cong {}_{(\GM)}\IH_{p,q}(N,\CC|_N,Q)$. 
\begin{rem}\label{rem:BM'}
    For any open subset of $U\subset cL$, we can define ${}_{(\GM)}\IH'_{p,q}(U,\CC_{cL}|_U,Q)$. 
    For any open subset $U\subset cL$, 
    we can also define ${}_{(\GM)}{\IC^{\BM}_{p,q}}'(U,Q)$ and ${}_{(\GM)}{\IH^{\BM}_{p,q}}'(U,Q)$ similarly. 
    We use these in the proof of Lemma \ref{lem:prodgenconeBM}. 
\end{rem}

Let $\CC'$ be a $\CC$-stratified open face structure on $X$.  
As in the above discussion, we choose a neighborhood $N'\cong \R^{i'}\times cL'$ with respect to $\XX_X^{\CC'}$ and consider ${}_{(\GM)}\IH'_{p,q}(cL',\CC'_{cL'},Q)$ for the filtration $\XX_{cL'}$ on $cL'$ induced by $\XX$. 
The open face structure $\CC$ induces a complex $\CC_{cL'}$ on $cL'$. Using the complex $\CC_{cL'}$, we also define ${}_{(\GM)}\IH'_{p,q}(cL',\CC_{cL'},Q)$. 
In this setting, we obtain the following statement as in the proofs of the cone formulas. 
\begin{lemma}\label{lem:prodgencone}
    Let $S$ be the stratum of $\XX_{cL'}$ including the smallest cell $v$ of $\CC'_{cL'}$. 
    Let $U=cL'\smallsetminus\{v\}$. In the non-$\GM$ case, we assume that the singular locus of $(cL',\XX_{cL'})$ is non-empty.
    Under the above setting, we have the following: 
    \begin{align*}
        &{}_{\GM}\IH'_{p,q}(cL',\CC_{cL'},Q)\cong
        \begin{cases}
            0 &\text{if } q \geq \codim(S)-\perv(S)-1,q \neq 0,\\
            {\bf F}^{Q}_p(\{v\}) &\text{if } q \geq \codim(S)-\perv(S),q=0,\\
            \sum \limits_{ \substack{\sigma' \in \CC'_{cL'}: \{v\} \prec \sigma',\\ \perv(\sigma')\geq \codim(\sigma')}}
                {\bf F}^{Q}_p(\sigma') &\text{if } q =\codim(S)-\perv(S)-1,q=0,\\
            _{\GM}\IH'_{p,q}(U,\CC_{cL'}|_{U},Q) &\text{if } q < \codim(S)-\perv(S)-1.
        \end{cases}\\
        &\IH'_{p,q}(cL',\CC_{cL'},Q)\cong
        \begin{cases}
            0 &\text{if } q \geq \codim(S)-\perv(S)-1,\\
            \IH'_{p,q}(U,\CC_{cL'}|_{U},Q) &\text{if } q < \codim(S)-\perv(S)-1.
        \end{cases}
    \end{align*}
\end{lemma}

For open face structures $\CC$ and $\CC'$ on $X$, we consider $\CC\cap \CC'$ in the paragraph below Definition \ref{def:facestructure}. 
By replacing $\CC$ in the construction of ${}_{(\GM)}\IH_{p,q}(X,\CC,Q)$ and ${}_{(\GM)}\IH^{p,q}(X,\CC,Q)$ with $\CC\cap \CC'$, 
we define ${}_{(\GM)}\IH_{p,q}(X,\CC\cap \CC',Q)$ and ${}_{(\GM)}\IH^{p,q}(X,\CC\cap \CC',Q)$, respectively. 
Note that there exists natural maps ${}_{(\GM)}\IH_{p,q}(X,\CC\cap \CC',Q)\to {}_{(\GM)}\IH_{p,q}(X,\CC,Q)$ and ${}_{(\GM)}\IH_{p,q}(X,\CC\cap \CC',Q)\to {}_{(\GM)}\IH_{p,q}(X,\CC',Q)$ induced by the inclusions of the chain complexes. 
Then the statement in \textsection\ref{sec:property} remains valid if the open face structures are replaced by $\CC\cap \CC'$. 
In particular, Proposition \ref{prop:Mayer} and Lemmas \ref{lem:projiso} and \ref{lem:prodgencone} is valid for $\CC\cap\CC'$.  
Using these results, we prove the following. 
\begin{proposition}\label{prop:indep}
    The homology $_{(\GM)}\IH_{p,q}(X,\XX,Q)$ is independent of the choice of an open face structure on $(X,\XX)$. 
\end{proposition}
\begin{proof}
    We prove the non-$\GM$ case. 
    We consider the filtration $\XX^{\CC\cap \CC'}_X$ induced by $\CC\cap \CC'$. 
    We set $F_{q}(U)\colonequals\IH_{p,q}(U,\XX|_U,\CC\cap\CC'|_U,Q)$ and $G_{q}(U)\colonequals\IH_{p,q}(U,\XX|_U,\CC|_U,Q)$ and verify conditions \ref{itm:mv1}--\ref{itm:mv4} in Pro\-position \ref{prop:argument} for $(X,\XX_X^{\CC\cap \CC'})$.  
    The conditions \ref{itm:mv2} and \ref{itm:mv4} are apparent. 
    By Proposition \ref{prop:Mayer}, the condition \ref{itm:mv1} is satisfied. 
    In the following, we prove the condition \ref{itm:mv3} for each $i$ in $\R^i\times cL$.  
    We need to prove that if $F_q(\R^i\times cL\smallsetminus \{v\})\to G_q(\R^i\times cL\smallsetminus\{v\})$ is an isomorphism, 
    then $F_q(\R^i\times cL)\to G_q(\R^i\times cL)$ is also an isomorphism. 
    This statement follows by Lemmas \ref{lem:projiso} and \ref{lem:prodgencone}. 
    Then the condition \ref{itm:mv3} is satisfied. 
    By Proposition \ref{prop:argument}, the natural map $\IH_{p,q}(X,\XX,\CC\cap\CC',Q)\to \IH_{p,q}(X,\XX,\CC,Q)$ is an isomorphism. 
    In the same way, the natural map $\IH_{p,q}(X,\XX,\CC\cap\CC',Q)\to \IH_{p,q}(X,\XX,\CC',Q)$ is an isomorphism. 
    Then we obtain an  isomorphism between $\IH_{p,q}(X,\XX,\CC',Q)$ and $\IH_{p,q}(X,\XX,\CC',Q)$. 
    Therefore, we obtain the statement in the non-$\GM$ case. 
    The proof for the $\GM$ is shown similarly. 
\end{proof}
\begin{corollary}
    The cohomology $_{(\GM)}\IH^{p,q}(X,\XX,Q)$ is independent of an open face structure on $(X,\XX)$.
\end{corollary}
\begin{proof}
    We prove the non-$\GM$ case. 
    Let $\CC$ and $\CC'$ be an open face structure on $X$. 
    By the same argument as in Proposition \ref{prop:indep}, it suffices to show an isomorphism between the cohomology groups associated with $\CC$ and $\CC\cap \CC'$. 
    By Proposition \ref{prop:indep}, the natural map $\IC_{p,q}(X,\XX,\CC\cap \CC',Q)\to \IC_{p,q}(X,\XX,\CC,Q)$ is a quasi-isomorphism. Since our complexes are bounded below and consist of projective modules, this morphism is a chain homotopy equivalence by \cite[Lemma A.4.3]{Friedman_2020}. 
    Therefore, the dual map $\IC^{p,q}(X,\XX,\CC,Q) \to \IC^{p,q}(X,\XX,\CC\cap \CC',Q)$ is a chain homotopy equivalence, and we obtain an isomorphism between cohomology groups. 
    The proof for the $\GM$ case is shown similarly. 
\end{proof}
The following claim follows from the preceding discussion. 
\begin{prop}\label{prop:tropstr}
    Let $\mathcal{F}$ be a tropical filtration (see Definition \ref{def:tropfilter}) and let $X$ and $Y$ be rational polyhedral spaces. 
    Let $\mathcal{F}(X)$ and $\mathcal{F}(Y)$ be filtrations on $X$ and $Y$ determined by $\mathcal{F}$, respectively.  
    Let perversities $\perv$ and $\perq$ be perversities on $(X,\mathcal{F}(X))$ and $(Y,\mathcal{F}(Y))$, respectively.
    We assume that $X$ has an $\mathcal{F}(X)$-stratified open face structure and that $Y$ has an $\mathcal{F}(Y)$-stratified open face structure. 
    If there exists an isomorphism $f:X\to Y$ between rational polyhedral spaces such that $\perv(S)=\perq(f(S))$ for all strata $S$ of $(X,\mathcal{F}(X))$, then the following hold. 
    \begin{align*}
        _{(\GM)}\IH_{p,q}(X,\mathcal{F}(X),Q)\cong {}_{(\GM)}\IHq_{p,q}(Y,\mathcal{F}(Y),Q),\\ {}_{(\GM)}\IH^{p,q}(X,\mathcal{F}(X),Q)\cong {}_{(\GM)}\IHq^{p,q}(Y,\mathcal{F}(Y),Q).
    \end{align*}
\end{prop} 
\begin{proof}
    We prove the isomorphism for the homology. 
    Let $\CC_X$ be an $\mathcal{F}(X)$-stratified open face structure on $X$. 
    Then the isomorphism between $X$ and $Y$ induces the open face structure $\CC_Y'$ of $Y$ corresponding to $\CC_X$. 
    We have an isomorphism between ${}_{(\GM)}\IH_{p,q}(X,\mathcal{F}(X),\CC_X,Q)$ and ${}_{(\GM)}\IHq_{p,q}(Y,\mathcal{F}(Y),\CC_Y',Q)$. 
    We also have an isomorphism between ${}_{(\GM)}\IHq_{p,q}(Y,\mathcal{F}(Y),\CC_Y',Q)$ and ${}_{(\GM)}\IHq_{p,q}(Y,\mathcal{F}(Y),\CC_Y,Q)$ by Proposition \ref{prop:indep}. 
    Therefore,  we obtain an isomorphism $${}_{(\GM)}\IH_{p,q}(X,\mathcal{F}(X),\CC_X,Q)\cong {}_{(\GM)}\IHq_{p,q}(Y,\mathcal{F}(Y),\CC_Y,Q).$$ 
    The proof for the cohomology is similar. 
\end{proof}

\section{Non-GM chain sheaf for a tropical manifold with singularities}\label{sec:non-GMmfd} 
\subsection{Construction of tropical Deligne sheaf}\label{subsec:deligne}
In this section, following the ideas of the generalized Deligne sheaf in \cite{Friedman_2010}, we develop the sheaf theory of 
the tropical intersection homology for filtrations using the notion of tropical manifold.  
Since our main approach follows \cite{Friedman_2010}, the details are omitted in this section. 
We assume that the coefficient ring $Q$ is $\ZZ$ or a field with $\QQ\subset Q\subset\RR$. 
This assumption is to use results in \cite{Gross_2023}. 

Let $(X,\XX)$ be an $n$-dimensional filtered rational polyhedral space: 
\begin{equation*}
    \XX \colon X = X^{n} \supset X^{n-1} \supset \cdots \supset X^0 \supset X^{-1} = \emptyset. 
\end{equation*}
In the following, unless otherwise specified, we use the notation $X^i$ and $X_i=X^i\smallsetminus X^{i-1}$ for $\XX$. 
We use the notation $X^i_{\CC}$ and $X^{\CC}_i$ for $\XX^{\CC}_X$ if $\CC$ is an open face structure on $X$.  
For another filtration $\XX'$, we denote them by $X^i_{\XX'}$ and $X^{\XX'}_i$, respectively. 
We define $U_k\colonequals X\smallsetminus X^{n-k}$. 
Since $X^{n-k-1}\supset X^{n-k}$, we have the natural inclusion $i_k \colon U_k \hookrightarrow U_{k+1}$.

We define a tropical manifold with singularities as follows. 
\begin{defn}[Tropical manifold with singularities]\label{def:tms}
    Let $X$ be a rational polyhedral space. 
    If $U\subset X$ is an open dense subset and $U$ is an $n$-dimensional tropical manifold (see Definition \ref{def:polyhedralspace}), 
    then we say that $(X,U)$ is an \emph{$n$-dimensional tropical manifold with singularities}. 
\end{defn}
If $X$ is a pure $n$-dimensional rational polyhedral space, then $(X,X_n^{\XX^{\text{trop}}_X})$ is a tropical manifold with singularities. 
Therefore, under such an assumption, we can choose an $n$-dimensional tropical manifold $U$ such that $(X,U)$ is an $n$-dimensional tropical manifold with singularities. 

Let $(X,U)$ be an $n$-dimensional tropical manifold with singularities. 
For an open subset $V$ of $X$, let $\CC_V$ be an open face structure on $V$ such that $U\cap V$ is the union of some strata of $\XX^{\CC_V}_V$.   
We set $V^n=V$ and $V^{-1}=\emptyset$, and for any $0\leq k \leq n-1$, we set 
$V^k\colonequals (X\smallsetminus U)\cap  V^k_{\XX^{\CC_V}_V}$. 
These make a filtration on $V$, which we denote by $\mathfrak{V}^{U,\CC_V}_V$: 
\begin{equation}\label{eq:newfilter}
    \mathfrak{V}^{U,\CC_V}_V \colon V=V^n \supset V^{n-1} \supset \cdots \supset V^0 \supset V^{-1}=\emptyset, 
\end{equation}
If $V=X$ and $\CC_V=\CC$, then we denote the above filtration by $\XX^{U,\CC}_X$. 

\begin{defn}[Condition $(C)$ and $(C)'$]\label{def:semilocgoodratpolyspace}
    Let $(X,U)$ be an $n$-dimensional tropical manifold with singularities. 
    Let $\XX$ be a filtration of formal dimension $n$ on $X$ such that $X^{n-1}=X\smallsetminus U$ (see Definition \ref{def:filteredspace}). 
    \begin{enumerate}
        \item We say that $\XX$ satisfies Condition $(C)$ if there exists a family of pairs $(V,\CC_{V})$ of an open subset $V$ of $X$ and an open face structure $\CC_V$ on $V$ 
            such that $X=\bigcup V$, $U\cap V$ is the union of some strata of $\XX^{\CC_V}_V$, and $\XX|_V=\mathfrak{V}^{U,\CC_V}_V$. 
        \item We say that $\XX$ satisfies Condition $(C)'$ if there exists an open face structure $\CC$ on $X$ 
            such that $U$ is the union of some strata of $\XX^{\CC}_X$, and $\XX=\XX^{U,\CC}_X$. 
    \end{enumerate}
\end{defn}
\begin{example}\label{ex:CX}
    Let $(X,U)$ be an $n$-dimensional tropical manifold with singularities. 
    Let $\CC$ be a face structure on $X$ such that $U$ is the union of some strata of $\XX^{\CC}_X$. 
    By the definition, $\XX^{U,\CC}_X$ satisfies $(C)'$. 
    In particular, $\XX^{U,\CC}_X$ satisfies $(C)$. 
\end{example}

To construct a tropical Deligne sheaf, we introduce the generalized truncation in \cite{Friedman_2010}. 
\begin{defn}[{Generalized truncation {\cite[Definition 3.2]{Friedman_2010}}}]\label{def:gentrun}
    Let $\mathcal{A}^{\bullet}$ be a sheaf complex  on $X$, $\mathcal{F}$ be a locally-finite collection  of subsets of $X$, and $\perv$ be a function $\colon \mathcal{F}\to \ZZ$.  
    Set $|\mathcal{F}| =\bigcup_{S\in \mathcal{F}}S$. 
    We define a presheaf $T^{\mathcal{F}}_{\leq \perv}\mathcal{A}^{\bullet}$ on $X$ as follows:
    \begin{align*}
        T^{\mathcal{F}}_{\leq \perv}\mathcal{A}^{\bullet}(U)\colonequals
        \begin{cases}
            \Gamma(U,\mathcal{A}^{\bullet}), & U\cap |\mathcal{F}|=\emptyset,\\
            \Gamma(U,\tau_{\leq \inf\{\perv(S)\mid S\in \mathcal{F},U\cap S\neq \emptyset\}}\mathcal{A}^{\bullet}),& U\cap |\mathcal{F}|\neq \emptyset,
        \end{cases}
    \end{align*}
    where $\tau$ is the usual truncation. 
    We define the \emph{generalized truncation sheaf} $\tau^{\mathcal{F}}_{\leq \perv}\mathcal{A}^{\bullet}$ to be the sheafification of $T^{\mathcal{F}}_{\leq \perv}\mathcal{A}^{\bullet}$.  
\end{defn}
\begin{defn}(Shifted complex)
Let $\mathcal A^{\bullet}[n]$ denote the shifted complex defined by $(\mathcal A^{\bullet}[n])^i=\mathcal A^{i+n}$. 
\end{defn}

Let $\XX$ be a filtration 
such that $(X,X_n)$ is a tropical manifold with singularities. Let $\perv$ be a perversity on $(X,\XX)$. 
Recall that $\Omega^p_{X_n}$ is the sheaf of tropical $p$-form on $X_n$ (see Definition \ref{def:pformsheaf}). 
Allowing $X_k$ to stand for the strata of $\XX$ in $X_k$ and using the generalized truncation,
we define a tropical Deligne sheaf  $\Q^{\bullet}_{\perv,p}$ as follows: 
\begin{equation}\label{eq:Delignesheaf}
    \Q^{\bullet}_{\perv,p}\colonequals \tau^{X_0}_{\leq \perv}R{i_n}_*\ldots \tau^{X_{n-1}}_{\leq \perv}R{i_1}_*\Omega^{n-p}_{X_n}.
\end{equation}
Here, $R{i_k}_*$ is a functor $\DD(Q_{U_k})\to \DD(Q_{U_{k+1}})$ and $\tau^{X_{n-k}}_{\leq \perv}$ is a functor $\DD(Q_{U_{k+1}})\to \DD(Q_{U_{k+1}})$.

For $i_k$ and a sheaf complex $\mathcal{A}^{\bullet}$ on $X$, we have the natural map $\mathcal{A}^{\bullet}|_{U_{k+1}}\to {i_k}_*{i_k}^* \mathcal{A}^{\bullet}|_{U_{k+1}}$  induced by the adjunction ${i_k}^* \dashv {i_k}_*$. 
We define the \emph{attachment map} $\alpha_k\colon  \mathcal{A}^{\bullet}|_{U_{k+1}}\to R{i_k}_*\mathcal{A}^{\bullet}|_{U_{k}}$ to be 
the composition of the  natural maps $\mathcal{A}^{\bullet}|_{U_{k+1}}\to {i_k}_*{i_k}^* \mathcal{A}^{\bullet}|_{U_{k+1}}={i_k}_*\mathcal{A}^{\bullet}|_{U_{k}}$ and ${i_k}_*\mathcal{A}^{\bullet}|_{U_{k}}\to R{i_k}_*\mathcal{A}^{\bullet}|_{U_{k}}$. 

\begin{defn}\label{def:ax2}
    Let $(X,\XX)$ be an $n$-dimensional filtered rational polyhedral space. 
    Suppose that $(X,X_n)$ is an $n$-dimensional tropical manifold with singularities. 
    Let $\perv$ be a perversity  on $(X,\XX)$ and let $p\in\ZZ$. 
    For a sheaf $\mathcal{A}^\bullet$ on $X$, we consider the condition $(AX)^{p}_{\perv,\XX}$:
    \begin{enumerate}[label=(AX\arabic*)]
        \item\label{itm:ax1} $\mathcal{A}^\bullet$ is isomorphic to some bounded sheaf in $\DD(Q_X)$, $\HH^i(\mathcal{A}^{\bullet})=0$ for $i<0$, and $\mathcal{A}^{\bullet}|_{X_n}\cong \Omega^{n-p}_{X_n}$ in $\DD(Q_{X_n})$. 
        \item\label{itm:ax2} For any stratum $S$ of $\XX$ and $x\in S$, $\HH^i(\mathcal{A}^{\bullet}_x)=0$ if $i>\perv(S)$.  
        \item\label{itm:ax3} For any stratum $S\subset X_k$ of $\XX$ and $x\in S$, the attachment map $\alpha_k\colon \mathcal{A}^{\bullet}|_{U_{k+1}}\to R{i_k}_*\mathcal{A}^{\bullet}|_{U_k}$ 
        induces an isomorphism $\HH^i((\mathcal{A}^{\bullet}|_{U_{k+1}})_x)\cong \HH^i((R{i_k}_*\mathcal{A}^{\bullet}|_{U_k})_x)$ if $\perv(S)\geq i$. 
    \end{enumerate}
\end{defn}
\begin{rem}\label{rem:delignetheory}
    Here, AX stands for axiom. 
    In the following, we will verify that $\Idelta^{p,\bullet}_X$ satisfies the above axiom up to shift. 
    Moreover, we will verify that a sheaf satisfying the above axiom is unique up to isomorphism in the derived category. 
    Therefore, we can say that a sheaf $\mathcal A^{\bullet}$ satisfies the above axiom if and only if $\mathcal A^{\bullet}\cong \Idelta^{p,\bullet}_X$ in the derived category. 
    This overall line of argument is the same as in classical Deligne sheaf theory. 
\end{rem}

Let $(X,\XX)$ be an $n$-dimensional filtered rational polyhedral space with an open face structure $\CC$. 
Set $B^i_{\epsilon}=(-\epsilon,\epsilon)^i\subset \R^i$ for $1>\epsilon>0$. 
Let $x\in X$ and $N$ be a distinguished neighborhood of $x$ for $\XX^{\CC}_X$ such that $N\cong \R^i\times cL$ and $x$ corresponds to $0\times v\in \R^i\times cL$ by this identification. 
Then we define $N_{\epsilon}\colonequals B^i_{\epsilon}\times L_{[0,\epsilon)}\subset N$ and $N'_{\epsilon}\colonequals  B^i_{\epsilon}\times L_{[0,\epsilon)}\smallsetminus \{v\}\subset N$. 
Let $S$ be the stratum containing $x$, and let $k$ be the dimension of $S$. 
In this setting, by \cite[V.1.7 (2)]{Borel_1984}, the induced map $\HH^i((\mathcal{A}^{\bullet}|_{U_{k+1}})_x)\to \HH^i((R{i_k}_*\mathcal{A}^{\bullet}|_{U_k})_x)$ in the condition \ref{itm:ax3}  is an isomorphism  if and only if
\begin{equation}\label{eq:3cond1}
    \HH^{\bullet}(\mathcal{A}_x)\to \varinjlim_{\epsilon>0} \Hyper^{\bullet}(\mathcal{A}|_{N'_{\epsilon}})
\end{equation} 
induced by the attachment map $\alpha_k$ is an isomorphism (for the definition of $\Hyper^{\bullet}$, see the paragraph above Remark \ref{rem:Verdier}). 
This condition is also equivalent to that
\begin{equation}\label{eq:3cond2}
    \varinjlim_{\epsilon>0}\Hyper^{\bullet}(\mathcal{A}|_{N_{\epsilon}})\to \varinjlim_{\epsilon>0} \Hyper^{\bullet}(\mathcal{A}|_{N'_{\epsilon}})
\end{equation}
induced by the attachment map $\alpha_k$ is an isomorphism. 

To prove that $\Idelta^{p,\bullet}_X[-n]$ for the filtration in this section satisfies the axiom $(AX)^{p}_{\perv,\XX}$, we use the following claims. 

\begin{lemma}\label{lem:prodgenconeBM}
    Let $(N,\XX_N)$ be an $n$-dimensional filtered rational polyhedral space with an open face structure $\CC_N$. 
    Let $k\in\ZZ$ be $0\leq k \leq n$. 
    Suppose that $\XX_N$ has singular strata, $\XX_N$ satisfies $(C)'$, and $(N,\XX^{\CC_N}_N)$ is stratified-homeomorphic to $\R^{n-k}\times cL$ for some compact $(k-1)$-dimensional filtered space $L$.  
    Then the following holds:
    \begin{align*}
        \IH_{p,q}^{\BM}(N,Q)\cong
        \begin{cases}
            \IH_{p,q}^{\BM}(\R^{n-k}\times(cL\smallsetminus \{v\}),Q)  &\text{if } q \geq n-\perv(\R^{n-k}\times\{v\}),\\
            0 &\text{if } q < n-\perv(\R^{n-k}\times\{v\}), 
        \end{cases}
    \end{align*}
    where for simplicity, we identify the subset of $N$ with the image of the stratified isomorphism. 
    Furthermore, the isomorphism in the case $q \geq n-\perv(\R^{n-k}\times\{v\})$ is induced by 
    the inclusion $\R^{n-k}\times (cL\smallsetminus \{v\})\hookrightarrow\R^{n-k}\times cL$.  
\end{lemma}
\begin{proof}
    Recall that by Remark \ref{rem:BM'}, we can define ${\IC^{\BM}_{p,\bullet}}'$ for $cL$ and an open subset of this.  
    By the same way as in the proof of Proposition \ref{prop:conebm} and Corollary \ref{cor:anotherconebm}, 
    we have the following isomorphism:
    \begin{align}\label{eq:coneBM'}
        {\IH_{p,q}^{\BM}}'(cL,Q)\cong
        \begin{cases}
            {\IH_{p,q}^{\BM}}'(cL\smallsetminus \{v\},Q)  &\text{if } q \geq k-\perv(\R^{n-k}\times\{v\}),\\
            0 &\text{if } q < k-\perv(\R^{n-k}\times\{v\}).   
        \end{cases}
    \end{align}
    The isomorphism in the case $q \geq k-\perv(\R^{n-k}\times\{v\})$ is induced by the inclusion $cL\smallsetminus\{v\}\to cL$ by the proof of Corollary \ref{cor:anotherconebm}. 

    Using the arguments in \cite[Propositions 2.20 and 2.22]{MR2276609}, there exists two isomorphisms ${\IH_{p,q}^{\BM}}'(cL,Q)\cong {\IH_{p,q+(n-k)}^{\BM}}(\R^{n-k}\times cL,Q)$ and 
    ${\IH_{p,q}^{\BM}}'(cL\smallsetminus\{v\},Q)\cong {\IH_{p,q+n-k}^{\BM}}(\R^{n-k}\times (cL\smallsetminus\{v\}),Q)$. 
    Roughly speaking, these are obtained as follows:
    For each $q$-simplex $\sigma\colon \Delta_q\to cL $, we consider the map $\id_{\R^{n-k}}\times \sigma \colon \R^{n-k}\times \Delta_q \to \R^{n-k}\times cL$. 
    By subdividing $\R^{n-k}\times \Delta_q$ into a collection of $\Delta_{q+(n-k)}$, we obtain the set consisting of $q+(n-k)$-simplices of $\R^{n-k}\times cL$. 
    By sending $\sigma$ to the sum of such $q+(n-k)$-simplices, we obtain the map inducing the above isomorphism for $cL$. 
    For $cL\smallsetminus \{v\}$, we can construct the map in the same way (for details, see \cite[Remarks 2.16 and 2.21]{MR2276609}). 
    Combining \eqref{eq:coneBM'} and the above two isomorphisms, we obtain the statement.
\end{proof}

\begin{lemma}\label{lem:boundsheaf}
    Let $(X,\XX)$ be an  $n$-dimensional  filtered rational polyhedral space with an open face structure. 
    Suppose that $(X,X_n)$ is an $n$-dimensional tropical manifold with singularities and $\XX$ satisfies $(C)$.  
    Then $\Idelta^{p,\bullet}_X[-n]$ (for the definition, see  \eqref{eq:sheafdef}) is isomorphic to some bounded sheaf in $\DD(Q_X)$ and $\HH^q(\Idelta^{p,\bullet}_X[-n])=0$ for $q<0$. 
\end{lemma}
\begin{proof}
    Since $\Idelta^{p,q}_X=0$ for $q>0$, $\HH^q(\Idelta^{p,\bullet}_X[-n])=\HH^{(-n+q)}(\Idelta^{p,\bullet}_X)=0$ for $q>n$. 
    Thus it suffices to show that $\HH^q((\Idelta^{p,\bullet}_X[-n])_x)=0$ for $q<0$ and any point $x\in X$. 
    Recall that for open subset $U\subset X$, $\Idelta^{p,\bullet}_X[-n]|_U\cong \Idelta^{p,\bullet}_U[-n]$ in $\DD(Q_X)$ (see the paragraph above Definition \ref{def:pformsheaf}). 
    Let $V$ be an open subset $V\subset X$ as in Definition \ref{def:semilocgoodratpolyspace}. 
    Since we can cover $X$ by such an open subset $V$, it suffices to show that $\HH^q((\Idelta^{p,\bullet}_V[-n])_x)=0$ for $q<0$ and any point $x\in V$. 
    This is equivalent to $\HH^q((\Idelta^{p,\bullet}_V[-n]))=0$ for $q<0$. 
    By considering such an open subset, it suffices to show that $\HH^q(\Idelta^{p,\bullet}_V[-n])=0$ for $q<0$ and $V\subset X$ such that $\XX|_V=\mathfrak{V}^{U,\CC_V}_V$ for some $\XX|_V$-stratified open face structure $\CC_V$.  

    Let $l((V,\XX|_V))$ denote the maximum codimension among strata of a filtered space $(V,\XX|_V)$. 
    We show that $\HH^q(\Idelta^{p,\bullet}_V[-n])=0$ for $q<0$ by the induction on $l=l((V,\XX|_V))$ as follows. 
    If $l=0$, $V$ is an $n$-dimensional tropical manifold. 
    Recall that for $x\in V$ and a distinguished neighborhood $N$ of $x$ for $\XX^{\CC_V}_V$, we have  $\HH^q((\Idelta^{p,\bullet}_V[-n])_x)\cong \IH^{\BM}_{p,n-q}(N,Q)$ (see \eqref{eq:stalkofsheaf}). 
    Since $N$ has no singular locus, we have $\IH^{\BM}_{p,n-q}(N,Q)\cong \HH^{\BM}_{p,n-q}(N,Q)$. 
    Furthermore, $\HH^{\BM}_{p,n-q}(N,Q)\cong \HH^{p,q}(N,Q)$ by the Poincar\'e duality for tropical manifolds. 
    Thus $\HH^q((\Idelta^{p,\bullet}_V[-n])_x)\cong \HH^{p,q}(N,Q)=0$ for $q<0$. 
    Then $\HH^q(\Idelta^{p,\bullet}_V[-n])=0$ for $q<0$ in the case $l=0$. 

    Assume that $\HH^q(\Idelta^{p,\bullet}_V[-n])=0$ for $q<0$ if $l\leq k$. 
    Let $l=k+1$. 
    For $x\in V^{n-1}$ and a distinguished neighborhood $N\cong \R^{i}\times cL$ of $x$ for $\XX^{\CC_V}_V$, let $N'$ denote $\R^{i}\times(cL\smallsetminus \{v\})$. 
    (Here we identify $\R^{i}\times(cL\smallsetminus \{v\})$ with the image of the stratified homeomorphism as in Lemma \ref{lem:prodgenconeBM}.) 
    By Lemma \ref{lem:prodgenconeBM}, $\HH^q(\Idelta^{p,\bullet}_V[-n]_x)\cong \IH^{\BM}_{p,n-q}(N,Q)$ is isomorphic to zero or $\IH_{p,n-q}^{\BM}(N',Q)$. 
    By the assumption of the induction, $\HH^q(\Idelta^{p,\bullet}_{N'}[-n])=0$ for $q<0$. 
    Then $\Hyper^{q}(\Idelta^{p,\bullet}_{N'}[-n])=0$ for $q<0$. 
    Applying Proposition \ref{prop:homofine} to $\Idelta^{p,\bullet}_{N'}[-n]$,  
    we obtain $\HH^{q}(\Gamma(\Idelta^{p,\bullet}_{N'}[-n]))\cong \Hyper^{q}(\Idelta^{p,\bullet}_{N'}[-n])$. 
    Therefore,  $\IH_{p,n-q}^{\BM}(N',Q)\cong\HH^{q}(\Idelta^{p,\bullet}_{N'}[-n])=0$ for $q<0$. 
    Thus $\HH^q(\Idelta^{p,\bullet}_V[-n]_x)=0$ for $q<0$. 
    As a result, we obtain that $\HH^q(\Idelta^{p,\bullet}_V[-n])=0$ for $q<0$. 
    This completes the inductive proof of $\HH^q(\Idelta^{p,\bullet}_V[-n])=0$ for $q<0$. 
\end{proof}
Let $\Phi$ be a paracompactifying family of supports. 
By Lemma \ref{lem:boundsheaf}, $\Idelta^{p,\bullet}_X$ is quasi-isomorphic to a sheaf complex in $\DD^+(Q_X)$. 
By Proposition \ref{prop:homofine}, 
there exists an isomorphism $\IH_{p,q}^\Phi(X,Q)\cong \Hyper^{-q}_{\Phi}(\Idelta^{p,\bullet}_X)$. 
As in \textsection\ref{subsec:sheaf}, we have a quasi-isomorphism $\Idelta^{p,\bullet}_X|_U\to \Idelta^{p,\bullet}_U$ for any open subset $U\subset X$. 
Thus we obtain $\Hyper^{-q}(\Idelta^{p,\bullet}_X|_U)\cong \Hyper^{-q}(\Idelta^{p,\bullet}_U)\cong \IH_{p,q}(U,Q)$.

\begin{prop}\label{prop:axidelta}
    Let $(X,\XX)$ be an $n$-dimensional filtered rational polyhedral space with a face structure. 
    Suppose that $(X,X_n)$ is an $n$-dimensional tropical manifold with singularities and $\XX$ satisfies $(C)$.  
    Let $\perv$ be a perversity  on $(X,\XX)$. 
    Then the sheaf $\Idelta^{p,\bullet}_X[-n]$ satisfies the axiom $(AX)^{p}_{\perv,\XX}$. 
\end{prop}
\begin{proof}
    We omit details (see \cite[Proposition 3.7]{Friedman_2010}). 
    By \cite[Theorem 6.7 and Proposition A.9]{Gross_2023} (which also holds over a field), there exists an isomorphism $\Omega^{n-p}_{X_n}\cong \Idelta^{p,\bullet}_{X_n}[-n]$ in $\DD(Q_{X_n})$. 
    Moreover, by composing this with the quasi-isomorphism $\Idelta^{p,\bullet}_X|_{X_n}\cong \Idelta^{p,\bullet}_{X_n}$ (see the paragraph above Definition \ref{def:pformsheaf}), there exists an isomorphism $\Idelta^{p,\bullet}_X|_{X_n}[-n]\cong \Omega^{n-p}_{X_n}$ in $\DD(Q_{X_n})$. 
    Combining this with Lemma \ref{lem:boundsheaf}, we see that $\Idelta^{p,\bullet}_X[-n]$ satisfies the condition \ref{itm:ax1}. 

    As in Lemma \ref{lem:boundsheaf}, we may assume that $\XX=\XX^{X_{n},\CC}_X$ for some $\XX$-stratified open face structure $\CC$. 
    For a point $x\in X$, let $N$ be a distinguished neighborhood of $x$ for $\XX^{\CC}_X$. 
    Then $\HH^q((\Idelta^{p,\bullet}_X[-n])_x)\cong \IH_{p,n-q}^{\BM}(N,Q)$ (see \eqref{eq:stalkofsheaf}). 
    Therefore, the condition \ref{itm:ax2} is verified by Lemma \ref{lem:prodgenconeBM}. 

    For a distinguished neighborhood $N$ of $X$ of $x$ for $\XX^{\CC}_X$, let $N\cong \R^i\times cL$.
    Since the direct system in \eqref{eq:3cond1} for $\Idelta^{p,\bullet}_X[-n]$ is constant, 
    by \eqref{eq:3cond1} and Proposition \ref{prop:homofine}, the condition \ref{itm:ax3} for a point $x\in X$ is equivalent to $\HH^q((\Idelta^{p,\bullet}_X[-n])_x)$$\cong \IH_{p,n-q}^{\BM}(\R^i\times cL\smallsetminus \{v\},Q)$. 
    Since $\HH^q((\Idelta^{p,\bullet}_X[-n])_x)\cong \IH_{p,n-q}^{\BM}(N,Q)$, by Lemma \ref{lem:prodgenconeBM}, $\Idelta^{p,\bullet}_X$ satisfies the condition \ref{itm:ax3}. 
\end{proof}
\begin{rem}
    In Proposition \ref{prop:axidelta}, we assume that the existence of a face structure (not just an open face structure) on $X$ to apply \cite[Proposition A.9]{Gross_2023}. 
\end{rem}
\begin{prop}\label{prop:axdeligne}
    Let $(X,\XX)$ be an $n$-dimensional filtered rational polyhedral space with a face structure. 
    Suppose that $(X,X_n)$ is an $n$-dimensional tropical manifold with singularities. 
    Let $\perv$ be a perversity  on $(X,\XX)$. 
    Let $\Q^{\bullet}_{\perv,p}$ be the tropical Deligne sheaf defined in \eqref{eq:Delignesheaf}. 
    Any sheaf complex on $X$ satisfying the axiom $(AX)^{p}_{\perv,\XX}$ is isomorphic in $\DD(Q_X)$ to the tropical Deligne sheaf $\Q^{\bullet}_{\perv,p}$ in \eqref{eq:Delignesheaf}. 
\end{prop}
\begin{proof}
    We apply the  argument  of \cite[Proposition 3.8]{Friedman_2010} to $\XX$.  
    Let $\mathcal{A}^{\bullet}$ be any sheaf complex satisfying the axiom $(AX)^{p}_{\perv,\XX}$. 
    We denote $\mathcal{A}^{\bullet}|_{U_k}$ and $\tau^{X_{n-k}}_{\leq \perv}$ by $\mathcal{A}^{\bullet}_k$ and $\tau_k$, respectively. 
    By the condition \ref{itm:ax1} of the axiom $(AX)^{p}_{\perv,\XX}$, we have $\mathcal{A}^{\bullet}_1\cong\Omega^{n-p}_{X_n}$. 
    We have the following commutative diagram: 
    \begin{equation}\label{eq:comtrun}
        \begin{tikzcd}
        \tau_{k}\mathcal{A}^{\bullet}_{k+1} \arrow[r,"\tau_{k}\alpha_k"] \arrow[d] & \tau_{k}R{i_k}_*\mathcal{A}^{\bullet}_k \arrow[d]  \\
        \mathcal{A}^{\bullet}_{k+1} \arrow[r,"\alpha_k"] & R{i_k}_*\mathcal{A}^{\bullet}_k.
        \end{tikzcd}
    \end{equation}

    For $x\in U_k$, there exists a neighborhood $V$ of $x$ such that $V$ do not intersect any strata in $X_{n-k}$. 
    Thus the left vertical map of \eqref{eq:comtrun} is an isomorphism on the stalk of $x\in U_k$. 
    Let $x\in S$ for a stratum $S\subset X_{n-k}$. 
    By the condition \ref{itm:ax2} of the axiom $(AX)^{p}_{\perv,\XX}$, we have $\HH^i(\mathcal{A}^{\bullet}_x)=0$ for $i>\perv(S)$. 
    Since $\tau_{k}\mathcal{A}^{\bullet}_{k+1}$ agrees with $\tau_{\leq \perv(S)}\mathcal{A}^{\bullet}_{k+1}$ in a neighborhood of $x$,  
    the left vertical map of \eqref{eq:comtrun} is a quasi-isomorphism on the stalk of $x$. 
    As a result, the left vertical map of \eqref{eq:comtrun} is a quasi-isomorphism on the stalk of each point of $U_{k+1}$, and  
    therefore this is a quasi-isomorphism as a sheaf complex on $U_{k+1}$.    

    For $x\in U_k$, $\alpha_k$ is a quasi-isomorphism on the stalk of $x$. 
    Thus the upper horizontal map of \eqref{eq:comtrun} is an isomorphism on the stalk of $x\in U_k$. 
    Let $x\in S$ for a stratum $S\subset X_{n-k}$. 
    By the condition \ref{itm:ax3} of the axiom $(AX)^{p}_{\perv,\XX}$, $\HH^i((\mathcal{A}^{\bullet}|_{U_{k+1}})_x)\cong \HH^i((R{i_k}_*\mathcal{A}^{\bullet}|_{U_k})_x)$ if $\perv(S)\geq i$. 
    Then $\tau_{k}\mathcal{A}^{\bullet}_{k+1}$ and $\tau_{k}R{i_k}_*\mathcal{A}^{\bullet}_k$ agree with $\tau_{\leq \perv(S)}\mathcal{A}^{\bullet}_{k+1}$ and $\tau_{\leq \perv(S)}R{i_k}_*\mathcal{A}^{\bullet}_k$ in a neighborhood of $x$, respectively.  
    It follows that the upper horizontal map $\tau_{k}\alpha_k$ of \eqref{eq:comtrun} is also a quasi-isomorphism on the stalk of each point of $U_{k+1}$. 
    Hence we obtain that $\tau_{k}\alpha_k$ is a quasi-isomorphism. 

    Thus we obtain $\mathcal{A}^{\bullet}_{k+1}\cong \tau_{k}R{i_k}_*\mathcal{A}^{\bullet}_k$. 
    By repeatedly applying this isomorphism, we obtain $\mathcal{A}^{\bullet}=\mathcal{A}^{\bullet}_{n+1}\cong \tau_n R{i_n}_*\ldots\tau_1 R{i_1}_*\Omega^{n-p}_{X_n}\cong \Q^{\bullet}_{\perv,p}$ in $\DD(Q_X)$. 
\end{proof}
\begin{corollary}\label{cor:uniq}
    Let $(X,\XX)$ be an $n$-dimensional filtered rational polyhedral space with a face structure. 
    Suppose that $(X,X_n)$ is an $n$-dimensional tropical manifold with singularities and $\XX$ satisfies $(C)$.  
    Let $\perv$ be a perversity  on $(X,\XX)$. 
    Let $\Q^{\bullet}_{\perv,p}$ be the tropical Deligne sheaf defined in \eqref{eq:Delignesheaf}. 
    There is an isomorphism  $\Idelta^{p,\bullet}_X[-n]\cong \Q^{\bullet}_{\perv,p}$ in $\DD(Q_X)$.
\end{corollary}
\begin{proof}
    By Propositions \ref{prop:axidelta} and \ref{prop:axdeligne}, this statement follows. 
\end{proof}

In this section, we only consider the filtration satisfying $(C)$. 
The tropical intersection homology for the filtration that does not satisfy $(C)$ can be described as the tropical intersection homology for the filtration satisfying $(C)$. 
Although we will not use this fact, we provide a sketch of the proof in the following remark. 
\begin{rem}
    Let $(X,\XX)$ be an $n$-dimensional filtered rational polyhedral space with a face structure $\CC$ such that $(X,X_n)$ is a tropical manifold with singularities.  
    We do not assume that $\XX$ satisfies $(C)$. 
    Then $\XX^{X_n,\CC}_X$ satisfies $(C)$ and $(X,\XX^{X_n,\CC}_X)$ is a tropical manifold with singularities. 
    For a stratum $S$ of $\XX^{X_n,\CC}_X$, we denote by $S_{\XX}$ the stratum of $\XX$ containing $S$. 
    Let $\perv$ be a perversity on $(X,\XX)$. 
    We define a perversity $\perv_{{X_n,\CC}}$ on $(X,\XX^{X_n,\CC}_X)$ by $\perv_{{X_n,\CC}}(S)\colonequals \perv(S_{\XX})+\dim(S_{\XX})-\dim(S)$ for a singular stratum $S$ of $\XX^{X_n,\CC}_X$ and  
    $\perv_{{X_n,\CC}}(S)\colonequals 0$ for a regular stratum $S$ of $\XX^{X_n,\CC}_X$. 
    In this setting, $\Idelta^{p,\bullet}_X$ satisfies $(AX)^{p}_{\perv_{{X_n,\CC}},\XX^{X_n,\CC}_X}$.    
    Therefore, the tropical intersection homology for $\XX$ and $\perv$ is isomorphic to that for $\XX^{X_n,\CC}_X$ and $\perv_{{X_n,\CC}}$ by Corollary \ref{cor:uniq}: 
    $$\IH_{p,q}(X,\XX,Q)\cong  \mathrm{I}^{\perv_{{X_n,\CC}}}\mathrm{H}_{p,q}(X,\XX^{X_n,\CC}_X,Q).$$
\end{rem}

\subsection{Poincar\'e duality}
We introduce the dual perversity for a perversity $\perv$. 
\begin{defn}[dual perversity]
    For a perversity $\perv$, we define the \emph{dual perversity} $\perd$ to be $\perd(S)\colonequals\codim(S)-2-\perv(S)$ for $S\in \SS$.
\end{defn}
Recall that $\DD_X\colon \DD^b(Q_X)\to \DD^b(Q_X)$ is the Verdier dualizing functor (see \textsection\ref{subsec:prelim}). 
Since any tropical Deligne sheaf is isomorphic to a bounded complex in $\DD(Q_X)$, we can consider the Verdier dual of a tropical Deligne sheaf. 

\begin{lemma}\label{lem:prodgencone2}
    Let $(N,\XX_N)$ be an $n$-dimensional filtered rational polyhedral space with an open face structure $\CC_N$ such that $(N,N_n^{\XX_N})$ is a tropical manifold with singularities. 
    Let $k\in\ZZ$ be $0\leq k \leq n$. 
    Suppose that $\XX_N$ has singular strata,  $\XX_N$ satisfies $(C)'$, 
    and $(N,\XX^{\CC_N}_N)$ is stratified-homeomorphic to $\R^{n-k}\times cL$ for some compact $(k-1)$-dimensional filtered space $L$.  
    Then the following holds:
    \begin{align*}
        \IH_{p,q}(N,Q)\cong
        \begin{cases}
            0  &\text{if } q \geq k-\perv(\R^{n-k}\times\{v\})-1,\\
            \IH_{p,q}(\R^{n-k}\times(cL\smallsetminus \{v\}),Q) &\text{if } q < k-\perv(\R^{n-k}\times\{v\})-1, 
        \end{cases}
    \end{align*}
    where for simplicity, we identify the subset of $X$ with the image of the stratified isomorphism. 
    Furthermore, the isomorphism in the case $q \geq n-\perv(\R^{n-k}\times\{v\})$ is induced by 
    the inclusion $\R^{n-k}\times (cL\smallsetminus \{v\})\hookrightarrow\R^{n-k}\times cL$.  
\end{lemma}
\begin{proof}
    As in \textsection\ref{subsec:otherprop}, we define $\IH'_{p,q}(cL,Q)$ and $\IH'_{p,q}(cL\smallsetminus \{v\},Q)$. 
    By Lemma \ref{lem:projiso}, we obtain $\IH_{p,q}(N,Q)\cong \IH'_{p,q}(cL,Q)$. 
    Since the argument in the proof of Lemma \ref{lem:projiso} also applies to $ \IH_{p,q}(\R^{n-k}\times(cL\smallsetminus \{v\}),Q)$, 
    we have $ \IH'_{p,q}(cL\smallsetminus \{v\},Q)\cong \IH_{p,q}(\R^{n-k}\times(cL\smallsetminus \{v\}),Q)$. 
    Since we assume that $\XX_N$ satisfies $(C)'$, the codimension of the stratum containing $\R^{n-k}\times \{v\}$ is $k$. 
    Putting these together, we obtain the statement by Proposition \ref{lem:prodgencone}. 
\end{proof}
\begin{lemma}\label{lem:constsystem}
    Let $(N,\XX_N)$ be an $n$-dimensional filtered rational polyhedral space with an open face structure $\CC_N$. 
    Suppose that $\XX_N$ has singular strata, $\XX_N$ satisfies $(C)'$, and $(N,\XX^{\CC_N}_N)$ is stratified-homeomorphic to $\R^{n-k}\times cL$ for some compact $(k-1)$-dimensional filtered space $L$.  
    For $r>0$, we consider $N_{r},N'_{r}\subset X$ (for the definition, see the paragraph below Remark \ref{rem:delignetheory}). 
    Let $r>r'>0$. Then the natural maps $\IH_{p,q}(N_{r'},Q)\to \IH_{p,q}(N_{r},Q)$ and $\IH_{p,q}(N'_{r'},Q)\to \IH_{p,q}(N'_{r},Q)$ are isomorphisms. 
\end{lemma}
\begin{proof}
    Since the proofs of the two isomorphisms are identical, we only show the first isomorphism. 
    Let $i$ be the natural map $\IC_{p,q}(N_{r'},Q)\to \IC_{p,q}(N_{r},Q)$. 
    We write a simplex $\sigma$ of $N_{r}\cong B^i_r\times L_{[0,r)}$ as $\sigma(x)=(\sigma_{B^i_r}(x),\sigma_{L_{[0,r)}}(x))$ (for details, see the first paragraph of the proof of Proposition \ref{prop:coneGM}). 
    Moreover, we write $\sigma_{L_{[0,r)}}(x)=(\sigma_{[0,r)}(x),\sigma_L(x))\in [0,r)\times L/ \{0\}\times L=L_{[0,r)}$. 
    Then we can write $\sigma=(\sigma_{B^i_r},\sigma_{[0,r)},\sigma_L)$. 
    We define $j\colon  \IC_{p,q}(N_{r},Q)\to \IC_{p,q}(N_{r'},Q)$ by $(\sigma_{B^i_r},\sigma_{[0,r)},\sigma_L)\mapsto (\frac{r'}{r}\sigma_{B^i_r},\frac{r'}{r}\sigma_{[0,r)},\sigma_L)$. 
    By the argument in \cite[Propositions 4.1.10 and 6.3.7]{Friedman_2020}, $i\circ j$ and $j\circ i$ are chain homotopic to identities. 
    Therefore, $i$ is a chain homotopy equivalence and we obtain the claim.   
\end{proof}
\begin{thm}\label{thm:gencomppoin}
    Assume that $Q$ is a field. Let $(X,\XX)$ be an  $n$-dimensional  filtered rational polyhedral space with a face structure. 
    Suppose that $(X,X_n)$ is an $n$-dimensional tropical manifold with singularities and $\XX$ satisfies $(C)$.  
    Let $\perv$ be a perversity  on $(X,\XX)$. 
    Let $\Q^{\bullet}_{\perv,p}$ be the tropical Deligne sheaf defined in \eqref{eq:Delignesheaf}. 
    Then $(\DD_X\Q^{\bullet}_{\perv,p})[-n]\cong \Q^{\bullet}_{\perd,n-p}$ holds in $\DD(Q_X)$. 
\end{thm}
\begin{proof}
    We use the argument of \cite[Theorem 4.3]{Friedman_2010}. 
    We prove this statement by showing that $(\DD_X\Q^{\bullet}_{\perv,p})[-n]$ satisfies the axiom $(AX)^{n-p}_{\perd,\XX}$. 
    By \cite[Proposition 6.6]{Gross_2023}, $(\DD_X\Q^{\bullet}_{\perv,p})[-n]|_{X_n}\cong (\DD_{X_n}\Omega^{n-p}_{X_n})[-n]\cong \Omega^{p}_{X_n}$ follows. 
    Thus the last part of the condition \ref{itm:ax1} of the axiom $(AX)^{n-p}_{\perd,\XX}$ holds. 

    Since $\DD_X$ is a functor to $\DD^b(Q_X)$, we see that $(\DD_X\Q^{\bullet}_{\perv,p})[-n]\in \DD^b(Q_X)$. 

    By applying \cite[Theorem 3.4.4]{MR2286904} (which also holds over a field), to $\Q^{\bullet}_{\perv,p}$, we have the following exact sequence for any open subset $U\subset X$:
    \begin{equation}\label{eq:intexc}
       0\to \Ext^1(\IH_{p,\bullet-1}(U,Q),Q)\to \Hyper^{\bullet}((\DD_X\Q^{\bullet}_{\perv,p})[-n]|_{U})\to \Hom(\IH_{p,\bullet}(U,Q),Q)\to 0. 
    \end{equation}
    Since $Q$ is a field,  $\Hyper^{\bullet}((\DD_X\Q^{\bullet}_{\perv,p})[-n]|_{U})\cong \Hom(\IH_{p,\bullet}(U,Q),Q)$ holds. 
    We have $\IH_{p,q}(U,Q)=0$ for $q<0$. Therefore, $\HH^q((\DD_X\Q^{\bullet}_{\perv,p})[n]|_x)=\varinjlim_{x\in U}\Hyper^q((\DD_X\Q^{\bullet}_{\perv,p})[-n]|_{U})=0$ holds for any point $x\in X$ and $q<0$. 
    Then the second part of the condition \ref{itm:ax1} hold. 
    Thus $(\DD_X\Q^{\bullet}_{\perv,p})[-n]$ satisfies the condition \ref{itm:ax1} of the axiom $(AX)^{n-p}_{\perd,\XX}$. 

    Let $x\in X$ and let $S$ be the stratum containing $x$. 
    By Lemma \ref{lem:prodgencone2}, $\IH_{p,q}(N,Q)=0$ for $q>\perd(S)$ and a distinguished neighborhood $N$ of $x$. 
    It follows that $(\DD_X\Q^{\bullet}_{\perv,p})[-n]$ satisfies the condition \ref{itm:ax2} of the axiom $(AX)^{n-p}_{\perd,\XX}$. 

    Let $x\in X$ and let $N$ be a distinguished neighborhood of $x$. 
    We can write $N\cong \R^{n-k}\times cL$ for some $k\in\ZZ_{\geq 0}$ and a compact $(k-1)$-dimensional filtered space $L$.  
    Let $S$ be the stratum of $N$ containing $x$. 
    By Lemma \ref{lem:constsystem} and $\Hyper^{\bullet}((\DD_X\Q^{\bullet}_{\perv,p})[-n]|_{U})\cong \Hom(\IH_{p,\bullet}(U,Q),Q)$, the direct system in \eqref{eq:3cond2} for $(\DD_X\Q^{\bullet}_{\perv,p})[-n]|_U$ is constant.  
    Then, to prove \ref{itm:ax3}, it suffices to show that 
    the morphism $\Hyper^q((\DD_X\Q^{\bullet}_{\perv,p})[-n]|_N)\to \Hyper^q((\DD_X\Q^{\bullet}_{\perv,p})[-n]|_{N\smallsetminus S})$ is an isomorphism up to $\perd(S)$. 
    Since by \cite[Lemma A.1]{MR2507117}, the exact sequence \eqref{eq:intexc} is functorial, this is equivalent to that $\Hom(\IH_{p,q}(N,Q),Q)\cong \Hom(\IH_{p,q}(N\smallsetminus S,Q),Q)$ for $q\leq \perd(S)$. 
    By Lemma \ref{lem:prodgencone2}, this isomorphism holds.  
    Then the condition \ref{itm:ax3} of the axiom $(AX)^{n-p}_{\perd,\XX}$ is satisfied. 
\end{proof}

\begin{thm}[Poincar\'e duality]\label{thm:genpoin}
    Assume that $Q$ is a field. Let $(X,\XX)$ be an  $n$-dimensional  filtered rational polyhedral space with a face structure. 
    Suppose that $(X,X_n)$ is an $n$-dimensional tropical manifold with singularities and $\XX$ satisfies $(C)$.  
    Let $\perv$ be a perversity  on $(X,\XX)$. 
    Then there exists an isomorphism 
    \begin{equation}\label{eq:poincare}
        \IH^{n-p,n-q}(X,\XX,Q)\cong \IHd_{p,q}^{\BM}(X,\XX,Q)
    \end{equation}
    between the tropical intersection cohomology and the tropical intersection Borel--Moore homology. 
\end{thm}
\begin{proof}
    By \eqref{eq:intexc} for $U=X$, we have $\Hyper^{n-q}((\DD_X\Q^{\bullet}_{\perv,n-p})[-n])\cong \Hom(\IH_{n-p,n-q}(X,\XX,Q),Q)$. 
    We have $(\DD_X\Q^{\bullet}_{\perv,n-p})[-n]\cong \Q^{\bullet}_{\perd,p}$ by Theorem \ref{thm:gencomppoin}.  
    We also have  $\Ideltad^{p,\bullet}_X[-n]\cong \Q^{\bullet}_{\perd,p}$ by Corollary \ref{cor:uniq}.  
    By these results, we obtain $\Hom(\IH_{n-p,n-q}(X,\XX,Q),Q)\cong \IHd_{p,q}^{\BM}(X,\XX,Q)$. 
    Since $Q$ is a field, $\IH^{n-p,n-q}(X,\XX,Q)\cong \Hom(\IH_{n-p,n-q}(X,\XX,Q),Q)$ is implied by 
    the universal coefficient theorem (\cite[Theorem 3.6.5]{Weibel_1994}). 
    Combining these isomorphisms, we obtain \eqref{eq:poincare}. 
\end{proof}

\begin{corollary}
    Assume that $Q$ is a field. Let $X$ be a pure $n$-dimensional rational polyhedral space with a face structure.  
    Suppose that $\XX_X^{\text{trop}}$ satisfies $(C)$.  
    Let $\perv$ be a perversity  on $(X,\XX)$. 
    Then there exists an isomorphism $$\IH^{n-p,n-q}(X,\XX_X^{\text{trop}},Q)\cong \IHd_{p,q}^{\BM}(X,\XX_X^{\text{trop}},Q)$$ between the tropical intersection cohomology and the tropical intersection Borel--Moore homology.  
\end{corollary}
\begin{corollary}
    Assume that $Q$ is a field. Let $X$ be an $n$-dimensional rational polyhedral space with a face structure. 
    Let $(X,U)$ be an $n$-dimensional tropical manifold with singularities. 
    Suppose that $U$ is the union of some strata of $\XX^{\CC}_X$. 
    Let $\perv$ be a perversity  on $(X,\XX)$. 
    Then there exists an isomorphism $$\IH^{n-p,n-q}(X,\XX^{U,\CC}_X,Q)\cong \IHd_{p,q}^{\BM}(X,\XX^{U,\CC}_X,Q)$$ between the tropical intersection cohomology and the tropical intersection Borel--Moore homology.  
\end{corollary}

\subsection{Bilinear pairings}
Let $(X,\XX)$ be an  $n$-dimensional  filtered rational polyhedral space with a face structure. 
We assume that $\XX$ is a filtration satisfying $(C)$ and that $(X,X_n)$ is a tropical manifold with singularities. 
Let $\perv$, $\perq$, and $\per$ be perversities on $X$ such that $\perv(S)+\perq(S)\leq\per(S)$ for any stratum $S$. 
By the same argument in \cite[Theorem 4.6]{Friedman_2010}, 
the natural morphism $\Omega^{p_1}_{X_{n}}\otimes\Omega^{p_2}_{X_{n}}\to\Omega^{p_1+p_2}_{X_n}$  induces the morphism of the tropical Deligne sheaves: 
\begin{equation}\label{eq:genpairing}
    \Q_{\perv,n-p_1}^{\bullet}\otimes^L\Q_{\perq,n-p_2}^{\bullet}\to\Q_{\per,n-(p_1+p_2)}^{\bullet}.
\end{equation}
This is roughly constructed as follows: 
By \cite[Lemma 4.9]{Friedman_2010} and the construction of a tropical Deligne sheaf \eqref{eq:Delignesheaf}, 
there exists a one-to-one correspondence between morphisms $(\Q_{\perv,n-p_1}^{\bullet}\otimes^L\Q_{\perq,n-p_2}^{\bullet})|_{U_k}\to \Q_{\per,n-(p_1+p_2)}^{\bullet}|_{U_k}$ and 
morphism $(\Q_{\perv,n-p_1}^{\bullet}\otimes^L\Q_{\perq,n-p_2}^{\bullet})|_{U_{k+1}}\to \Q_{\per,n-(p_1+p_2)}^{\bullet}|_{U_{k+1}}$. 
The above natural morphism $\Omega^{p_1}_{X_{n}}\otimes\Omega^{p_2}_{X_{n}}\to\Omega^{p_1+p_2}_{X_n}$ defines the morphism $(\Q_{\perv,n-p_1}^{\bullet}\otimes^L\Q_{\perq,n-p_2}^{\bullet})|_{U_{1}}\to \Q_{\per,n-(p_1+p_2)}^{\bullet}|_{U_{1}}$. 
By repeatedly applying the above correspondence, we obtain the desired morphism \eqref{eq:genpairing}. 

By \cite[Exercise II.17]{MR1074006}, there exists a morphism:
\begin{equation*}
    \Hyper^{q_1}(\Q_{\perv,n-p_1}^{\bullet}) \otimes \Hyper^{q_2}(\Q_{\perq,n-p_2}^{\bullet})\to \Hyper^{q_1+q_2}(\Q_{\perv,n-p_1}^{\bullet}\otimes^L\Q_{\perq,n-p_2}^{\bullet}). 
\end{equation*}
Using \eqref{eq:genpairing}, we have $\Hyper^{q_1+q_2}(\Q_{\perv,n-p_1}^{\bullet}\otimes^L\Q_{\perq,n-p_2}^{\bullet})\to \Hyper^{q_1+q_2}(\Q_{\per,n-(p_1+p_2)}^{\bullet})$. 
Combining these morphisms, we obtain the following: 
\begin{equation}\label{eq:genoricup}
    \IH_{n-p_1,n-q_1}^{\BM} (X,\XX,Q)\otimes\IHq_{n-p_2,n-q_2}^{\BM} (X,\XX,Q)\to \IHr_{n-(p_1+p_2),n-(q_1+q_2)}^{\BM}(X,\XX,Q).
\end{equation}

Let $a_X$ be the map $X\to\{\text{pt}\}$. 
By applying \cite[(2.6.23), Exercise I.24]{MR1074006} to  $a_X$, we obtain the following: 
\begin{equation*}
    \Hyper^{q_1}(\Q_{\perv,n-p_1}^{\bullet}) \otimes \Hyper^{q_2}_c(\Q_{\perq,n-p_2}^{\bullet})\to \Hyper^{q_1+q_2}_c(\Q_{\perv,n-p_1}^{\bullet}\otimes^L\Q_{\perq,n-p_2}^{\bullet}). 
\end{equation*}
Using \eqref{eq:genpairing}, we have $\Hyper^{q_1+q_2}_c(\Q_{\perv,n-p_1}^{\bullet}\otimes^L\Q_{\perq,n-p_2}^{\bullet})\to \Hyper^{q_1+q_2}_c(\Q_{\per,n-(p_1+p_2)}^{\bullet})$. 
Combining these morphisms, we obtain the following: 
\begin{equation}\label{eq:genoricap}
    \IH_{n-p_1,n-q_1}^{\BM} (X,\XX,Q)\otimes\IHq_{n-p_2,n-q_2} (X,\XX,Q)\to \IHr_{n-(p_1+p_2),n-(q_1+q_2)}(X,\XX,Q).
\end{equation}
\begin{defn}[{\cite[Definition 7.2.6]{Friedman_2020}}]
    We assume that $Q$ is a field.   
    Suppose that $\perv$, $\perq$, and $\per$ are perversities on X. We will say that $(\perv,\perq;\per)$ is an \emph{agreeable triple} if 
    \begin{equation}\label{eq:agreeable}
        \per(S) \leq \perv(S)+\perq(S)+2 - \codim(S) 
    \end{equation}
    for any singular stratum $S$.
\end{defn}
By a straightforward calculation, \eqref{eq:agreeable} holds if and only if $\perd(S)+D\perq(S)\leq D\per(S)$ if and only if $D\perq(S)+\per(S)\leq\perv(S)$. 

Suppose that the coefficient ring $Q$ is a field and $(\perv,\perq;\per)$ is an agreeable triple.  
Since $\perd(S)+D\perq(S)\leq D\per(S)$ holds for any stratum $S$, by combining the pairing $\eqref{eq:genoricup}$  with the Poincar\'e duality, 
we obtain a bilinear pairing analogous to a cup product: 
\begin{equation}\label{eq:gencup}
    \IH^{p_1,q_1}(X,\XX,Q)\otimes\IHq^{p_2,q_2}(X,\XX,Q)\to \IHr^{p_1+p_2,q_1+q_2}(X,\XX,Q). 
\end{equation}

Suppose that the coefficient ring $Q$ is a field and  $(\per,\perv;\perq)$ is an agreeable triple. 
Since $D\perv(S)+\perq(S)\leq \per(S)$ holds for any stratum $S$, by combining the pairing $\eqref{eq:genoricap}$  with the Poincar\'e duality, 
we obtain a bilinear pairing like a cap product: 
\begin{equation}\label{eq:gencap}
    \IH^{p_1,q_1} (X,\XX,Q)\otimes\IHq_{n-p_2,n-q_2}(X,\XX,Q)\to \IHr_{(n-p_1)-p_2,(n-q_1)-q_2}(X,\XX,Q). 
\end{equation}

\section{Examples}\label{sec:ex}
In this section, unless otherwise specified, the coefficient ring $Q$ is a Dedekind domain with $\ZZ\subset Q\subset\RR$. 
\subsection{Failure of the Poincaré duality in GM tropical intersection (co)homology}\label{subsec:counter}
We show that, in the $\GM$ tropical intersection (co)homology, the Poincar\'e duality does not hold in general, even for a tropical manifold. 

For $\mathbf{e}_1=(1,0)$, $\mathbf{e}_2=(0,1)$, and $\mathbf{e}_3=(-1,-1)$ in $\RR^2$, 
let $X=\RR_{\geq 0}\mathbf{e}_1+\RR_{\geq 0}\mathbf{e}_2+\RR_{\geq 0}\mathbf{e}_3$. 
We can see $X$ as a $1$-dimensional fan with three rays.  
This is in fact the matroidal fan induced from the matroid $U_3^1$ on $\{1,2,3\}$, where the independent sets are $A\subset\{1,2,3\}$ with $|A|\leq 1$.  
We consider $1$-dimensional filtration $\XX^{\text{trop}}_X$ on $X$. 
This filtration is the $1$-dimensional filtration such that $X^1=X$, $X^0=\{(0,0)\}$, and $X^{-1}=\emptyset$. 

Let $Q$ be a field. Let $v=(0,0)$ and let $\perv$ be a perversity on $(X,\XX^{\text{trop}}_X)$ defined by $\perv(\{v\})=0$. 
We compute the cohomology by the cone formula (Proposition \ref{prop:coneGM}) and ${}_{\GM}\IH^{p,q}(X,Q)\cong \Hom({}_{\GM}\IH_{p,q}(X,Q),Q)$ as follows. 
\begin{align*}
    _{\GM}\IH^{0,q}(X,Q)\cong
    \begin{cases}
        Q &\text{if } q =0,\\
        0 &\text{if } q \neq 0.
    \end{cases}
\end{align*}

For the $\BM$ homology, we have $_{\GM}\IHd_{1,q}^{\BM}(X,Q)=0$. 
Indeed, we can show this as follows. 
Since the second and third cases in Proposition \ref{prop:coneGM}  do not occur for the perversity $\perd$, we can prove Proposition \ref{prop:conebm} for the $\GM$ homology in this setting.  
Then we have 
\begin{align*}
    _{\GM}\IHd^{\BM}_{1,q}(X,Q)\cong
    \begin{cases}
        _{\GM}\IHd_{1,q-1}(X\smallsetminus \{(0,0)\},Q) &\text{if } q \geq 2,\\
        0 &\text{if } q < 2.
    \end{cases}
\end{align*}
Since $_{\GM}\IHd_{1,q-1}(X\smallsetminus \{(0,0)\},Q)\cong \HH_{1,q-1}(X\smallsetminus \{(0,0)\},Q)=0$ for $q\geq 2$. 
Then we obtain $_{\GM}\IHd_{1,q}^{\BM}(X,Q)=0$.

Therefore, $\GM$ homology does not satisfy the Poincar\'e duality as above. 
\begin{rem}
    The space $X$ satisfies the Poincar\'e duality for non-$\GM$ homology by Theorem \ref{thm:genpoin}. 
    In fact, in the non-$\GM$ case, all the above (co)homology groups are zero. 
\end{rem}

\subsection{Computations for 1-dimensional cases}\label{subsec:onedim}  
Let $X$ be a pure $1$-dimensional rational polyhedral space with the filtration $\XX_X^{\text{trop}}$ in \eqref{eq:canstrat}. 
In this subsection, we suppose that there exists an open face structure $\CC$ such that $\XX^{\text{trop}}_X=\XX^{\CC}_X$ and we fix such $\CC$.  
Let $\perv$ be a perversity on $(X,\XX_X^{\text{trop}})$. 
Note that $\perv$ is a map from the 0-dimensional strata to $\Z$. 
We call $1$-dimensional strata  edges and 0-dimensional strata  vertices. 
We denote by $E(X)$ and $V(X)$ the set of edges of $X$ and the set of vertices $X$, respectively. 

\subsubsection{Non-GM case}\label{subsubsec:comp}
For the non-$\GM$ case, we can compute the homology as the direct sum of the homology groups of the closure of each edge. 
Firstly, we consider the homology of the closure $\bar{e}$ of an edge $e$. 
The closure $\bar{e}$ is homeomorphic to either an open interval, a half-open/closed interval, or a closed interval.  
We set 
\begin{align*}
    A=\{e\in E(X)\mid &\quad \bar{e} \text{ has two vertices }v_1,v_2\text{ and } \perv(\{v_i\}) \geq 0 \text{ for }i=1 \text{ and }2\}, \\
    B=\{e \in E(X)\mid &\quad \bar{e} \text{ has no vertices, or }\perv(\{v\})<0 \text{ for any vertex of }\bar{e}\}.
\end{align*}

Since $\IH_{0,q}(\bar{e},Q)$ coincides with the classical intersection homology of $\bar{e}$, we obtain that
\begin{align*}
    \IH_{0,q}(\bar{e},Q)\cong
    \begin{cases}
        Q &\text{if } (q =1\text{ and }e\in A) \text{ or } (q=0\text{ and }e \in B),\\
        0 &\text{otherwise}. 
    \end{cases}
\end{align*}
Since the multi-tangent at each edge is $Q$ for $p=1$, the homology for the case $p=1$ is the same as above. 
In the case $p\neq 0,1$, the multi-tangent is $0$ at each edge. 
Therefore, the homology is zero for $p\neq 0,1$. 

We set $a=|A|$ and $b=|B|$. 
Since the multi-tangent at each $\sigma\in \CC$ with $d(\sigma)>0$ is the same for $p=0,1$, $\IH_{p,q}(X,\XX_X^{\text{trop}},Q)$ coincides with the ordinary intersection homology. 
Therefore, by \cite[Proposition 6.3.47]{Friedman_2020}, $\IH_{p,q}(X,\XX_X^{\text{trop}},Q)$ is isomorphic to the direct sum of $\IH_{p,q}(\bar{e},Q)$ for $p=0,1$. 
In the case $p\neq 0,1$, $\IH_{p,q}(X,\XX_X^{\text{trop}},Q)=0$ since the multi-tangent at  each edge is trivial. 
Then we obtain the following:
\begin{align*}
    \IH_{p,q}(X,\XX_X^{\text{trop}},Q)\cong
    \begin{cases}
        Q^{\oplus a} &\text{if } q=1\text{ and }p=0,1,\\
        Q^{\oplus b} &\text{if } q=0\text{ and }p=0,1,\\
        0 &\text{otherwise}. 
    \end{cases}
\end{align*}
\begin{example}
    Let $\perv=\perz$.  
The non-$\GM$ homology of $X$ in degree one is the free $Q$-module with rank the number of the edges that have two endpoints. 
Furthermore, the non-$\GM$ homology in other degrees is zero.
\end{example}
 By the above computations, we can directly verify the Poincar\'e duality isomorphism if $X$ is a $1$-dimensional compact rational polyhedral space. 
\subsubsection{GM case}\label{subsubsec:GM}

For the $\GM$ case, we can compute intersection homology by the tropical homology. 

We set $V_{\perv}\colonequals \{v\in V(X)\mid \perv(\{v\})<0\}$. 

\begin{lemma}\label{lem:exgmfirst}
    We have ${}_{\GM}\IH_{p,q}(X\smallsetminus  V_{\perv},\XX_X^{\text{trop}},Q)\cong {}_{\GM}\IH_{p,q}(X,\XX_X^{\text{trop}},Q)$. 
\end{lemma}
\begin{proof}
    Setting $F_{q}(U)\colonequals {}_{\GM}\IH_{p,q}(U\smallsetminus (U\cap V_{\perv}),Q)$ and $G_{q}(U)\colonequals {}_{\GM}\IH_{p,q}(U,Q)$ for an open subset $U\subset X$,
we verify the conditions in Proposition \ref{prop:argument} for $O_X$ and $(X,\XX^{\text{trop}}_X)$.  The condition \ref{itm:mv1}, \ref{itm:mv2}, and \ref{itm:mv4} are apparent. 

We will verify the condition \ref{itm:mv3}. 
Since $X$ is $1$-dimensional, it suffices to show that $F_q(N)\cong G_q(N)$ for any open subset $N\subset X$ such that $N$ is stratified-homeomorphic to $cL$ for some finite set $L$ of discrete points. 
We divide the proof into several cases according to $q$ and the value of $\perv(\{v\})$ for the vertex $v$ of $cL$. 

Suppose that $\perv(\{v\})\geq 0$, then  $N\smallsetminus N\cap V_{\perv}=N$. 
It follows that $F_q(N)= G_q(N)$ for any $q$. 

Suppose that $q<-\perv(\{v\})$ and $\perv(\{v\})<0$. 
Since $0<1-\perv(\{v\})-1=-\perv(\{v\})$, we have $G_q(N)\cong F_q(N)$ by the bottom isomorphism in Proposition \ref{prop:coneGM}. 

Suppose that $q\geq -\perv(\{v\})$ and $\perv(\{v\})<0$.  
Then $G_q(N)=0$ by the first isomorphism in Proposition \ref{prop:coneGM}. 
Since $N\smallsetminus \{v\}$ has no singular strata, we have also $F_q(N)={}_{\GM}\IH_{p,q}(N\smallsetminus \{v\},Q) \cong \HH_{p,q}(N\smallsetminus \{v\},Q) =0$ for $q>0$. 
Since $q>0$ in this case, we have $G_q(N)=0 \cong F_q(N)$. 

Therefore, the condition \ref{itm:mv3} holds in any case. 

By Proposition \ref{prop:argument},  we obtain that ${}_{\GM}\IH_{p,q}(X\smallsetminus  V_{\perv},\XX_X^{\text{trop}},Q)\cong {}_{\GM}\IH_{p,q}(X,\XX_X^{\text{trop}},Q)$. 
\end{proof}
By Remark \ref{rem:GMandTropical}, there exists the natural map 
\begin{equation}\label{eq:exgmsecond}
   _{\GM}\IH_{p,q}(X\smallsetminus  V_{\perv},\XX_X^{\text{trop}},Q)\to\HH_{p,q}(X\smallsetminus  V_{\perv},Q), 
\end{equation}
which is induced by the inclusion. 
\begin{lemma}\label{lem:exgmsecond}
    The map \eqref{eq:exgmsecond} is an isomorphism. 
\end{lemma}
\begin{proof}
    We apply Proposition \ref{prop:argument} to $F_q(U)={}_{\GM}\IH_{p,q}(U,Q)$ and $G_q(U)=\HH_{p,q}(U,Q)$ for an open subset $U\subset X\smallsetminus\{v\}$ and for $ O_{X\smallsetminus V_{\perv}}$. 
The conditions \ref{itm:mv1}, \ref{itm:mv2}, and \ref{itm:mv4} are apparent. 
We need to show the condition \ref{itm:mv3}. 
For this, we consider an open subset $N\subset X\smallsetminus V_{\perv}$ such that $N\cong cL$ for some finite set $L$ of discrete points as  above. 
We denote by $v$ the vertex of $cL$. 

We will show the isomorphism between $_{\GM}\IH_{p,q}(N,Q)$ and $\HH_{p,q}(N,Q)$. 
Since $N$ is an open subset of $X\smallsetminus V_{\perv}$, we may assume that $\perv(\{v\})\geq 0$ for the vertex $v$ of the vertex of $cL$. 
We divide the proof into several cases according to $q$. 

We first consider the case $q=0$. 
By Proposition \ref{prop:coneGM}, we have $$_{\GM}\IH_{p,0}(N,Q)\cong \sum \limits_{ \substack{\sigma' \in \CC_F: \{v\} \prec \sigma',\\ \perv(\sigma')\geq \codim(\sigma')}}{\bf F}^{Q}_p(\sigma')\cong {\bf F}^{Q}_p(\{v\}).$$
By Remark \ref{rem:perv} (2), we may regard $\HH_{p,q}(F,Q)$ as the $\GM$ homology for a sufficiently large perversity. 
Then, $\HH_{p,0}(N,Q)\cong {\bf F}^{Q}_p(\{v\})$ by Proposition \ref{prop:coneGM}. 
It follows that we have $_{\GM}\IH_{p,0}(N,Q)\cong \HH_{p,0}(N,Q)$. 

For the case $q\neq 0$, ${}_{\GM}\IH_{p,q}(N,Q)$ is isomorphic to $0$ or ${}_{\GM}\IH_{p,q}(N\smallsetminus\{v\},Q)$ by Proposition \ref{prop:coneGM}. 
Since  ${}_{\GM}\IH_{p,q}(N\smallsetminus\{v\},Q)\cong \HH_{p,q}(N\smallsetminus\{v\},Q)=0$ for $q\neq 0$, we have ${}_{\GM}\IH_{p,q}(N,Q)=0$ for $q\neq 0$. 
As $\HH_{p,q}(N,Q)=0$ for $q\neq 0$, we have $_{\GM}\IH_{p,q}(N,Q)=0\cong \HH_{p,q}(N,Q)$ for $q\neq 0$. 
Therefore, the condition \ref{itm:mv3} holds for any case. 

As a result, we obtain an isomorphism 
$$_{\GM}\IH_{p,q}(X\smallsetminus  V_{\perv},\XX_X^{\text{trop}},Q)\cong\HH_{p,q}(X\smallsetminus  V_{\perv},Q), $$
by Proposition \ref{prop:argument}. 
\end{proof}

Combining Lemmas \ref{lem:exgmfirst} and \ref{lem:exgmsecond}, we obtain the following. 
\begin{prop}
    We have 
    $_{\GM}\IH_{p,q}(X,\XX_X^{\text{trop}},Q)\cong\HH_{p,q}(X\smallsetminus  V_{\perv},Q)$.
\end{prop}

\subsection{Examples with a non-canonical filtration and singularities}\label{subsec:filtsing}
We present examples of tropical intersection homology with a filtration different from $\XX_X^{\text{trop}}$.  
 
\subsubsection{Tropical manifold with singularities}\label{subsubsec:tropmfd}

Let $(X,U)$ be a $1$-dimensional tropical manifold with singularities (see Definition \ref{def:tms}). 
Assume that $X$ is compact. We consider the following filtration:
\[
\XX\colon X\supset X\smallsetminus U\supset \emptyset.
\]
In this setting, a perversity $\perv$ is a map from the set of singular points to $\ZZ$. 
We consider a perversity $\perv$ such that there exists $m\in \Z$ such that $\perv(\{v\})=m$ for any singular point $v$. 

By direct computation, we can compute the tropical intersection homology. 
\begin{align*}
    \IH_{p,q}(X,Q)\cong
    \begin{cases}
        \HH_{p,q}(U,Q) &\text{if } m<0,\\
        \HH_{p,q}^{\BM}(U,Q) &\text{if } m\geq 0. 
    \end{cases}
\end{align*}

The first equality follows by a similar argument as in  Lemma \ref{lem:exgmfirst}. 

We will verify the second equality. 
By a similar argument as in Lemma \ref{lem:exgmsecond},  
we have an isomorphism $\IH_{p,q}(X,Q)\cong \HH_q(\HC_{p,\bullet}(X,Q))$. 
Since $\HC_{p,\bullet}(X,Q)\cong {}_{\GM}\HC(X,X\smallsetminus U,Q)$ by the definition, 
we have $\HH_q(\HC_{p,\bullet}(X,Q))\cong \HH_{p,q}(X,X\smallsetminus U,Q)$. 
Using a slightly smaller open subset $U'$, we obtain 
\begin{align*}
    \IH_{p,q}(X,Q)\cong \HH_{p,q}(X,X\smallsetminus U,Q)\cong \HH_{p,q}(X,X\smallsetminus U',Q) \\
    \cong \HH_{p,q}(U,U\smallsetminus U',Q) \cong \HH_{p,q}^{\BM}(U,Q).
\end{align*}

Thus, if $Q$ is a field in the above setting, we obtain $\IH^{1-p,1-q}(X,Q)\cong\IHd_{p,q}(X,Q)$ by the Poincar\'e duality in tropical homology. 
This verifies Theorem \ref{thm:genpoin} for the case under consideration. 

On the other hand, using Proposition \ref{prop:argument} as in the argument in \textsection\ref{subsubsec:GM}, we can compute the GM case:
\begin{align*}
    _{\GM}\IH_{p,q}(X,Q)\cong
    \begin{cases}
        \HH_{p,q}(U,Q) &\text{if } m<0,\\
        \HH_{p,q}(X,Q) &\text{if } m\geq 0. 
    \end{cases}
\end{align*}

\subsubsection{Rational polyhedral space with singular points}\label{subsubsec:ratpoly}
We define a rational polyhedral space with singular points inspired by the definition of an affine manifold with singularities \cite{MR1975331}. 
\begin{defn}[rational polyhedral space with singular points]
    Let $B$ be a topological space and let $X\subset B$ be an open dense subset of $B$ such that $B\smallsetminus X$ is a finite set. 
    If $X$ is a rational polyhedral space, 
    then we call $(B,X)$ a \emph{rational polyhedral space with singular points}. 
\end{defn}
\begin{example}[Forcus-forcus singularity]
    Let $X$ be a space obtained by gluing 
    $\RR^2\smallsetminus (\{0\} \times \RR_{\geq 0})$ and $\RR^2\smallsetminus (\{0\} \times \RR_{\leq 0})$ via 
    \[
    (\RR\smallsetminus \{0\}) \times \RR \longrightarrow  (\RR\smallsetminus \{0\}) \times \RR \qquad (x,y)\longmapsto 
    \begin{cases}
        (x,y)&  x<0\\
        (x,x+y)& x>0. 
    \end{cases}
    \]
    Let $B$ be the union of the origin and $X$. 
    Then $(B,X)$ is a rational polyhedral space with singular points. 
    This is called the  focus-focus singularity in \cite[p. 51]{MR2893676}. 
    \end{example}
\begin{example}[Wedge product of two $\R^2$]
    Let $B$ be the wedge product of $\R^2$ obtained by gluing the origins of two $\R^2$. 
    We define $X$ to be $(\R^2\smallsetminus \{0\})\bigsqcup (\R^2\smallsetminus \{0\}) $. 
    Then $(B,X)$ is also a rational polyhedral space with singular points. 
\end{example}
For $\R^2$, we consider the complex on consisting of 
the four cones $\RR_{\geq 0}\mathbf{e}_1+\RR_{\geq 0}\mathbf{e}_2$, $\RR_{\leq0}\mathbf{e}_1+\RR_{\geq 0}\mathbf{e}_2$, $\RR_{\leq 0}\mathbf{e}_1+\RR_{\leq 0}\mathbf{e}_2$, and $\RR_{\geq 0}\mathbf{e}_1+\RR_{\leq 0}\mathbf{e}_2$ 
and their faces. 
Using this complex, we define the complexes on $B$ of the above two examples. 
We denote these complexes by $\CC$ in both cases. 
Then we naturally define $\prec$ on $\CC$. 
Note that $\CC|_X$ defines an open face structure on $X$ in each case. 
We will denote 

Let $(B,X)$ be the focus-focus singularity or the wedge product of two $\R^2$ as above. We give the following filtration:
\begin{equation}\label{eq:singfilter}
    \XX\colon B\supset B\smallsetminus X=\{v\}\supset \{v\}\supset \emptyset.
\end{equation}

\begin{rem}
    In \eqref{eq:singfilter}, by considering two identical sets in the middle, we regard $(B,\XX)$ as a $2$-dimensional CS set.
\end{rem}

Then two $(B,X)$ in the above have filtrations $\XX$ and complex structures $\CC$. 
We will construct homology using $(B,X)$, $\CC$, and $\XX$. 
A rational polyhedral space with singular points is no longer  a rational polyhedral space. Thus, to define homology, we need some new coefficients. 
Let $\iota\colon X\to B$ be the inclusion map. 
We define a sheaf of extended $p$-form as follows:
\begin{align*}
    \Omega^p_B\colonequals\iota_*\Omega^p_X
\end{align*}

We define a \emph{generalized $p$-th multi-tangent} ${\bf F}^Q_{p}(\sigma)$ at $\sigma\in \CC$ to be $\Hom(\Omega^p_B,Q_{\sigma})$ (for the definition of $Q_{\sigma}$, see the paragraph below Definition \ref{def:pformsheaf}). 
We regard ${\bf F}^Q_{p}(\sigma)=\Hom(\Omega^p_B,Q_{\sigma})$ as a $Q$-module. 
For $\tau,\sigma\in \CC$ with $\tau\prec\sigma$, the natural map $Q_{\sigma}\to Q_{\tau}$ induces ${\bf F}^Q_{p}(\sigma)\to {\bf F}^Q_{p}(\tau)$. 
Using this multi-tangent, we define a generalized tropical intersection homology in the same way as in \textsection\ref{subsec:tropinthom}. 
We also denote by $_{(\GM)}\IH_{p,q}(B,Q)$ the generalized tropical intersection homology. 
Moreover, using this multi-tangent, we also define a generalized tropical homology $\HH_{p,q}(B,Q)$. 
\begin{rem}
The generalized tropical homology for the focus-focus singularity coincides with one in \cite[Definition 6.1]{yamamoto2024}. 
  In \cite[p. 3083]{MR4347312}, a similar extension was carried out. For these two extended homologies, the chain groups do not coincide 
  but the homology groups of the above focus-focus singularity coincide.    
\end{rem}
Noting that $1-\perv(\{v\})=2-\perv(\{v\})-1$, we argue as in  Proposition \ref{prop:coneGM}, and we have the following.  
\begin{align*}
        _{\GM}\IH_{p,q}(B,Q)\cong
        \begin{cases}
            0 &\text{if } q \geq 1-\perv(\{v\})\text{ and }q \neq 0,\\
            \HH_{p,0}(B) & \text{if } q \geq 1-\perv(\{v\})\text{ and }q = 0,\\
            _{\GM}\IH_{p,q}(B\smallsetminus \{v\},Q)\quad(\cong \HH_{p,q}(X,Q)) &\text{if } q < 1-\perv(\{v\}).
        \end{cases}
\end{align*}

For the non-$\GM$ case, we argue as in Proposition \ref{prop:conenon}, and we  have the following.  
\begin{align*}
        \IH_{p,q}(B,Q)\cong
        \begin{cases}
            0 &\text{if } q \geq 1-\perv(\{v\}),\\
            \IH_{p,q}(B\smallsetminus \{v\},Q)\quad(\cong \HH_{p,q}(X,Q)) &\text{if } q < 1-\perv(\{v\}).
        \end{cases}
\end{align*}
\begin{rem}
We can also define  BM homology and  cohomology for the above examples. 
In the non-$\GM$ case, we have their cone formulas. 
These statements have the same form as  Corollary \ref{cor:anotherconebm} and  Proposition \ref{prop:conecohom}. 
Assume that $Q$ is a field. 
By these statements and the Poincar\'e duality for tropical manifolds, we obtain that $\IH_{p,q}^{\BM}(B,Q)\cong\IHd^{2-p,2-q}(B,Q)$ in the above two examples. 
In particular, $\IH_{p,q}^{\BM}(B,Q)\cong\IH^{2-p,2-q}(B,Q)$ holds for the perversity such that $\perv(\{v\})=0$. 
\end{rem}

\bibliographystyle{amsalpha}
\bibliography{bibs}

\newcommand{\etalchar}[1]{$^{#1}$}
\providecommand{\bysame}{\leavevmode\hbox to3em{\hrulefill}\thinspace}
\providecommand{\MR}{\relax\ifhmode\unskip\space\fi MR }
% \MRhref is called by the amsart/book/proc definition of \MR.
\providecommand{\MRhref}[2]{%
  \href{http://www.ams.org/mathscinet-getitem?mr=#1}{#2}
}
\providecommand{\href}[2]{#2}
\begin{thebibliography}{ABL{\etalchar{+}}15}

\bibitem[ABL{\etalchar{+}}15]{MR3284392}
Pierre Albin, Markus Banagl, Eric Leichtnam, Rafe Mazzeo, and Paolo Piazza,
  \emph{Refined intersection homology on non-{W}itt spaces}, J. Topol. Anal.
  \textbf{7} (2015), no.~1, 105--133. \MR{3284392}

\bibitem[AM69]{MR242802}
M.~F. Atiyah and I.~G. Macdonald, \emph{Introduction to commutative algebra},
  Addison-Wesley Publishing Co., Reading, Mass.-London-Don Mills, Ont., 1969.
  \MR{242802}

\bibitem[Ban07]{MR2286904}
M.~Banagl, \emph{Topological invariants of stratified spaces}, Springer
  Monographs in Mathematics, Springer, Berlin, 2007. \MR{2286904}

\bibitem[Bea08]{Borel_1984}
A.~Borel and et~al., \emph{Intersection cohomology}, Modern Birkh\"auser
  Classics, Birkh\"auser Boston, Inc., Boston, MA, 2008, Notes on the seminar
  held at the University of Bern, Bern, 1983, Reprint of the 1984 edition.
  \MR{2401086}

\bibitem[Bre97]{Bredon_1997}
Glen~E. Bredon, \emph{Sheaf theory}, second ed., Graduate Texts in Mathematics,
  vol. 170, Springer-Verlag, New York, 1997. \MR{1481706}

\bibitem[CSAT18]{MR3807751}
David Chataur, Martintxo Saralegi-Aranguren, and Daniel Tanr\'e, \emph{Blown-up
  intersection cohomology}, An alpine bouquet of algebraic topology, Contemp.
  Math., vol. 708, Amer. Math. Soc., [Providence], RI, 2018, pp.~45--102.
  \MR{3807751}

\bibitem[CSAT20]{Chataur_2019}
\bysame, \emph{Blown-up intersection cochains and {D}eligne's sheaves}, Geom.
  Dedicata \textbf{204} (2020), 315--337. \MR{4056706}

\bibitem[FM13]{MR3046315}
Greg Friedman and James~E. McClure, \emph{Cup and cap products in intersection
  (co)homology}, Adv. Math. \textbf{240} (2013), 383--426. \MR{3046315}

\bibitem[Fri07]{MR2276609}
Greg Friedman, \emph{Singular chain intersection homology for traditional and
  super-perversities}, Trans. Amer. Math. Soc. \textbf{359} (2007), no.~5,
  1977--2019. \MR{2276609}

\bibitem[Fri09]{MR2507117}
\bysame, \emph{Intersection homology and {P}oincar\'e{} duality on
  homotopically stratified spaces}, Geom. Topol. \textbf{13} (2009), no.~4,
  2163--2204. \MR{2507117}

\bibitem[Fri10]{Friedman_2010}
\bysame, \emph{Intersection homology with general perversities}, Geom. Dedicata
  \textbf{148} (2010), 103--135. \MR{2721621}

\bibitem[Fri20]{Friedman_2020}
\bysame, \emph{Singular intersection homology}, New Mathematical Monographs,
  vol.~33, Cambridge University Press, Cambridge, 2020. \MR{4406774}

\bibitem[GM80]{GORESKY1980135}
Mark Goresky and Robert MacPherson, \emph{Intersection homology theory},
  Topology \textbf{19} (1980), no.~2, 135--162. \MR{572580}

\bibitem[GM83]{MR696691}
\bysame, \emph{Intersection homology. {II}}, Invent. Math. \textbf{72} (1983),
  no.~1, 77--129. \MR{696691}

\bibitem[GS03]{MR1975331}
Mark Gross and Bernd Siebert, \emph{Affine manifolds, log structures, and
  mirror symmetry}, Turkish J. Math. \textbf{27} (2003), no.~1, 33--60.
  \MR{1975331}

\bibitem[GS11]{MR2893676}
\bysame, \emph{An invitation to toric degenerations}, Surveys in differential
  geometry. {V}olume {XVI}. {G}eometry of special holonomy and related topics,
  Surv. Differ. Geom., vol.~16, Int. Press, Somerville, MA, 2011, pp.~43--78.
  \MR{2893676}

\bibitem[GS23]{Gross_2023}
Andreas Gross and Farbod Shokrieh, \emph{A sheaf-theoretic approach to tropical
  homology}, J. Algebra \textbf{635} (2023), 577--641. \MR{4637248}

\bibitem[Hat02]{MR1867354}
Allen Hatcher, \emph{Algebraic topology}, Cambridge University Press,
  Cambridge, 2002. \MR{1867354}

\bibitem[IKMZ19]{MR3961331}
Ilia Itenberg, Ludmil Katzarkov, Grigory Mikhalkin, and Ilia Zharkov,
  \emph{Tropical homology}, Math. Ann. \textbf{374} (2019), no.~1-2, 963--1006.
  \MR{3961331}

\bibitem[JRS18]{MR3894860}
Philipp Jell, Johannes Rau, and Kristin Shaw, \emph{Lefschetz {$(1,1)$}-theorem
  in tropical geometry}, \'Epijournal G\'eom. Alg\'ebrique \textbf{2} (2018),
  Art. 11, 27. \MR{3894860}

\bibitem[JSS19]{MR3903579}
Philipp Jell, Kristin Shaw, and Jascha Smacka, \emph{Superforms, tropical
  cohomology, and {P}oincar\'e{} duality}, Adv. Geom. \textbf{19} (2019),
  no.~1, 101--130. \MR{3903579}

\bibitem[Kin85]{KING1985149}
Henry~C. King, \emph{Topological invariance of intersection homology without
  sheaves}, Topology Appl. \textbf{20} (1985), no.~2, 149--160. \MR{800845}

\bibitem[KS90]{MR1074006}
Masaki Kashiwara and Pierre Schapira, \emph{Sheaves on manifolds}, Grundlehren
  der mathematischen Wissenschaften [Fundamental Principles of Mathematical
  Sciences], vol. 292, Springer-Verlag, Berlin, 1990, With a chapter in French
  by Christian Houzel. \MR{1074006}

\bibitem[KS06]{MR2182076}
\bysame, \emph{Categories and sheaves}, Grundlehren der mathematischen
  Wissenschaften [Fundamental Principles of Mathematical Sciences], vol. 332,
  Springer-Verlag, Berlin, 2006. \MR{2182076}

\bibitem[Max19]{Maxim_2019}
Lauren\c tiu~G. Maxim, \emph{Intersection homology \& perverse sheaves---with
  applications to singularities}, Graduate Texts in Mathematics, vol. 281,
  Springer, Cham, 2019. \MR{4269941}

\bibitem[Mik25]{mikami2025}
Ryota Mikami, \emph{Tropical intersection homology}, arXiv:2412.20748, 2025.

\bibitem[MZ14]{MR3330789}
Grigory Mikhalkin and Ilia Zharkov, \emph{Tropical eigenwave and intermediate
  {J}acobians}, Homological mirror symmetry and tropical geometry, Lect. Notes
  Unione Mat. Ital., vol.~15, Springer, Cham, 2014, pp.~309--349. \MR{3330789}

\bibitem[Rot17]{MR3677125}
Joseph~J. Rotman, \emph{Advanced modern algebra. {P}art 2.}, third ed.,
  Graduate Studies in Mathematics, vol. 180, American Mathematical Society,
  Providence, RI, 2017, With a foreword by Bruce Reznick. \MR{3677125}

\bibitem[Rud21]{MR4347312}
Helge Ruddat, \emph{A homology theory for tropical cycles on integral affine
  manifolds and a perfect pairing}, Geom. Topol. \textbf{25} (2021), no.~6,
  3079--3132. \MR{4347312}

\bibitem[Sha13]{MR3032930}
Kristin~M. Shaw, \emph{A tropical intersection product in matroidal fans}, SIAM
  J. Discrete Math. \textbf{27} (2013), no.~1, 459--491. \MR{3032930}

\bibitem[Spa88]{MR932640}
N.~Spaltenstein, \emph{Resolutions of unbounded complexes}, Compositio Math.
  \textbf{65} (1988), no.~2, 121--154. \MR{932640}

\bibitem[Wei94]{Weibel_1994}
Charles~A. Weibel, \emph{An introduction to homological algebra}, Cambridge
  Studies in Advanced Mathematics, vol.~38, Cambridge University Press,
  Cambridge, 1994. \MR{1269324}

\bibitem[Yam24]{yamamoto2024}
Yuto Yamamoto, \emph{Tropical contractions to integral affine manifolds with
  singularities}, arXiv:2105.10141, 2024.

\end{thebibliography}

\end{document}